\documentclass[a4paper, reqno]{amsart}
\usepackage{amsmath, amsthm, amssymb, amstext}

\usepackage[left=2.3cm,right=2.3cm,top=2.3cm,bottom=2.3cm]{geometry}

\usepackage{hyperref,xcolor}
\usepackage[hyperpageref]{backref}
\hypersetup{pdfborder={0 0 0},colorlinks}

\usepackage{enumitem}
\allowdisplaybreaks
\newtheorem{theorem}{Theorem}

\newtheorem{lemma}[theorem]{Lemma}
\newtheorem{proposition}[theorem]{Proposition}
\newtheorem{corollary}[theorem]{Corollary}
\newtheorem{definition}[theorem]{Definition}
\newtheorem{example}[theorem]{Example}

\numberwithin{theorem}{section}
\numberwithin{equation}{section}

\title[On sharp singular Moser-Trudinger type inequalities]{On sharp singular Moser-Trudinger type inequalities and applications to zero mass $(p,Q)$-Laplace subelliptic equations in the Heisenberg group}
\author[G.M. Figueiredo]{Giovany M. Figueiredo}
\address[G.M. Figueiredo]{Universidade de Brasília, Departamento de Matemática, 70910-900 Brasília DF, Brazil}
\email{giovany@unb.br}
\author[D.K. Mahanta]{Deepak Kumar Mahanta}
\address[D.K. Mahanta]{Department of Mathematics, Indian Institute of Technology Jodhpur, Rajasthan 342030, India}
\email{mahanta.1@iitj.ac.in}
\author[A. Sarkar]{Abhishek Sarkar}
\address[A. Sarkar]{Department of Mathematics, Indian Institute of Technology Jodhpur, Rajasthan 342030, India}
\email{abhisheks@iitj.ac.in}

\begin{document}
\begin{abstract}
This article establishes a sharp singular Moser-Trudinger type inequality in a new function space and explores a concentration-compactness principle in the Heisenberg group, inspired by P.\textcolor{blue}{-}L. Lions. We analyze the equivalence of sharp critical and subcritical singular Moser-Trudinger inequalities, focusing on their asymptotic behavior and connections between their suprema. Additionally, we apply the mountain pass theorem to demonstrate the existence of positive ground state solutions for zero mass $(p, Q)$-Laplace subelliptic equations with singular exponential nonlinearity for the following equation:
$$
    -\Delta_{H,p} u-\Delta_{H,Q} u=\frac{f(\xi,u)}{r(\xi)^\vartheta}\quad\text{in}\quad \mathbb{H}^n,
$$
with $1<p<Q$, $\vartheta\in(0,Q)$, $Q=2n+2$, the function $r(\cdot)$ is called the Kor\'anyi norm in $\mathbb{H}^n$ and the nonlinearity $f:\mathbb{H}^n\times \mathbb{R}\to \mathbb{R}$ is a Carath\'eodory function, which behaves like $\exp{(\alpha|s|^{\frac{Q}{Q-1}})}$ as $|s|\to~+\infty$ for some $\alpha>0$.
\end{abstract}
\maketitle
\section{Introduction and main results}\label{sec1}
In recent decades, there has been a significant interest in the study of geometric and functional inequalities in various function spaces due to a wide range of applications in harmonic analysis, geometric analysis, CR geometry, PDEs, modern physics, and other fields of mathematics. For instance, we refer the readers to study some of the exciting literature on Gagliardo–Nirenberg-Sobolev inequalities \cite{Demange-2008}, Caffarelli–Kohn–Nirenberg inequalities \cite{Caffarelli-Kohn-Nirenberg-1984}, Hardy-Littlewood-Sobolev inequalities \cite{Garofalo-Tralli-2022}, Hardy-Rellich inequalities \cite{Ghoussoub-Moradifam-2011}, Hardy-Sobolev inequalities \cite{Filippas-Maz’ya-Tertikas-2007}, Poincar\'{e} inequalities \cite{Boulkhemair-Chakib-2007, Keith-Zhong-2008}, Moser-Trudinger inequalities \cite{Beckner-1993, Cohn-Lu-2004} and Adams inequalities \cite{Adams-1988, Fontana-Morpurgo-2011}, etc. To the best of our knowledge, several authors concentrate on the stability, existence, and non-existence of extremal and attainability of best constants of such types of inequalities. In this context, we cite \cite{Catrina-Wang-2001, Chen-Lu-Tang-2024, Chern-Lin-2010, Hudson-Leckband-2004, Lu-Yang-2009, Lu-Zhu-2011, Tertikas-Zographopoulos-2007} and the references therein.

The main focus of this article is to discuss one such inequality, namely the celebrated Moser-Trudinger inequality. Generally, the existence of such types of inequalities arises in the literature due to the failure of the Sobolev embedding $W^{1,n}(\mathbb{R}^n)\hookrightarrow L^\infty(\mathbb{R}^n)$. For example, if we take $$u(x):=\psi(x)\log\bigg(\log \bigg(1+\frac{1}{|x|}\bigg)\bigg),~~\forall~\psi\in C^\infty_0(\mathbb{R}^n,[0,1])\quad\text{and}\quad n\geq 2,$$ where $\psi\equiv 1$ in $B_1(0)$ and $\psi\equiv 0$ in $ \mathbb{R}^n\setminus B_{2}(0)$, then one can easily verify $u\in W^{1,n}(\mathbb{R}^n)$ but $u\notin L^\infty(\mathbb{R}^n)$. Although, the standard Sobolev embeddings do not work, Trudinger \cite{Trudinger-1967} proved that for any bounded domain $\Omega\subset\mathbb{R}^n$ with $n\geq 2$, $W^{1,n}(\Omega)$ is continuously immersed into the Orlicz space $L_{\Phi_\alpha}(\Omega)$, where $\Phi_\alpha$ is a Young function given by $\Phi_\alpha(s)=\operatorname{\exp}\big(\alpha|s|^{\frac{n}{n-1}}\big)-1$ for some $\alpha>0$. In this direction, we would also like to cite the works of Yudovich \cite{Yudovich-1961} and Pokhozhaev \cite{Pokhozhaev-1965}. The sharpness of
the exponent $\alpha$ was first studied by Moser \cite{Moser-1970/71}. He used classical Schwarz symmetrization techniques to establish the following inequality:
\begin{align}\label{eq10.1}
    \sup_{u\in W^{1,n}_0(\Omega),~\|\nabla u\|_{L^n(\Omega)}~\leq 1}~\int_{\Omega} \operatorname{\exp}(\alpha|u|^{\frac{n}{n-1}})\,\mathrm{d}x\leq C(n)|\Omega|,~\forall~\alpha\leq \alpha_n, 
 \end{align}
 where $\alpha_n=n\omega_{n-1}^{\frac{1}{n-1}}$ and $\omega_{n-1}$ is the surface measure of the unit sphere in $\mathbb{R}^n$. Moreover, the above supremum is no longer finite if $\alpha> \alpha_n$. Inequality of this kind is currently referred to as the Moser-Trudinger type inequality. Note that the supremum \eqref{eq10.1} becomes meaningless for domains with infinite measure even if $\alpha\leq \alpha_n$. This makes the study more interesting, and many mathematicians dived deep into the study of the Moser-Trudinger inequality in the whole $\mathbb{R}^n$. For instance, Cao \cite{Cao-1992} first generalized the result obtained by Moser \cite{Moser-1970/71} and established a version of the Moser-Trudinger inequality in $\mathbb{R}^2$ as follows: 
 \begin{align}\label{eq10.2}
     \int_{\mathbb{R}^2}\big(\operatorname{\exp}(4\pi u^2)-1\big)\,\mathrm{d}x\leq C(m,M),~~\forall~u\in H^1(\mathbb{R}^2),
 \end{align}
  where $\|\nabla u\|^2_2\leq m<1 $, $\|u\|_2\leq M<+\infty$ and $C(m,M)$ is a suitable constant depending only on $m$ and $M$. Later, the inequality \eqref{eq10.2} was fully extended into higher dimensions (in $\mathbb{R}^n$) independently by do \'O \cite{do O-1997}.
 
 In addition, for any arbitrary domain $\Omega\subset\mathbb{R}^2$, it has been observed by Ruf \cite{Ruf-2005} that the following Moser-Trudinger inequality holds:
 \begin{align}\label{eq10.3}
    \sup_{u\in H^1_0(\Omega),~\int_\Omega(|\nabla u|^2+|u|^2)\,\mathrm{d}x\leq1} \int_{\Omega}\big(\operatorname{\exp}(4\pi u^2)-1\big)\,\mathrm{d}x\leq C,
 \end{align}
where $C>0$ is a suitable constant. It is important to note that the inequality mentioned above is sharp; that is, for any growth $\exp(\alpha u^2)$ with $\alpha>4\pi$, the supremum \eqref{eq10.3} becomes infinity. Moreover, by employing blow-up analysis techniques, Li-Ruf \cite{Li-Ruf-2008} again extended the previous result into an arbitrary domain $\Omega\subset\mathbb{R}^n$ and thus\textcolor{blue}{,} into $\mathbb{R}^n$.

To the best of the author's knowledge, the singular version of the Moser-Trudinger inequality \eqref{eq10.1} in bounded domains was first studied by Adimurthi-Sandeep in \cite{Adimurthi-Sandeep-2007}. It states that for all $\beta\in[0,n)$ and $\alpha\leq \alpha_{\beta,n}:=\big(1-\frac{\beta}{n}\big)\alpha_n$, there holds:
\begin{align}\label{eq10.4}
    \sup_{u\in W^{1,n}_0(\Omega),~\|\nabla u\|_{L^n(\Omega)\leq1}}\int_\Omega\frac{\operatorname{\exp}(\alpha|u|^{\frac{n}{n-1}})}{|x|^\beta}\,\mathrm{d}x<+\infty.
\end{align}
Consequently, the above supremum is sharp; that is, if we replace $\alpha_{\beta,n}$ by a larger constant, then the supremum \eqref{eq10.4} is no longer finite. Moreover, the above result has been fully generalized into the whole $\mathbb{R}^n$ by Adimurthi-Yang \cite{Adimurthi-Yang-2010} and de Souza \cite{de Souza-2011}, respectively.

It is well known that the embedding $W^{1,n}_0(\Omega)\hookrightarrow L_{\Phi_\alpha}(\Omega)$ is not compact, and thus, to deal with variational problems, we have several barriers to show the Palais-Smale compactness condition. To overcome these difficulties, Lions \cite{Lions-1985} established a concentration-compactness principle, which is a generalization of the Moser-Trudinger inequality. It states that if $u_k\rightharpoonup u$ weakly in $W^{1,n}_0(\Omega)$ with $u\neq 0$, $\|\nabla u_k\|_{L^n(\Omega)}\leq 1$, $|\nabla u_k|^n \rightharpoonup \mu$ weakly in $\mathcal{M}(\overline{\Omega})$, where $\mathcal{M}(\overline{\Omega})$ is the space of Radon measure in $\mathbb{R}^n$, then 
 \begin{align}\label{eq10.5}
   \sup_{k\in\mathbb{N}} \int_{\Omega} \operatorname{\exp}(\alpha_n p|u_k|^{\frac{n}{n-1}})\,\mathrm{d}x\leq C(p,\Omega), ~\forall~ p\in [1,\eta),   
 \end{align}
 where $\eta=(1-\|\nabla u^\ast \|^n_{L^n(\Omega)})^{-\frac{1}{n-1}}$ and $u^\ast$ is the spherical symmetric decreasing rearrangement of $u$. Further, it has been proved by \v{C}ern\'y-Cianchi-Hencl \cite{Cerny-Cianchi-Hencl-2013} that the exponent $p$ that appeared in \eqref{eq10.5} is sharp and established an improved version of Lions concentration-compactness principle in $W^{1,n}_0(\Omega)$. Notice that analogous versions of the concentration-compactness principle in the spirit of Lions in $W^{1,n}(\mathbb{R}^n)$ can also be seen in the works of \cite{do O-de Souza-de Medeiros-Severo-2014, Zhang-Chen-2018} and the references therein.

It should be pointed out that Moser-Trudinger inequalities are distinguished in two different ways, for instance, critical and subcritical Moser-Trudinger inequalities. The subcritical inequalities are subject to the partial Sobolev norms for the functions, whereas the critical sharp inequalities are subject to the complete Sobolev norms for the functions. These two categories of inequality differ somewhat from one another. However, these inequalities are surprisingly equivalent in the context of standard Sobolev spaces as well as in some other function spaces. For more details, we refer to \cite{de Oliveira-do O-2023, Lam-Lu-Zhang-2017, Lam-Lu-Zhang-2019, Lam-Lu-Zhang-2019@} and the related references. 

In addition, it is important to emphasize that there have been substantial advances in the Moser-Trudinger inequalities to study the existence and multiplicity of solutions to quasilinear elliptic PDEs involving the $n$-Laplace operator as well as the $(p,n)$-Laplace operator in $\mathbb{R}^n$. For a detailed study, one may see \cite{Adimurthi-Yang-2010, Carvalho-Figueiredo-Furtado-Medeiros-2021, Lam-Lu-2012, Mahanta-Mukherjee-Sarkar-2025, Mahanta-Winkert-2026, Yang-2012@} and the references therein. Besides the above-quoted works, some improved and complementary results on the Moser-Trudinger inequalities, which serve as good starting points for understanding the topic, can be found in \cite{Adimurthi-Druet-2005, Csato-Roy-2016, Ibrahim-Masmoudi-Nakanishi-2015, Ishiwata-2011, Malchiodi-Martinazzi-2014, Masmoudi-Sani-2015, Struwe-1988}.

Next, we shall provide a brief historical background on the Moser-Trudinger inequalities in the Heisenberg group. To present our main results, we need to define some basic terms related to the Heisenberg group. Please see Section \ref{sec2} for more details on this topic. In the Euclidean context, it is well known that proofs of the Moser-Trudinger inequalities and the associated concentration-compactness principles are based on conventional Schwarz symmetrization techniques and the P\'olya–Szeg\"{o} inequality. Due to the failure of the the P\'olya–Szeg\"{o} inequality in the Heisenberg group setting, for instance, see \cite[Theorem 5.3]{Manfredi-Vera de Serio-2019} and \cite{Jerison-Lee-1988}, one cannot directly employ similar strategies developed in the list of papers (see \cite{Adimurthi-Yang-2010,  Cao-1992,  do O-1997, do O-de Souza-de Medeiros-Severo-2014}) to prove the Moser-Trudinger inequalities as well as the concentration-compactness principle in the Heisenberg group. In addition, for the aforementioned reason, we also highlight here that the study of extremals of such inequalities in the Heisenberg group context remains an open problem today.

Initially, Cohn-Lu \cite{Cohn-Lu-2001} studied a version of the Moser-Trudinger inequality on bounded domains in the Heisenberg group. They established that for any bounded domain $\Omega\subset\mathbb{H}^n$ and $0<\alpha\leq \alpha_Q$, where $\alpha_Q$ is defined in \eqref{eq9.65}, there holds:
\begin{align}\label{eq10.6}
    \sup_{u\in HW^{1,Q}_0(\Omega),~\|D_H u\|_{L^Q(\Omega)}~\leq 1}~\frac{1}{|\Omega|}\int_{\Omega} \operatorname{\exp}\big(\alpha|u|^{\frac{Q}{Q-1}}\big)\,\mathrm{d}\xi\leq C_0,
 \end{align}
 where the uniform constant $C_0>0$ depends only on $Q$. Consequently, the constant $\alpha_Q$ is sharp; that is, if $\alpha>\alpha_Q$, then the above integral is finite, but the supremum \eqref{eq10.6} is infinity. Again, the above inequality was generalized into a singular version of the Moser-Trudinger inequality by Lam-Lu-Tang \cite{Lam-Lu-Tang-2012}. It can be stated in the following manner: for any bounded domain $\Omega\subset\mathbb{H}^n$ and $\beta\in[0,Q)$ and $\alpha_{\beta,Q}:=\big(1-\frac{\beta}{Q}\big)\alpha_Q$, then there exists a constant $C>0$ depending only on $\beta$ and $Q$ such that
 \begin{align}\label{eq10.7}
    \sup_{u\in HW^{1,Q}_0(\Omega),~\|D_H u\|_{L^Q(\Omega)\leq1}}\frac{1}{|\Omega|^{1-\frac{\beta}{Q}}}\int_\Omega\frac{\operatorname{\exp}\big(\alpha_{\beta,Q}|u|^{\frac{Q}{Q-1}}\big)}{r(\xi)^\beta}\,\mathrm{d}\xi\leq C.
\end{align}
In addition, the constant $\alpha_{\beta,Q}$ is sharp in the sense that, if we replace  $\alpha>\alpha_{\beta,Q}$ in place of $\alpha_{\beta,Q}$, the above supremum \eqref{eq10.7} is no longer finite.

Later, it has been observed that Lam-Lu \cite{Lam-Lu-201@} extended the result obtained by Adimurthi-Yang \cite{Adimurthi-Yang-2010} in the Euclidean space $\mathbb{R}^n$ into the whole $\mathbb{H}^n$, and proved a singular Moser-Trudinger inequality by using symmetrization-free arguments. They showed that for any $\tau>0$, $\beta\in[0,Q)$ and $0<\alpha\leq\alpha_{\beta,Q}:=\big(1-\frac{\beta}{Q}\big)\alpha_Q$, there holds:
  \begin{align}\label{eq10.8}
 \sup_{u\in HW^{1,Q}(\mathbb{H}^n),~\|u\|_{1,\tau}\leq 1} ~\int_{\mathbb{H}^n}\cfrac{\Phi\big(\alpha|u|^{\frac{Q}{Q-1}}\big)}{r(\xi)^\beta}\,\mathrm{d}\xi<+\infty,   
 \end{align}
 where
 \begin{align}\label{eq10.9}
  \quad\|u\|_{1,\tau}=(\tau\|u\|^Q_Q+\|D_H u\|^Q_Q)^{\frac{1}{Q}}\quad\text{and}\quad \Phi(s)=\exp(s)-\sum_{j=0}^{Q-2}\frac{s^j}{j!}.    
 \end{align}
Additionally, if $\alpha>\alpha_{\beta,Q}$, then the integral in \eqref{eq10.8} is still finite, but the above supremum is infinite. 

In this context, we also mention here that the subcritical singular Moser-Trudinger inequality can be found in \cite{Lam-Lu-Tang-2014}, a weighted Moser-Trudinger inequality was established in \cite{Lam-Tang-2013}, whereas by using a degenerate potential function, the authors in \cite{Chen-Lu-Zhu-2023} proved a version of such a type of inequality. However, by introducing symmetrization arguments in the Heisenberg group and a weaker form of the P\'olya–Szeg\"{o} types inequality that appeared in \cite[Theorem 5.3]{Manfredi-Vera de Serio-2019}, Cohn-Lam-Lu-Yang \cite{Cohn-Lam-Lu-Yang-2012} demonstrated an alternative form of the singular Moser-Trudinger inequality. It states that for any $0<\alpha\leq\alpha^\ast$ and $0\leq \beta<Q$, there holds:
 \begin{align*}
 \sup_{\|u\|_{HW^{1,Q}}\leq 1} ~\int_{\mathbb{H}^N}\cfrac{\Phi\big(\alpha|u|^{\frac{Q}{Q-1}}\big)}{r(\xi)^\beta}\,\mathrm{d}\xi<+\infty\iff \frac{\alpha}{\alpha^\ast}+\frac{\beta}{Q}\leq 1,    
 \end{align*}
 where $\alpha^{\ast}=\frac{\alpha_Q}{c^\ast}$ and for $u^\ast$ is the decreasing rearrangement of the nonnegative function $u$, $c^\ast$ is given by $$c^\ast=\inf\bigg\{ c^{\frac{1}{Q-1}}:~\displaystyle\int_{\mathbb{H}^n} |D_H u^\ast|^Q\,\mathrm{d}\xi\leq c \displaystyle\int_{\mathbb{H}^n} |D_H u|^Q\,\mathrm{d}\xi,~u\in HW^{1,Q}(\mathbb{H}^n)\bigg\}.$$

By the analogue of the Lions concentration\textcolor{blue}{-}compactness principle \cite{Lions-1985} and introducing a completely symmetrization-free method, Li-Lu-Zhu \cite{Li-Lu-Zhu-2018} provided the concentration\textcolor{blue}{-}compactness results for both the bounded domains and the whole space $\mathbb{H}^n$, respectively. Their result for any bounded domain $\Omega\subset\mathbb{H}^n$ can be stated as follows:
for any $\beta\in[0,Q)$ and a sequence $\{u_k\}_k\subset HW^{1,Q}_0(\Omega)$ with $\|D_H u_k\|_{L^Q(\Omega)}=1$ and $u_k\rightharpoonup u\neq 0$ weakly in $HW^{1,Q}_0(\Omega)$ as $k\to\infty$, then
$$\sup_{k\in\mathbb{N}}\int_\Omega\frac{\operatorname{\exp}\big(\alpha_{\beta,Q}p|u_k|^{\frac{Q}{Q-1}}\big)}{r(\xi)^\beta}\,\mathrm{d}\xi<+\infty,$$
where $$0<p<M_{Q,u}:=\big(1-\|D_H u\|^Q_{L^Q(\Omega)}\big)^{-\frac{1}{Q-1}}. $$
The constant $M_{Q,u}$ is sharp; that is, the above supremum is infinite for any $p\geq M_{Q,u}$. Similarly, for the whole space $\mathbb{H}^n$, it can be read as follows: for any $\beta\in[0,Q)$ and a sequence $\{u_k\}_k\subset HW^{1,Q}(\mathbb{H}^n)$ with $\| u_k\|_{HW^{1,Q}}=1$ and $u_k\rightharpoonup u\neq 0$ weakly in $HW^{1,Q}(\mathbb{H}^n)$ as $k\to\infty$, then
$$\sup_{k\in\mathbb{N}}\int_{\mathbb{H}^n}\cfrac{\Phi\big(\alpha_{\beta,Q}p|u_k|^{\frac{Q}{Q-1}}\big)}{r(\xi)^\beta}\,\mathrm{d}\xi<+\infty,$$
where $$0<p<\widetilde{M}_{Q,u}:=\big(1-\| u_k\|_{HW^{1,Q}}^Q\big)^{-\frac{1}{Q-1}} $$
and $\Phi(\cdot)$ is defined in \eqref{eq10.9}. In addition, the constant $\widetilde{M}_{Q,u}$ is sharp in the sense that the supremum defined above is no longer finite for any $p> \widetilde{M}_{Q,u}$. Moreover, we also refer to \cite{Chen-Lu-Zhu-2021,  Li-Lu-Zhu-2021} for the complete treatment in this direction.

Over the past few decades, the study of nonhomogeneous quasilinear elliptic equations involving the $(p,q)$-Laplace operator has gained considerable attention due to its wide range of applications in the applied sciences. Indeed, from a physical point, if $u$ denotes a concentration variable, $f(x,u)$ describes a local reaction term related to the source and loss process, and $\text{div}(\mathcal{D}u~ \nabla u)$ represents the diffusion with a diffusion coefficient $\mathcal{D}u$, where $ \mathcal{D}u:=|\nabla u|^{p-2} \nabla u+|\nabla u|^{q-2} \nabla u$, then the study of reaction-diffusion equations, namely
$$ u_t=\text{div}(\mathcal{D}u~ \nabla u)+f(x,u)$$ can be frequently observed in physics, plasma physics, biophysics, chemical reaction design, electromagnetism, and other fields. In many situations, $f(x,u)$ has a polynomial type with variable coefficients, but for the Liouville-Bratu-Gelfand and Frank-Kamenetsky models, $f(x,u)$ exhibits exponential growth at infinity. In this context, we refer \cite{Cassani-Tavares-Zhang-2020} and the references therein. 

The following form of the $(p,q)$-Laplace elliptic equations has been explored by several mathematicians:
$$ -\Delta_pu-\Delta_qu+V_1(x)|u|^{p-2}u+V_2(x)|u|^{q-2}u=f(x,u)\quad\text{in}\quad\mathbb{R}^n ,$$
where $p,q\in(1,+\infty)$ with $p<q$ and $V_1$ and $V_2$ are two potential functions satisfying some suitable hypotheses. The existence, non-existence, and regularity of weak solutions to the above problems have been studied by several authors (see \cite{Ambrosio-2023, Cai-Zhang-2024, He-Li-2008, Mahanta-Mukherjee-Winkert-2025, Zouai-Benouhiba-2023}) using both variational and non-variational techniques. If we take $V_1(x)=V_2(x)=0$ in $\mathbb{R}^n$, then this class of equations are so-called zero mass $(p,q)$-Laplace equations. In recent times, many authors have focused on the study of zero-mass quasilinear elliptic equations in the Euclidean context; for instance, see \cite{Carvalho-Figueiredo-Furtado-Medeiros-2021, Shen-Radulescu-2024} and the related references. It is noteworthy to mention that there are very limited amounts of literature available on the $Q$-Laplace and the $(p,Q)$-Laplace subelliptic equations in the Heisenberg group setting. In this direction, we refer to \cite{Balogh-Manfredi-Tyson-2003, Lam-Lu-201@, Lam-Lu-Tang-2012, Li-Lu-Zhu-2018, Pucci-Temperini-2022} and the references therein.  However, to the best of the authors' knowledge, zero-mass subelliptic equations have not yet been studied in the context of the Heisenberg group. This is one of the key motivations for studying this article. 

To study such a class of subelliptic equations, we assume $1<p<Q$ and define the function space $W$ by
$$ W:=\{u\in L^1_{\textit{loc}}(\mathbb{H}^n)\colon~|D_Hu|_H\in L^p(\mathbb{H}^n)\cap L^Q(\mathbb{H}^n)\}, $$
 which is the completion of the $C_0^\infty(\mathbb{H}^n)$ with respect to the norm 
$$ \|u\|:=\Big(\|D_Hu\|^Q_p+\|D_Hu\|^Q_Q\Big)^{\frac{1}{Q}},~\forall~u\in W, $$
where $\|\cdot\|_m$ is the standard $L^m(\mathbb{H}^n)$ norm for any $m\in[1,+\infty)$. In addition, for any $\Omega\subset \mathbb{H}^n$, we denote the norm of $L^m(\Omega)$ by $\|\cdot\|_{L^m(\Omega)}$. Also, if we introduce another norm $\|\cdot\|$ defined by $\|\cdot\|_W:=\|D_H \cdot\|_p+\|D_H \cdot\|_Q$, then $\|\cdot\|$ and $\|\cdot\|_W$ are equivalent norms for the space $W$. Indeed, one has
\begin{align}\label{eq1.99}
\|u\|\leq \|u\|_W\leq 2^{\frac{Q-1}{Q}}\|u\|,~\forall~u\in W. 
\end{align}
It is easy to see that $(W,\|\cdot\|_W)$ is a uniform convex, reflexive, and separable Banach space. Further, since reflexivity is preserved under equivalent norms, we deduce that $(W,\|\cdot\|)$ is also a reflexive Banach space. In addition, the Folland-Stein space $S^{1,p}(\mathbb{H}^n)$ is defined as the completion of $C_0^\infty(\mathbb{H}^n)$ equipped with the norm 
$$ \|u\|_{S^{1,p}}:=\|D_H u\|_p,~\forall~u\in S^{1,p}(\mathbb{H}^n). $$
In virtue of the Sobolev embedding, the chain of embeddings $W\hookrightarrow S^{1,p}(\mathbb{H}^n)\hookrightarrow L^{p^\ast}(\mathbb{H}^n)$ holds. Moreover, for any $p\in[1,+\infty)$ and $\vartheta\in [0,Q)$, the space $L^p(\mathbb{H}^n,r(\xi)^{-\vartheta} \mathrm{d}\xi)$ is defined by
$$ L^p(\mathbb{H}^n,r(\xi)^{-\vartheta} \mathrm{d}\xi):=\{u:\mathbb{H}^n\to\mathbb{R}~\text{is measurable}~:~|u|^p r(\xi)^{-\vartheta}\in L^1(\mathbb{H}^n)\},$$ 
endowed with the norm
$$ \|u\|_{p,\vartheta}:=\bigg(\int_{\mathbb{H}^n} |u|^p r(\xi)^{-\vartheta}\,\mathrm{d}\xi\bigg)^{\frac{1}{p}}.$$

Additionally, the following notations are used throughout this article: 
\begin{itemize}
\item $(W^*,\|\cdot\|_{*})$ stands for the continuous dual of $(W,\|\cdot\|)$, 
\item $\langle \cdot, \cdot \rangle$ stands for the duality order pair between $W^\ast$ and $W$,
\item  $u^+=$ max~$\{u,0\}$ and $u^-=$ max~$\{- u,0\}$,
\item $o_k(1)$ stands for a real sequence that converges to zero as $k\to\infty$, 
\item $\rightharpoonup $ and $\rightarrow$ stand for the weak convergence and the strong convergence,
\item $B_r(\xi)$ stands for an open ball centered at $\xi\in \mathbb{H}^n$ with radius $r>0$ and $B_r=B_r(O)$,
\item $g\lesssim ~( \gtrsim)~ h$, which means $g\leq ( \geq)~C h$ for some suitable constant $C>0$,
\item $g\underset{\approx}{<}h$ means $g$ is close enough to $h$, and $O(g)$ denotes some constant multiple of $g$,
\item $g\thickapprox h$ means $g$ is approximately equal to function $h$, 
\item $|U|$ represents the $(2n+1)$-dimensional Lebesgue measure for $U\subset\mathbb{H}^n$. \end{itemize}

We would like to highlight the key aspects and novelties of this study, which are listed below\textcolor{blue}{:}
\begin{itemize}
    \item [(a)] In the Heisenberg group setting, due to the absence of the P\'olya–Szeg\"{o} type inequality, we cannot use Schwarz symmetrization to prove the singular Moser-Trudinger type inequality in $W$ (see Theorem \ref{thm1.1}) as well as the associated concentration-compactness principle (see Theorem \ref{thm1.2}). We use a rearrangement-free approach introduced by Lam-Lu \cite{Lam-Lu-201@})  to prove such inequalities.
    \item [(b)] The embedding $W$ into the space $L_{\Phi_{\alpha,j_0}}(\mathbb{H}^n)$ is not compact when there is no singularity, meaning $\vartheta=0$. This is explained in Theorem \ref{thm1.1}. The space $L_{\Phi_{\alpha,j_0}}(\mathbb{H}^n)$ is defined by the Young function $\Phi_{\alpha,j_0}(\cdot)$ shown in equation \eqref{eq1.55}. Therefore, it is quite challenging to confirm the Palais-Smale compactness condition for the associated energy functional related to our main problem when we are considering the non-singular case. It should be noted here that a similar type of difficulty also arises for the singular case, that is, $\vartheta\neq 0$. To surpass this situation, we shall establish a concentration-compactness principle (see Theorem \ref{thm1.2}) by using a rearrangement-free approach discussed in \cite{Chen-Lu-Zhu-2021, Li-Lu-Zhu-2018, Li-Lu-Zhu-2021}.
    \item [(c)] In our study, we rely on the compact embedding $ W \hookrightarrow L^\theta(\mathbb{H}^n,r(\xi)^{-\vartheta} \mathrm{d}\xi)$ for all $\theta\geq p^\ast$ and $\vartheta\in(0,Q)$, for instance, see Theorem \ref{thm1.4}. Thus, our approach is different from $\vartheta=0$. However, one can apply the same procedure developed in this article and Lemma \ref{lem2.1} to study the non-singular case, that is, $\vartheta=0$.
    \item [(d)] For every Palais-Smale sequence $\{u_k\}_k \subset W$ with $u_k\rightharpoonup u$ in $W$ as $k\to\infty$, we cannot immediately obtain $$|D_H u_k|_H^{m-2}D_H u_k \rightharpoonup |D_H u|_H^{m-2}D_H u\quad\text{in}\quad L^{\frac{m}{m-1}}(\mathbb{H}^n,\mathbb{R}^{2n})\quad\text{as}\quad k\to\infty$$ for $m\in\{p,Q\}$. It suggests to us that more careful study is needed to prove $$D_H u_k\to D_H u \quad\text{a.e. in}\quad\mathbb{H}^n\quad\text{as}\quad k\to\infty.$$
     \item[(e)] The proofs employ advanced methodologies, such as variational and topological tools.
\end{itemize}

Moreover, inspired by the works quoted above, for example, we refer to \cite{Carvalho-Figueiredo-Furtado-Medeiros-2021, Chen-Lu-Zhu-2023, Lam-Lu-201@, Li-Lu-Zhu-2018, Pucci-Temperini-2022}; for the first time in literature, we establish a version of sharp singular Moser-Trudinger inequality, its concentration-compactness principle, asymptotic behavior of the sharp subcritical singular Moser-Trudinger inequality, relation between sharp critical and subcritical singular Moser-Trudinger inequalities in a suitable space $W$ and as a consequence of these inequalities, we study a class of subelliptic PDEs involving $(p,Q)$-Laplace operator in the whole space $\mathbb{H}^n$.

In view of Lemma \ref{lem2.1}, one sees that the embedding $W\hookrightarrow L^\tau (\mathbb{H}^n)$ is continuous for all  $\tau\in [p^\ast,+\infty)$. Thus, we have the following version of the Moser-Trudinger type inequality, which actually says that $W$ can be continuously immersed into the Orlicz space $L_{\Phi_{\alpha,j_0}}(\mathbb{H}^n)$ for the non-singular case, that is, $\vartheta=0$, where the associated Young function $\Phi_{\alpha,j_0}:\mathbb{R}\to [0,+\infty)$ is given by
\begin{align}\label{eq1.55}
 \Phi_{\alpha,j_0}(s):=\operatorname{exp}(\alpha|s|^{\frac{Q}{Q-1}})-\sum_{j=0}^{j_0-1} \frac{\alpha^j}{j!}|s|^{\frac{Qj}{Q-1}},~~\forall~s\in\mathbb{R}\quad\text{and}\quad\alpha> 0,
 \end{align}
with $$ \quad j_0:=\left\lceil \frac{p^\ast (Q-1)}{Q} \right\rceil=\min\bigg\{j\in\mathbb{N}:~j\geq \frac{p^\ast (Q-1)}{Q}\bigg\},$$
where $\left\lceil \cdot \right\rceil$ is the notation of the ceiling function. In addition, we mentioned that our article also discusses the more generalized singular Moser-Trudinger type inequality, which also includes the non\textcolor{blue}{-}singular case, that is, $\vartheta\in[0,Q)$. It can be read as follows.
\begin{theorem}[\bf Sharp Singular Moser-Trudinger Type Inequality] \label{thm1.1}
Assume that $1<p<Q$ and $\vartheta\in[0,Q)$. Denote 
\begin{align}\label{eq9.65}
\alpha_Q:=Q\bigg(2\pi^n\Gamma\bigg(\frac{1}{2}\bigg)\Gamma\bigg(\frac{Q-1}{2}\bigg){\Gamma\bigg(\frac{Q}{2}\bigg)}^{-1}{\Gamma(n)}^{-1}\bigg)^{Q^\prime-1},~~\text{ where}~~ \frac{1}{Q}+\frac{1}{Q^\prime}=1.
\end{align} 
Then for all $0<\alpha\leq \alpha_{\vartheta,Q}:=\big(1-\frac{\vartheta}{Q}\big)\alpha_Q$ and $u\in W$, there holds
\begin{align}\label{eq1.12}
    \sup_{u\in W,~\|u\|\leq 1}\int_{\mathbb{H}^n}\frac{\Phi_{\alpha,j_0}(u)}{r(\xi)^\vartheta}\,\mathrm{d}\xi<+\infty.
\end{align}
Furthermore, the constant $\alpha_{\vartheta,Q}$ is sharp in the sense that the supremum \eqref{eq1.12} is infinite for $\alpha> \alpha_{\vartheta,Q}$.
\end{theorem}
 Now, by drawing inspiration from the seminal works of Li-Lu-Zhu \cite{Li-Lu-Zhu-2018, Li-Lu-Zhu-2021} and Chen-Lu-Zhu \cite{Chen-Lu-Zhu-2021}, we establish a version of the concentration-compactness principle as an improvement of the Moser-Trudinger inequality stated above as follows.
\begin{theorem}[\bf Sharp Singular Concentration-Compactness Principle] \label{thm1.2}
Assume that $1<p<Q$ and $\vartheta\in[0,Q)$. Let $\{u_k\}_k\subset W$ be a sequence such that $\|u_k\|=1$ and $u_k\rightharpoonup u\neq 0$ in $W$ as $k\to\infty$, then  for all $\zeta$ satisfying $$ 0<\zeta< Z_Q(u):= \big(1-\|u\|^{Q}\big)^{-\frac{1}{Q-1}} \quad\quad (Z_Q(u)=+\infty \quad\text{whenever}\quad\|u\|=1),$$
we have 
\begin{align}\label{eq1.13}
\sup_{k\in\mathbb{N}}\int_{\mathbb{H}^n}\frac{\Phi_{\zeta\alpha_{\vartheta,Q},j_0}(u_k)}{r(\xi)^\vartheta}\,\mathrm{d}\xi<+\infty.
\end{align}
In addition, the constant $Z_Q(u)$ is sharp in the sense that the supremum \eqref{eq1.13} is infinite for $\zeta> Z_Q(u)$.
\end{theorem}
The obvious question that follows is yet unanswered:
 \[ \text{ Does the above supremum specified in \eqref{eq1.13} still hold true for}~ \zeta=Z_Q(u)~?\]     

The preceding result provides us with asymptotic lower and upper bounds for the sharp subcritical singular Moser-Trudinger type inequality. Indeed, we have the following theorem.
\begin{theorem}[\bf Sharp Subcritical Singular Moser-Trudinger Type Inequality]\label{thm5.1} Assume that $1<p<Q$,  $\vartheta\in[0,Q)$ and $0<\zeta <\alpha_Q$. Then there holds:
$$\text{SCMTI}(\zeta,\vartheta):=\sup_{u\in W,~\|D_H u\|_{Q}\leq 1}\cfrac{1}{\|D_H u\|_p^{p^\ast(1-\frac{\vartheta}{Q})}}\int_{\mathbb{H}^n}\frac{\Phi_{\zeta(1-\frac{\vartheta}{Q}),j_0}(u)}{r(\xi)^\vartheta}\,\mathrm{d}\xi<+\infty. $$
In addition, the constant $\alpha_Q$ is sharp in the sense that $\text{SCMTI}(\alpha_Q,\vartheta)=+\infty$. Further, there exist two suitable positive constants $c(\vartheta,p,Q)$ and $C(\vartheta,p,Q)$, which depend on $\vartheta$, $p$ and $Q$ such that when $\zeta\underset{\approx}{<}\alpha_Q$, there holds:
$$ \cfrac{c(\vartheta,p,Q)}{\Big(1-\big(\frac{\zeta}{\alpha_Q}\big)^{Q-1}\Big)^{\frac{p^\ast}{Q}\big(1-\frac{\vartheta}{Q}\big)}}\leq \text{SCMTI}(\zeta,\vartheta)\leq \cfrac{C(\vartheta,p,Q)}{\Big(1-\big(\frac{\zeta}{\alpha_Q}\big)^{Q-1}\Big)^{\frac{p^\ast}{Q}\big(1-\frac{\vartheta}{Q}\big)}}.$$    
\end{theorem}
The next theorem addresses a necessary and sufficient condition for the sharp critical singular Moser-Trudinger type inequality to be bounded. Ultimately, we shall show that sharp critical and subcritical singular Moser-Trudinger type inequalities are equivalent.
\begin{theorem}[\bf Sharp Critical Singular Moser-Trudinger Type Inequality]\label{thm5.22} Assume that $1<p<Q$,  $\vartheta\in[0,Q)$ and $a,b>0$. Then there holds:
$$\text{CMTI}_{a,b}(\vartheta):=\sup_{u\in W,~\|D_H u\|^a_{Q}+\|D_H u\|^b_p\leq 1}\int_{\mathbb{H}^n}\frac{\Phi_{\alpha_{\vartheta,Q},j_0}(u)}{r(\xi)^\vartheta}\,\mathrm{d}\xi<+\infty\iff b\leq Q. $$
In addition, the constant $\alpha_Q$ is sharp in the sense that the above supremum is infinite when a larger constant replaces $\alpha_Q$. Notably, the following identity holds:
\begin{align}\label{eq1.15}
  \text{CMTI}_{a,b}(\vartheta)=\sup_{\zeta\in(0,\alpha_Q)}\Bigg(\cfrac{1-\big(\frac{\zeta}{\alpha_Q}\big)^{(\frac{Q-1}{Q})a}}{\big(\frac{\zeta}{\alpha_Q}\big)^{(\frac{Q-1}{Q})b}}\Bigg)^{\frac{p^\ast}{b}\big(1-\frac{\vartheta}{Q}\big)}\text{SCMTI}(\zeta,\vartheta).  
\end{align}
In particular, there hold $\text{CMTI}_{Q,Q}(\vartheta)<+\infty$ and
$$\text{CMTI}_{Q,Q}(\vartheta)=\sup_{\zeta\in(0,\alpha_Q)}\Bigg(\cfrac{1-\big(\frac{\zeta}{\alpha_Q}\big)^{Q-1}}{\big(\frac{\zeta}{\alpha_Q}\big)^{Q-1}}\Bigg)^{\frac{p^\ast}{Q}\big(1-\frac{\vartheta}{Q}\big)}\text{SCMTI}(\zeta,\vartheta).   $$
\end{theorem}
The compact embedding described below plays a significant role in our analysis. Its statement can be interpreted as follows.
\begin{theorem}\label{thm1.4}
Let $1<p<Q$ and $\vartheta\in(0,Q)$. Then the embedding $W\hookrightarrow L^\theta(\mathbb{H}^n,r(\xi)^{-\vartheta}\mathrm{d}\xi)$ is compact for all $\theta\in[p^\ast,+\infty)$.
\end{theorem}
Due to the applications of Theorems \ref{thm1.1}, \ref{thm1.2}, and \ref{thm1.4},  we study the ground state solutions to a class of zero mass $(p,Q)$-Laplace subelliptic equations with singular exponential growth in $\mathbb{H}^n$ as follows:
\begin{equation}\label{main problem}
   -\Delta_{H,p} u-\Delta_{H,Q} u=\frac{f(\xi,u)}{r(\xi)^\vartheta}\quad\text{in}\quad \mathbb{H}^n,\tag{$\mathcal{P}$}
\end{equation}
with $1<p<Q$, $\vartheta\in(0,Q)$ and $Q=2n+2$ is the homogeneous dimension of the Heisenberg group $\mathbb{H}^n$. The nonlinear term $f:\mathbb{H}^n\times \mathbb{R}\to \mathbb{R}$ is a Carath\'eodory function, which satisfies some hypotheses and the function $r(\cdot)$ is known as the Kor\'anyi norm in $\mathbb{H}^n$, which is defined as follows:
$$r(\xi)=r(z,t)=(|z|^4+t^2)^{\frac{1}{4}},~\forall~\xi\in \mathbb{H}^n,$$
where $\xi=(z,t)\in \mathbb{H}^n$, $z=(x,y)=(x_1,\dots,x_n,y_1,\dots,y_n)\in\mathbb{R}^{2n}$, $t\in \mathbb{R}$, and $|z|$ is the Euclidean norm in $\mathbb{R}^{2n}$. In addition, the operator $\Delta_{H,m} u=\text{div}_H(|D_H u|_H^{m-2}D_H u)$ with $m>1$ is the standard $m$-Laplace subelliptic operator of $u$, where $D_H u$ is the subelliptic gradient (horizontal gradient) of  $u$, which is the $2n$-dimensional vector defined by
$$ D_H u=(X_1 u,\dots,X_n u,Y_1 u,\dots,Y_n u), $$
and $\{X_j, Y_j\}_{j=1}^{n}$ is the standard basis of the horizontal left-invariant vector fields on $\mathbb{H}^n$ with
$$X_j=\frac{\partial}{\partial x_j}+2y_j \frac{\partial}{\partial t}\quad\text{and}\quad Y_j=\frac{\partial}{\partial y_j}-2x_j \frac{\partial}{\partial t},~\forall~ j\in\{1,\dots,n\}. $$
Throughout this article, we assume that the nonlinearity $f:\mathbb{H}^n\times \mathbb{R}\to \mathbb{R}$ satisfies:
\begin{itemize}
    \item [($f_1$)] Let $f:\mathbb{H}^n\times \mathbb{R}\to \mathbb{R}$ be a Carath\'eodory function such that $f(\xi,s)=0$ for all $(\xi,s)\in \mathbb{H}^n\times (-\infty,0]$ and $f(\xi,s)>0$ for all $(\xi,s)\in \mathbb{H}^n\times (0,+\infty)$. In addition, there exists a constant $q>\max\{p^\ast,Q\}$, where $p^\ast:=\frac{Qp}{Q-p}$ is the critical Sobolev exponent such that for all $\xi\in \mathbb{H}^n$,  $$f(\xi,s)=o(|s|^{q-1})\quad\text{as}\quad s\to 0.$$  
    \item [($f_2$)] The function $f:\mathbb{H}^n\times \mathbb{R}\to \mathbb{R}$ has critical exponential growth at infinity, that is, it behaves like $\exp{(\alpha|s|^{\frac{Q}{Q-1}})}$ as $|s|\to~+\infty$ for some $\alpha>0$, which means that there exists a positive constant $\alpha_0$ such that  
    \begin{align*}
	\lim_{|s|\to +\infty} \frac{|f(\xi,s)|}{\operatorname{exp}\big(\alpha|s|^{\frac{Q}{Q-1}}\big)}=
	\begin{cases}
		0       & \text{if } \alpha>\alpha_0, \\
		+\infty & \text{if }\alpha<\alpha_0, 
	\end{cases}~~~~\text{uniformly with respect to}~~ \xi\in \mathbb{H}^n.
\end{align*}
      \item [($f_3$)] There exists a constant $\mu>Q$ such that
      $$0<\mu F(\xi,s):=\mu\int_{0}^{s} f(\xi,t)\,\mathrm{d}t\leq f(\xi,s)s\quad\text{for all}\quad(\xi,s)\in \mathbb{H}^n\times (0,+\infty). $$
       \item [($f_4$)] There exist positive constants $s_0$ and $K_0$ such that 
       $$0< F(\xi,s)\leq K_0f(\xi,s)\quad\text{for all}\quad(\xi,s)\in \mathbb{H}^n\times[s_0,+\infty). $$
       \item [($f_5$)] There exists a constant $\varrho>\max\{p^\ast,Q\}$ and $\beta>0$ such that 
       $$ F(\xi,s)\geq \beta s^\varrho\quad\text{for all}\quad(\xi,s)\in\mathbb{H}^n\times (0,+\infty).$$
       \item [($f_6$)] The function $s\mapsto\displaystyle\frac{f(\xi,s)}{s^{Q-1}}$ is strictly increasing for all $(\xi,s)\in \mathbb{H}^n\times (0,+\infty)$.
\end{itemize}

In virtue of $(f_1)$ and $(f_2)$, for any fixed $\ell>\max\{p^\ast,Q\}$ and $\alpha>\alpha_0$, there exists $\varepsilon>0$ and a constant $C_\varepsilon>0$ depending on $\varepsilon,~\alpha$ and $\ell$ such that 
\begin{align}\label{eq1.1}
    |f(\xi,s)|\leq \varepsilon |s|^{q-1}+C_\varepsilon|s|^{\ell-1}\Phi_{\alpha,j_0}(s)\quad\text{for all}\quad(\xi,s)\in\mathbb{H}^n\times\mathbb{R}
\end{align}
and thus, by employing $(f_3)$, we have
\begin{align}\label{eq1.2}
    |F(\xi,s)|\leq \varepsilon |s|^{q}+C_\varepsilon|s|^{\ell}\Phi_{\alpha,j_0} (s)\quad\text{for all}\quad(\xi,s)\in\mathbb{H}^n\times\mathbb{R},
\end{align}
where the Young function $\Phi_{\alpha,j_0}(\cdot)$ is given in \eqref{eq1.55}. It is evident that all of the aforementioned assumptions on the nonlinear term are satisfied by the following function. 
\begin{example}
A model function satisfying the hypotheses $(f_1)-(f_6)$ can be considered as $$F(\xi,s):=\beta (s^+)^\varrho \operatorname{exp}\Big(\alpha_0(s^+)^{\frac{Q}{Q-1}}\Big)\quad\text{for all}\quad(\xi,s)\in\mathbb{H}^n\times\mathbb{R},$$ where $Q=2n+2$, $\beta>0$, $0<\alpha_0 <\alpha$, $\varrho>q,\mu>\max\{p^\ast,Q\}$ and $f(\xi,s)=\frac{\partial F(\xi,s)}{\partial s}$ for all $(\xi,s)\in \mathbb{H}^n\times \mathbb{R}$.
\end{example}
Next, we define the weak formulation to the problem \eqref{main problem} as follows.
\begin{definition}
    We say that $u\in W$ is a weak solution to the problem \eqref{main problem}, if there holds
    $$\int_{\mathbb{H}^n}|D_H u|_H^{p-2}(D_H u,D_H \psi)_H\,\mathrm{d}\xi+\int_{\mathbb{H}^n}|D_H u|_H^{Q-2}(D_H u,D_H \psi)_H\,\mathrm{d}\xi=\int_{\mathbb{H}^n} \frac{f(\xi,u)}{r(\xi)^\vartheta}\psi\,\mathrm{d}\xi,~\forall~\psi\in W.$$  
\end{definition}
Let $I \colon W\to \mathbb{R}$ be the associated energy functional to the problem \eqref{main problem}, which is given by 
\begin{align}\label{eq1.18}
   I(u):=\frac{1}{p}\|D_H u\|^p_p+\frac{1}{Q}\|D_H u\|^Q_Q-\int_{\mathbb{H}^n}\frac{F(\xi,u)}{r(\xi)^\vartheta}\,\mathrm{d}\xi,~\forall~ u\in W.  
\end{align}
It is easy to see that $I$ is well-defined, of class $C^1(W,\mathbb{R})$ and its Gâteaux derivative is given by
\begin{align*}
    \langle{I^\prime(u),\psi}\rangle=\int_{\mathbb{H}^n}\Big( |D_H u|_H^{p-2}(D_H u,D_H \psi)_H+|D_H u|_H^{Q-2}(D_H u,D_H \psi)_H\Big)\,\mathrm{d}\xi-\int_{\mathbb{H}^n}\frac{f(\xi,u)}{r(\xi)^\vartheta}\psi\,\mathrm{d}\xi,
\end{align*}
for all $u,\psi\in W$. By using standard arguments, one sees that the critical points of $I$ are exactly weak solutions to the problem \eqref{main problem}. Moreover, we say that a solution $v_0\in W$  to the problem \eqref{main problem} is a ground state solution, if there holds:
$$ I(v_0)=\inf\{I(u)\colon~u\in W\setminus\{0\}~~\text{and}~~I^\prime(u)=0\}.$$

Now, we state the main existence result for ground state solutions to the problem \eqref{main problem} as follows.
\begin{theorem} \label{thm1.5}
Let $1<p<Q$, $\vartheta\in(0,Q)$ and hypotheses $(f_1)-(f_6)$ be satisfied. In addition, suppose that there exists $\beta_0>0$ large enough such that $(f_5)$ holds for all $\beta\geq \beta_0$, then the problem \eqref{main problem} has a positive ground state solution.
\end{theorem}

The rest of this article is organized as follows: In Section \ref{sec2}, we briefly discuss the notations and provide some essential results in the Heisenberg group setting, which will be used in the following main sections. In Section \ref{sec3} and Section \ref{sec4}, by using a rearrangement-free approach, we shall prove Theorem \ref{thm1.1} as well as Theorem \ref{thm1.2}. In Section \ref{sec5} and Section \ref{sec6}, we shall provide the proof of Theorem \ref{thm5.1} and Theorem \ref{thm5.22}, respectively. Furthermore, we establish a compact embedding result; for instance, see Theorem \ref{thm1.4} in Section \ref{sec7}. Eventually, we shall apply the mountain pass theorem, Theorems \ref{thm1.1}, \ref{thm1.2}, and \ref{thm1.4} to prove Theorem \ref{thm1.5} in Section \ref{sec8}.
\section{Some useful preliminary results}\label{sec2}
In this section, we first provide some fundamental properties of the Heisenberg group and technical lemmas that will be utilized in the next sections.

The Heisenberg group $(\mathbb{H}^n,o):=(\mathbb{R}^{2n+1},o) $ is a non-abelian group, which is equipped with the group law
$$\xi o~ \xi^\prime:=\bigg(z+z^\prime,t+t^\prime+2\displaystyle\sum_{i=1}^{n}(y_ix_i^\prime-x_iy_i^\prime)\bigg),~\forall~\xi,~  \xi^\prime\in \mathbb{H}^n,
$$
where $$\xi=(z,t)=(x_1,\dots,x_n,y_1,\dots,y_n,t)~~\text{ and}~~ \xi^\prime=(z^\prime,t^\prime)=(x^\prime_1,\dots,x^\prime_n,y^\prime_1,\dots,y^\prime_n,t^\prime).$$ 
Additionally, we have $$\xi^{-1}=-\xi\quad\text{ and}\quad (\xi o~ \xi^\prime)^{-1}=(\xi^\prime)^{-1} o~\xi^{-1}.$$
The standard basis $\{X_j, Y_j\}_{j=1}^{n}$ for the Lie algebra of left-invariant vector fields on $\mathbb{H}^n$ is given by
$$X_j=\frac{\partial}{\partial x_j}+2y_j \frac{\partial}{\partial t},~~ Y_j=\frac{\partial}{\partial y_j}-2x_j \frac{\partial}{\partial t},~T=\frac{\partial}{\partial t}\quad\text{for}\quad j=1,\dots,n. $$ 
It satisfies the Heisenberg canonical commutation relations:
$$[X_j,Y_k]=-4 \delta_{jk}T,~[X_j,X_k]=[X_j,T]=[Y_j,Y_k]=[Y_j,T]=0.$$
Since all the commutators of length greater than two vanish, $\mathbb{H}^n$ is a nilpotent, graded, and stratified group of step two. A left invariant vector field $X$ that belongs to the span of $\{X_j, Y_j\}_{j=1}^{n}$, is called horizontal. For any $R>0$, the dilation function $\delta_R:\mathbb{H}^n\to\mathbb{H}^n$ is given by
$$ \delta_R(\xi)=(Rz,R^2t),~\forall~\xi=(z,t)\in \mathbb{H}^n. $$
The Jacobian determinant of $\delta_R$ is $R^{Q}$, where $Q=2n+2$ is the homogeneous dimension of $\mathbb{H}^n$. Moreover, the Kor\'anyi norm in $\mathbb{H}^n$ is defined by
$$r(\xi)=r(z,t)=(|z|^4+t^2)^{\frac{1}{4}},~\forall~\xi=(z,t)\in \mathbb{H}^n.$$
Further, the corresponding distance is called the Kor\'anyi distance and is given by 
$$ d_K(\xi,\xi^\prime)=r(\xi^{-1}o ~\xi^\prime),~\forall~(\xi,\xi^\prime)\in \mathbb{H}^n\times \mathbb{H}^n. $$ 
Note that the  Kor\'anyi norm is homogeneous of degree 1 with respect to the dilation. Indeed, for any $R>0$, we have
$$ r(\delta_R(\xi))=r(Rz,R^2t)=(|Rz|^4+R^4t^2)^{\frac{1}{4}}=R~ r(\xi),~\forall~\xi=(z,t)\in \mathbb{H}^n.$$
The Kor\'anyi open ball in $\mathbb{H}^n$ of radius $R$ and centered at $\xi_0$ is defined by
 $$B_R(\xi_0)=\{\xi\in \mathbb{H}^n:~d_K(\xi,\xi_0)<R \}.$$ For simplicity, we put $B_R=B_R(O)$, where $O=(0,0)$ is the natural origin of $\mathbb{H}^n$.
 
The Lebesgue measure on $\mathbb{R}^{2n+1}$
remains invariant under the left translations of the Heisenberg
group. In addition, since the Haar measures on Lie groups are unique up to constant multipliers, we symbolize by $\mathrm{d}\xi$ the Haar measure on $\mathbb{H}^n$, which actually coincides with $(2n + 1)$–Lebesgue measure. For any set $U\subset\mathbb{H}^n$, $|U|$ denotes its $(2n + 1)$–Lebesgue measure. Again, the Haar measure on $\mathbb{H}^n$ is $Q$-homogeneous with respect to $\delta_R$. Notice that we also have
$$|\delta_R(U)| = R^Q|U|,\quad \mathrm{d}(\delta_R (\xi)) = R^Q\mathrm{d}\xi \quad\text{and} \quad|B_R|=R^Q|B_1|.$$
Let $u\in C^1(\mathbb{H}^n,\mathbb{R})$ be fixed. Then\textcolor{blue}{,} the subelliptic gradient (horizontal gradient) $D_H u$ is a $2n$-dimensional vector, which is given by
$$ D_H u=(X_1 u,\dots,X_n u,Y_1 u,\dots,Y_n u). $$
In addition, we denote
$$|D_H u|_H=\sqrt{(D_H u,D_H u)_H}=\bigg(\sum_{j=1}^{n} (X_ju)^2+(Y_ju)^2\bigg)^{\frac{1}{2}},$$
where $(\cdot,\cdot)_H$ is the inner product and $|\cdot|_H$ is the corresponding norm in $\mathbb{R}^{2n}$.
Note that, if $g\in C^1(\mathbb{R)}$, then $$D_H g(u) =g^\prime(u) D_H u. $$
Let $A\colon\xi\in\mathbb{H}^n\to A(\xi)=(A_1\xi,\cdots,A_{2n}\xi)$ be a $C^1$ vector field. Then\textcolor{blue}{,} horizontal divergence of $A$ is defined by
$$\text{div}_H A(\xi)=\sum_{j=1}^{n} X_j(A_j\xi)+\sum_{j=1}^{n} Y_j(A_{n+j}\xi),~~\forall~\xi \in\mathbb{H}^n.$$
Similarly, if $u\in C^2(\mathbb{H}^n)$, then the sub-Laplace (the horizontal Laplace) operator in $\mathbb{H}^n$ of u is given by
$$\Delta_H u=\displaystyle\sum_{j=1}^{n}(X_j^2+Y_j^2)u=\displaystyle\sum_{j=1}^{n}\bigg(\frac{\partial^2}{\partial x_j^2}+\frac{\partial^2}{\partial y_j^2}+4y_j \frac{\partial^2}{\partial x_j\partial t}-4x_j\frac{\partial^2}{\partial y_j\partial t}\bigg)u+4|z|^2\frac{\partial^2 u}{\partial t^2}. $$
A widespread generalization of the sub-Laplace operator is the $p$-Laplace subelliptic operator with $p>1$ on the Heisenberg group, which is  defined by
$$\Delta_{H,p} u=\text{div}_H(|D_H u|_H^{p-2}D_H u),~\forall~u\in C^\infty_0(\mathbb{H}^n). $$
For $m\geq 1$, we define
$$
  HW^{1,m}(\mathbb{H}^n)=\big\{ u\in L^m(\mathbb{H}^n):|D_H u|_H\in L^m(\mathbb{H}^n)\big\},   
$$
endowed with the norm
$$ \|u\|_{HW^{1,m}}=(\|u\|^m_m+\|D_H u\|^m_m)^{\frac{1}{m}}.$$
It is easy to see that the space $(HW^{1,m}(\mathbb{H}^n),\|\cdot\|_{HW^{1,m}})$ is a reflexive and separable Banach space. Moreover, $C^\infty_0(\mathbb{H}^n)$ is a dense subset of $HW^{1,m}(\mathbb{H}^n)$. In a similar manner, for any open set $\Omega$ in $\mathbb{H}^n$, we can define the space $(HW^{1,m}(\Omega),\|\cdot\|_{HW^{1,m}(\Omega)})$. Further, the space $HW^{1,m}_0(\Omega)$ is defined as the closure of $C^\infty_0(\Omega)$ with respect to the norm of $HW^{1,m}(\Omega)$.

Next, we shall present some basic properties of Schwarz symmetrization. For this instance, we refer the readers to see \cite{Baernstein-2019, Bennett-Sharpley-1988}. First, we discuss the symmetrization of the Lebesgue measurable set $\Omega$ in $\mathbb{H}^n$. Let $\Omega^\ast$ be the open ball centered at the origin of $\mathbb{H}^n$ such that $|\Omega^\ast|=|\Omega|$. Additionally, we denote $|\Omega|=+\infty$ implies $\Omega^\ast=\mathbb{H}^n$ and $0<|\Omega|<+\infty$ implies $\Omega^\ast=B_R$ such that $|\Omega^\ast|=|\Omega|=\sigma_Q R^Q$, where $\sigma_Q$ is the volume of the unit ball in $\mathbb{H}^n$ and $R=R(\Omega)$ is the volume radius of $\Omega$. It is well known that $\Omega^\ast$ is called the symmetric decreasing rearrangement of $\Omega$. Let $g\colon \Omega\to\mathbb{R}$ be a measurable function vanishing at infinity, that is, there holds:
\begin{align*}
 |\{\xi\in \Omega\colon~|g(\xi)|>t\}|=\int_{\{\xi\in \Omega\colon~|g(\xi)|>t\}} \mathrm{d}\xi<+\infty,~\forall~t>0.  
\end{align*}
Its distribution function $\mu_g\colon[0,+\infty)\to [0,+\infty]$ is defined by $$\mu_g(t)=|\{\xi\in \Omega\colon~|g(\xi)|>t\}|,~~\forall~t\geq 0.$$ 
 The unidimensional decreasing rearrangement of $g$ is a function $g^\sharp\colon[0,|\Omega|]\to [0,+\infty]$ defined by $$g^\sharp(0):=\text{ess sup}~ g\quad\text{and}\quad g^\sharp(s):=\inf\{t\geq 0\colon~\mu_g(t)<s\},~\forall~s\in(0,|\Omega|].$$ 
Indeed, it has been observed by Bennett-Sharpley \cite{Bennett-Sharpley-1988} that the following identity also holds:
$$g^\sharp(s)=m_{\mu_g}(s)=\sup\{t\geq 0:~\mu_g(t)>s\},~\forall~s\geq 0,$$
where $m_{\mu_g}$ is the distribution function of $\mu_g$. Further, $g^\sharp$ is a nonnegative, decreasing, and right continuous function on $[0,+\infty)$. Now, we define $g^\ast\colon \Omega^\ast\to [0,+\infty]$, which is the Schwarz symmetrization of $g$, alternatively known as the spherically symmetric decreasing rearrangement of $g$, and given by 
$$ g^\ast(\xi):= g^\sharp\big(\sigma_Q r(\xi)^Q\big),~\forall ~\xi\in \Omega^\ast.$$
It is always a radially symmetric, nonnegative, and nonincreasing function. Further, if $\Psi\colon\mathbb{R}\to \mathbb{R}$ is nondecreasing, then 
\begin{equation}\label{eq2.1}
 (\Psi(g))^\ast=\Psi(g^\ast).   
\end{equation}
It has been observed that $g^\ast$ and $|g|$ are equimeasurable, that is, $\{\xi\colon~ g^\ast(\xi)>t\}=\{\xi\colon~ |g(\xi)|>t\}^\ast$. Let $u\colon \mathbb{H}^n\to\mathbb{R}$ be a nonnegative function. 
Then\textcolor{blue}{,} a weak form of the P\'olya–Szeg\"{o} types inequality in the Heisenberg group was proved in \cite[Theorem 5.3]{Manfredi-Vera de Serio-2019}. It states that
\begin{align}\label{eq2.2}
    \|D_H u^\ast\|_p\leq C \|D_H u\|_p,~~\forall~p\geq 1,
\end{align}
where $C\geq 1$ is a constant depending on $p$ and $Q$ only. The result stated below can be found in the proof of \cite[Theorem 1.4]{Chen-Lu-Zhu-2023} for $p=Q$, and it is straightforward to verify that 
\begin{align}\label{eq2.3}
    \|D_H |u|\|_p\leq \|D_H u\|_p,~~\forall~p\geq 1.
\end{align}

In addition, we note that the following Hardy-Littlewood inequality holds.
\begin{lemma}{\cite[Corollary 2.16]{Baernstein-2019}}\label{lem2.22}
Let $g$ and $h$ be two nonnegative measurable functions defined on $\mathbb{H}^n$ with $\mu_g(t)<+\infty$ and $\mu_h(t)<+\infty$ for all $t>0$. Then there holds 
$$\int_{\mathbb{H}^n} g(\xi)h(\xi)\,\mathrm{d}\xi\leq \int_{\mathbb{H}^n} g^\ast(\xi) h^\ast(\xi)\,\mathrm{d}\xi$$
with the notion that the left side is infinity, so is the right side.
\end{lemma}
The following lemma follows from \cite[Proposition 2.1]{Carvalho-Figueiredo-Furtado-Medeiros-2021}. However, for the sake of completeness, we include the proof of the preceding lemma.
\begin{lemma}\label{lem2.1}
 Let $1<p<Q$, then the embedding $W\hookrightarrow L^\tau (\mathbb{H}^n)$ is continuous for all $\tau\in [p^\ast,+\infty)$. Further, the embedding $W\hookrightarrow L^\tau (B_R)$ is compact for all $R>0$ and $\tau\in [1,+\infty)$.
\end{lemma}
\begin{proof} Let $\eta\in(p,Q)$ and $\tau\in [p^\ast,+\infty)$ be such that $\tau<\eta^\ast:=\frac{Q\eta}{Q-\eta}$. Then, without losing generality, there exists $\sigma\in(0,1)$ such that $\frac{1}{\eta}=\frac{\sigma}{Q}+\frac{1-\sigma}{p}$. By applying H\"older's inequality, we get
 \begin{align*}
    \|D_H u\|_\eta&=\bigg(\int_{\mathbb{H}^n}|D_H u|^\eta_H\,\mathrm{d}\xi\bigg)^{\frac{1}{\eta}}=\bigg(\int_{\mathbb{H}^n}|D_H u|^{\sigma\eta}_H |D_H u|^{(1-\sigma)\eta}_H\,\mathrm{d}\xi\bigg)^{\frac{1}{\eta}}\\
    &\leq \bigg(\int_{\mathbb{H}^n}|D_H u|^Q_H\,\mathrm{d}\xi\bigg)^{\frac{\sigma}{Q}}\bigg(\int_{\mathbb{H}^n}|D_H u|^p_H\,\mathrm{d}\xi\bigg)^{\frac{1-\sigma}{p}} \leq\|u\|^\sigma \|u\|^{1-\sigma}=\|u\|,
 \end{align*}
for all $u\in W$. It follows from the above inequality that the chain of embeddings $W\hookrightarrow S^{1,\eta}(\mathbb{H}^n)\hookrightarrow L^{\eta^\ast}(\mathbb{H}^n)$ is continuous. Recall that the embeddings $W\hookrightarrow S^{1,p}(\mathbb{H}^n)\hookrightarrow L^{p^\ast}(\mathbb{H}^n)$ hold and $\tau\in[p^\ast,\eta^\ast)$. Thus, there exists $\lambda\in(0,1)$ such that $\frac{1}{\tau}=\frac{\lambda}{\eta^\ast}+\frac{1-\lambda}{p^\ast}$. Due to the interpolation inequality, we have
$$\|u\|_\tau\leq\|u\|^\lambda_{\eta^\ast}\|u\|^{1-\lambda}_{p^\ast}\lesssim \|u\|^\lambda\|u\|^{1-\lambda}=\|u\|<+\infty. $$
So, the embedding $W\hookrightarrow L^\tau (\mathbb{H}^n)$ is continuous for all $\tau\in [p^\ast,+\infty)$. Next, we claim that the embedding $W\hookrightarrow HW^{1,Q}(B_R)$ is continuous for all $R>0$. Indeed, let us choose $\theta>\max\{p^\ast,Q\}$, then by using H\"older's inequality and the embedding $W\hookrightarrow L^\theta(\mathbb{H}^n)$, one has
\begin{align*}
\int_{B_R}|u|^Q\,\mathrm{d}\xi&\leq\bigg(\int_{B_R}\,\mathrm{d}\xi\bigg)^{\frac{\theta-Q}{\theta}} \bigg(\int_{B_R}|u|^\theta\,\mathrm{d}\xi\bigg)^{\frac{Q}{\theta}}\leq |B_R|^{\frac{\theta-Q}{\theta}}\bigg(\int_{\mathbb{H}^n}|u|^\theta\,\mathrm{d}\xi\bigg)^{\frac{Q}{\theta}}\lesssim |B_R|^{\frac{\theta-Q}{\theta}}\|u\|^Q.
\end{align*}
As a byproduct of the above inequality, we obtain 
\begin{align*}
 \|u\|_{HW^{1,Q}(B_R)}\lesssim  \sqrt[Q]{|B_R|^{\frac{\theta-Q}{\theta}}\|u\|^Q+\|D_H u\|^Q_Q}\lesssim \sqrt[Q]{\Big(|B_R|^{\frac{\theta-Q}{\theta}}+1\Big)}\|u\|,~\forall~u\in W.
\end{align*}
This shows the claim. Now, we recall that the embedding $HW^{1,Q}(B_R)\hookrightarrow L^\tau(B_R)$ is compact for all $R>0$ and $\tau\in[1,+\infty)$, thanks to Pucci-Temperini \cite[Page No. $608$, Eq.~$(26)$]{Pucci-Temperini-2022}. Hence, the embedding $W\hookrightarrow L^\tau (B_R)$ is compact for all $\tau\in [1,+\infty)$. This finishes the proof.
\end{proof}
In reference to the work of Yang \cite[Lemma 2.1 and Lemma 2.2]{Yang-2012@}, we present the following result:.  
\begin{lemma}\label{lem2.2}
The function $s\mapsto\operatorname{\exp}(s)-\sum_{j=0}^{j
_0-1}\frac{s^j}{j!}$ is continuous, increasing, and convex in $[0,+\infty)$, where $j_0$ is defined in \eqref{eq1.55}. Additionally, for any $\alpha>0$ and $m\geq1$, we have
 $$ \bigg(\operatorname{exp}(\alpha|s|^{\frac{Q}{Q-1}})-\sum_{j=0}^{j_0-1} \frac{\alpha^j}{j!}|s|^{\frac{Qj}{Q-1}}\bigg)^m\leq  \operatorname{exp}(\alpha m|s|^{\frac{Q}{Q-1}})-\sum_{j=0}^{j_0-1} \frac{\alpha^jm^j}{j!}|s|^{\frac{Qj}{Q-1}},~\forall~s\in\mathbb{R}. $$
\end{lemma}
\section{Sharp singular Moser-Trudinger type inequality: Proof of Theorem \ref{thm1.1}}\label{sec3}
In this section, we focus on proving a variant of the singular Moser-Trudinger type inequality by employing a rearrangement-free approach introduced by Lam-Lu \cite{Lam-Lu-201@}. However, we shall also deduce a consequence result of Theorem \ref{thm1.1}, which is based on Schwarz symmetrization arguments and the utilization of a weaker form of the P\'olya–Szeg\"{o} types inequality (see \cite[Theorem 5.3]{Manfredi-Vera de Serio-2019}) in the Heisenberg group setting.
\begin{proof}[\bf{Proof of Theorem \ref{thm1.1}}.]
Due to the density argument of $C^\infty_0(\mathbb{H}^n)$ in $W$ and \eqref{eq2.3}, without loss of generality, one can assume $u\in C^\infty_0(\mathbb{H}^n)\setminus\{0\}$, $u\geq 0$ and $\|u\|\leq 1$. In addition, we define
$$\digamma(u):=2^{-\frac{1}{Q(Q-1)}}\|D_H u\|_p\qquad\text{and}\qquad\mathbb{U}:=\{\xi\in \mathbb{H}^n\colon~u(\xi)>\digamma(u)\}.$$
By employing the embedding $S^{1,p}(\mathbb{H}^n)\hookrightarrow L^{p^\ast}(\mathbb{H}^n)$, one has
\begin{equation}\label{eq3.1}
    \digamma(u)<1\quad\text{and}\quad |\mathbb{U}|\leq 2^{\frac{p^\ast}{Q(Q-1)}}C(p,Q),
\end{equation}
where $C(p,Q)>0$ is a suitable constant depending on $p$ and $Q$. It follows that $\mathbb{U}$ is a bounded set in $\mathbb{H}^n$. Notice that
$$\int_{\mathbb{H}^n}\frac{\Phi_{\alpha,j_0}(u)}{r(\xi)^\vartheta}\,\mathrm{d}\xi=\bigg(\int_{\mathbb{H}^n\setminus \mathbb{U}}+\int_{\mathbb{U}}\bigg)\frac{\Phi_{\alpha,j_0}(u)}{r(\xi)^\vartheta}\,\mathrm{d}\xi:=J_1+J_2. $$
In view of \eqref{eq3.1}, we deduce that $\mathbb{H}^n\setminus \mathbb{U}\subset \{\xi\in \mathbb{H}^n:~u(\xi)<1\}$. Thus, by using Lemma \ref{lem2.1} and $\|u\|\leq 1$, we obtain
\begin{equation}\label{eq3.2}
\begin{aligned}
J_1 & \leq \int_{\{\xi\in \mathbb{H}^n\colon~ u(\xi)<1\}}\frac{1}{r(\xi)^\vartheta}\sum_{j=j_0}^\infty \frac{\alpha^j}{j!}|u|^{\frac{Qj}{Q-1}}\,\mathrm{d}\xi \leq \int_{\{\xi\in \mathbb{H}^n\colon~ u(\xi)<1\}}\frac{1}{r(\xi)^\vartheta}\sum_{j=j_0}^\infty \frac{\alpha^j}{j!}|u|^{p^\ast}\,\mathrm{d}\xi \\ 
 & \leq \operatorname{\exp}(\alpha) \bigg(\int_{\{\xi\in \mathbb{H}^n\colon~ u(\xi)<1,~r(\xi)<1\}}\frac{|u|^{p^\ast}}{r(\xi)^\vartheta}\,\mathrm{d}\xi+  \int_{\{\xi\in \mathbb{H}^n\colon~u(\xi)<1,~r(\xi)\geq1\}}\frac{|u|^{p^\ast}}{r(\xi)^\vartheta}\,\mathrm{d}\xi\bigg)\\ 
& \leq \operatorname{\exp}(\alpha)\bigg( \int_{\{ \xi\in \mathbb{H}^n\colon~r(\xi)<1\}}\frac{1}{r(\xi)^\vartheta}\,\mathrm{d}\xi+\int_{\{ \xi\in \mathbb{H}^n\colon~r(\xi)\geq1\}}|u|^{p^\ast}\,\mathrm{d}\xi\bigg) \leq C(\alpha,\vartheta,p,Q),
\end{aligned}
\end{equation}
where $C(\alpha,\vartheta,p,Q)>0$ is a suitable constant depending on  $\alpha$, $\vartheta$, $p$ and $Q$. Next, to evaluate $J_2$, we define the function $w\colon\mathbb{U}\to\mathbb{R}$ by $$w(\xi):=u(\xi)-\digamma(u),~\forall~\xi\in\mathbb{U}.$$ It follows directly that $w\in HW^{1,Q}_0(\mathbb{U})$. Recall the following elementary inequality
$$(a+b)^\eta\leq a^\eta+\eta 2^{\eta-1}(a^{\eta-1}b+b^\eta),~\forall~\eta\geq 1~~\text{and}~~a,b\geq 0.$$
By using the above inequality together with Young's inequality, we get
\begin{align*}
    |u(\xi)|^{\frac{Q}{Q-1}}&\leq (|w(\xi)|+|\digamma(u)|)^{\frac{Q}{Q-1}}\leq |w(\xi)|^{\frac{Q}{Q-1}}+\frac{Q}{Q-1}~2^{\frac{1}{Q-1}}\Big(|w(\xi)|^{\frac{1}{Q-1}}|\digamma(u)|+|\digamma(u)|^{\frac{Q}{Q-1}}\Big)\\
    &\leq |w(\xi)|^{\frac{Q}{Q-1}}+\frac{Q}{Q-1}~2^{\frac{1}{Q-1}}\bigg(\frac{|w(\xi)|^{\frac{Q}{Q-1}}|\digamma(u)|^Q}{Q}+\frac{Q-1}{Q}+|\digamma(u)|^{\frac{Q}{Q-1}}\bigg)\\
    &\leq |w(\xi)|^{\frac{Q}{Q-1}}\bigg(1+\frac{2^{\frac{1}{Q-1}}}{Q-1}~|\digamma(u)|^Q\bigg)+D_Q,~\forall~\xi\in \mathbb{U},
\end{align*}
where $$D_Q:=\frac{Q}{Q-1}~2^{\frac{1}{Q-1}}\bigg(\frac{Q-1}{Q}+1\bigg) .$$
Again, we define $\psi\colon\mathbb{U}\to\mathbb{R}$ as follows:
$$\psi(\xi):=\bigg(1+\frac{2^{\frac{1}{Q-1}}}{Q-1}~|\digamma(u)|^Q\bigg)^{\frac{Q-1}{Q}}w(\xi),~~\forall~\xi\in\mathbb{U}. $$
It follows immediately that $\psi\in HW^{1,Q}_0(\mathbb{U})$ and $|u|^{\frac{Q}{Q-1}}\leq |\psi|^{\frac{Q}{Q-1}}+D_Q$ in $\mathbb{U}$. Now, by using the inequality $(1-x)^\eta\leq 1-\eta x$ for all $x\in[0,1]\textcolor{blue}{,}~\eta\in(0,1]$ and $\|u\|\leq1$, we obtain
\begin{align}
 \|D_H \psi\|^{\frac{Q}{Q-1}}_{L^Q(\mathbb{U})}&\notag=\bigg(1+\frac{2^{\frac{1}{Q-1}}}{Q-1}~|\digamma(u)|^Q\bigg) \|D_H w\|^{\frac{Q}{Q-1}}_{L^Q(\mathbb{U})}\leq\bigg(1+\frac{1}{Q-1}\|D_H u\|^Q_p\bigg)\|D_H u\|^{\frac{Q}{Q-1}}_{Q}\\ \notag
 &\leq \bigg(1+\frac{1}{Q-1}\|D_H u\|^Q_p\bigg)\bigg(1-\frac{1}{Q-1}\|D_H u\|^Q_p\bigg)\\ \notag
 &=1-\frac{1}{(Q-1)^2}\|D_H u\|^{2Q}_p\leq 1.
\end{align}
This shows that $\psi\in HW^{1,Q}_0(\mathbb{U})$ and $\|D_H \psi\|_{L^Q(\mathbb{U})}\leq 1$. Hence, by employing the Moser-Trudinger inequality \cite[Theorem 1.4]{Lam-Lu-201@} and \eqref{eq3.1}, we have
\begin{equation}\label{eq3.3}
\begin{aligned}
  J_2&\leq \int_\mathbb{U}\frac{\operatorname{\exp}\big(\alpha|u|^{\frac{Q}{Q-1}}\big)}{r(\xi)^\vartheta}\,\mathrm{d}\xi\leq  \int_\mathbb{U}\frac{\operatorname{\exp}\big(\alpha_{\vartheta,Q}\big(|\psi|^{\frac{Q}{Q-1}}+D_Q\big)\big)}{r(\xi)^\vartheta}\,\mathrm{d}\xi \\
  &\leq \operatorname{\exp}(\alpha_{\vartheta,Q} D_Q)C(\vartheta,Q)|\mathbb{U}|^{1-\frac{\vartheta}{Q}}\leq C(\vartheta,p,Q),
\end{aligned}
\end{equation}
where $C(\vartheta,p,Q)>0$ is a suitable constant depending on $\vartheta$, $p$ and $Q$. By combining \eqref{eq3.2} and \eqref{eq3.3}, we conclude the validity of \eqref{eq1.12}.

To complete the proof, it suffices to check the sharpness of $\alpha_{\vartheta,Q}$. For this, we define
$$\xi=(z,t)\in\mathbb{H}^n\setminus\{O\},~z^\ast=\frac{z}{r(\xi)},~t^\ast=\frac{t}{r(\xi)^2}\quad\text{and}\quad\xi^\ast=(z^\ast,t^\ast). $$
It is easy to see that $\xi^\ast\in\Sigma:=\{\xi\in \mathbb{H}^n\colon~r(\xi)=1\}$, the Heisenberg unit sphere. Moreover, inspired by \cite{Cohn-Lu-2001}, we set
$$ c_m :=\int_\Sigma|z^\ast|^m\,\mathrm{d}\mu(\xi^\ast)=\omega_{2n-1}\Gamma\bigg(\frac{1}{2}\bigg)\Gamma\bigg(\frac{Q+m-2}{4}\bigg)\Gamma\bigg(\frac{Q+m}{4}\bigg)^{-1},~\forall~m\in\{p,Q\} ,$$
where $\mathrm{d}\mu(\xi^\ast)$ is the Radon measure on $\Sigma$, $\omega_{2n-1}=2\pi^n\Gamma(n)^{-1}$ is the surface measure of the unit sphere in $\mathbb{R}^{2n}$. In addition, one has $$\alpha_Q:=Qc_Q^{\frac{1}{Q-1}}=Q\bigg(2\pi^n\Gamma\bigg(\frac{1}{2}\bigg)\Gamma\bigg(\frac{Q-1}{2}\bigg){\Gamma\bigg(\frac{Q}{2}\bigg)}^{-1}{\Gamma(n)}^{-1}\bigg)^{Q^\prime-1}.$$
Let $s>0$ be fixed and choose the Moser's function sequence $\{v_k\}_k$ as follows:
\begin{align}\label{eq3.99}
  v_k(\xi):=\begin{cases}
    c^{-\frac{1}{Q}}_Q(\ln k)^{\frac{Q-1}{Q}}\quad\quad\quad\quad\quad\text{if}\quad r(\xi)\in[0,sk^{-1}],\\
    c^{-\frac{1}{Q}}_Q(\ln k)^{-\frac{1}{Q}}\ln\big(\frac{s}{r(\xi)}\big)\quad\quad\text{if}\quad r(\xi)\in[sk^{-1},s],\\
    0 \quad\quad\quad\quad\quad\quad\quad\quad\quad\quad\quad\text{if}\quad r(\xi)\in[s,+\infty).
\end{cases}   
\end{align}
Due to \cite[Lemma 2.1]{Cohn-Lu-2001}, one has $|D_H r(\xi)|_H=\frac{|z|}{r(\xi)}$, and thus, we have
$$ |D_H v_k|_H=c^{-\frac{1}{Q}}_Q(\ln k)^{-\frac{1}{Q}}\frac{|z|}{r(\xi)^2}\chi _{B_s\setminus B_{sk^{-1}}}(\xi).$$
By direct calculations, we deduce that
\begin{align*}
    \int_{\mathbb{H}^n}|D_H v_k|^p_H\,\mathrm{d}\xi&=\int_\Sigma\int_{sk^{-1}}^{s}\bigg(c^{-\frac{1}{Q}}_Q(\ln k)^{-\frac{1}{Q}}\frac{|z^\ast|}{\eta}\bigg)^p\eta^{Q-1}\,\mathrm{d}\eta\,\mathrm{d}\mu(\xi^\ast)\\
    &=c^{-\frac{p}{Q}}_Q(\ln k)^{-\frac{p}{Q}}c_p\int_{sk^{-1}}^{s}\eta^{Q-p-1}\,\mathrm{d}\eta=\frac{c^{-\frac{p}{Q}}_Q(\ln k)^{-\frac{p}{Q}}c_p s^{Q-p}}{Q-p}\bigg(1-\frac{1}{k^{Q-p}}\bigg)
\end{align*}
and 
\begin{align*}
    \int_{\mathbb{H}^n}|D_H v_k|^Q_H\,\mathrm{d}\xi&=\int_\Sigma\int_{sk^{-1}}^{s}\bigg(c^{-\frac{1}{Q}}_Q(\ln k)^{-\frac{1}{Q}}\frac{|z^\ast|}{\eta}\bigg)^Q\eta^{Q-1}\,\mathrm{d}\eta\,\mathrm{d}\mu(\xi^\ast)\\
    &=c^{-1}_Q(\ln k)^{-1}c_Q\int_{sk^{-1}}^{s}\eta^{-1}\,\mathrm{d}\eta=c^{-1}_Q(\ln k)^{-1}c_Q\ln\bigg(\frac{s}{sk^{-1}}\bigg)=1.
\end{align*}
From the above estimates, we obtain $\|v_k\|\to 1$ as $k\to\infty$. In addition, let $\alpha>\alpha_{\vartheta,Q}$, and thus\textcolor{blue}{,} there exists $\varepsilon_0>0$ such that $\alpha=(1+\varepsilon_0)\alpha_{\vartheta,Q}$. It follows that
\begin{align*}
  \sup_{u\in W,~\|u\|\leq 1}\int_{\mathbb{H}^n}\frac{\Phi_{\alpha,j_0}(u)}{r(\xi)^\vartheta}\,\mathrm{d}\xi&\geq \int_{\{\xi\in \mathbb{H}^n\colon~r(\xi)<sk^{-1}\}}\frac{\Phi_{\alpha,j_0}(v_k/\|v_k\|)}{r(\xi)^\vartheta}\,\mathrm{d}\xi\\
  &=\int_\Sigma\int_{0}^{sk^{-1}}\Phi_{\alpha,j_0}\Bigg(\frac{c^{-\frac{1}{Q}}_Q(\ln k)^{\frac{Q-1}{Q}}}{\|v_k\|}\Bigg)\eta^{Q-\vartheta-1}\,\mathrm{d}\eta\,\mathrm{d}\mu(\xi^\ast)\\
  &\gtrsim \bigg(\int_\Sigma \,\mathrm{d}\mu(\xi^\ast)\bigg)\frac{\big(sk^{-1}\big)^{Q-\vartheta}}{Q-\vartheta}\operatorname{\exp}\Bigg(\alpha\Bigg|\frac{c^{-\frac{1}{Q}}_Q(\ln k)^{\frac{Q-1}{Q}}}{\|v_k\|}\Bigg|^{\frac{Q}{Q-1}}\Bigg)\\
  &=\bigg(\int_\Sigma \,\mathrm{d}\mu(\xi^\ast)\bigg)\frac{s^{Q-\vartheta}}{Q-\vartheta}\operatorname{\exp}\Bigg[(Q-\vartheta)\ln k\Bigg(\frac{1+\varepsilon_0}{\|v_k\|^{\frac{Q}{Q-1}}}-1\Bigg)\Bigg]\\ &\to+\infty\quad\text{as}\quad k\to\infty.
\end{align*}
The proof of the theorem is complete at this point.
\end{proof}
The next result can be obtained by applying the above theorem and a similar strategy developed in the proof of \cite[Lemma 4.1]{Nguyen-2019}. For simplicity, we will present an alternative approach to establish the following corollary.
\begin{corollary}\label{cor3.1}
For all $\alpha>0$, $\vartheta\in[0,Q)$ and $u\in W$, there holds
$$\frac{\Phi_{\alpha,j_0}(u)}{r(\xi)^{\vartheta}}\in L^1(\mathbb{H}^n).$$ 
\end{corollary}
\begin{proof}
Let $u^\ast$ be the Schwarz symmetrization of $|u|$. Then, we obtain from \eqref{eq2.2} and \eqref{eq2.3} that $u^\ast\in W$. Moreover, since  $(1/r(\xi)^\vartheta)^\ast=1/r(\xi)^\vartheta$, by using Lemma \ref{lem2.22} and  \eqref{eq2.1}, one has
 $$\int_{\mathbb{H}^n}\frac{\Phi_{\alpha,j_0}(u)} {r(\xi)^\vartheta}\,\mathrm{d}\xi\leq \int_{\mathbb{H}^n}\frac{\Phi_{\alpha,j_0}(u^\ast)} {r(\xi)^\vartheta}\,\mathrm{d}\xi .$$
 Now\textcolor{blue}{,} we fix $R>0$, then we can write
 $$\int_{\mathbb{H}^n}\frac{\Phi_{\alpha,j_0}(u^\ast)} {r(\xi)^\vartheta}\,\mathrm{d}\xi=\bigg(\int_{\mathbb{H}^n\setminus B_R}+\int_{B_R}\bigg)\frac{\Phi_{\alpha,j_0}(u^\ast)} {r(\xi)^\vartheta}\,\mathrm{d}\xi:=J_1+J_2. $$
 In addition, by using the properties of the Schwarz symmetrization, for all $u^\ast\in L^m(\mathbb{H}^n)$ with $m\geq1$ and $s>0$, we have
 \begin{equation}\label{eq3.333}
 \begin{aligned}
     \int_{\mathbb{H}^n}u^\ast(\xi)^{m}\,\mathrm{d}\xi& = \int_{\mathbb{H}^n}u^\sharp\big(\sigma_Q r(\xi)^Q\big)^{m}\,\mathrm{d}\xi\geq \int_{\{\xi\in\mathbb{H}^n:~r(\xi)\leq s\}}u^\sharp\big(\sigma_Q r(\xi)^Q\big)^m\,\mathrm{d}\xi\\
     &=\int_\Sigma\int_{0}^{s}u^\sharp\big(\sigma_Q \eta^Q\big)^m\eta^{Q-1}\,\mathrm{d}\eta\,\mathrm{d}\mu(\xi^\ast)\geq \bigg(\int_\Sigma\,\mathrm{d}\mu(\xi^\ast)\bigg)\frac{s^Q}{Q}u^\sharp\big(\sigma_Q s^Q\big)^m. 
 \end{aligned}
 \end{equation}
 In particular, for all $m\geq 1$, one has
 \begin{align}\label{eq3.4}
 u^\ast(\xi)\leq C(m,Q)r(\xi)^{-\frac{Q}{m}}\|u^\ast\|_{m},~\forall~u^\ast\in L^m(\mathbb{H}^n)  \quad\text{and}\quad\xi\in\mathbb{H}^n\setminus\{O\},     
 \end{align}
 where $C(m,Q)>0$ is a constant depending on $m$ and $Q$, which is defined as follows: 
 $$C(m,Q):=\bigg(\frac{Q}{\int_\Sigma\,\mathrm{d}\mu(\xi^\ast)}\bigg)^{\frac{1}{m}}\quad\text{with}\quad \Sigma:=\{\xi\in \mathbb{H}^n\colon~r(\xi)=1\}. $$
By using the monotone convergence theorem, \eqref{eq3.4} and Lemma \ref{lem2.1}, we have
\begin{equation}\label{eq3.5}
 \begin{aligned}
     J_1&\leq \frac{1}{R^\vartheta}\int_{\mathbb{H}^n\setminus B_R}\sum_{j=j_0}^{\infty}\frac{\alpha^j}{j!}|u^\ast|^{\frac{Qj}{Q-1}}\,\mathrm{d}\xi=\frac{1}{R^\vartheta}\sum_{j=j_0}^{\infty}\int_{\mathbb{H}^n\setminus B_R}\frac{\alpha^j}{j!}|u^\ast|^{\frac{Qj}{Q-1}}\,\mathrm{d}\xi\\ 
     &=\frac{1}{R^\vartheta}\sum_{j=j_0}^{\infty}\frac{\alpha^j}{j!}\int_{\mathbb{H}^n\setminus B_R}|u^\ast|^{\frac{Qj}{Q-1}-p^\ast}|u^\ast|^{p^\ast}\,\mathrm{d}\xi\\ 
     &\leq \frac{\Big(C(p,Q) R^{\frac{\vartheta-Q}{p^\ast}}\Big)^{-p^\ast}}{\|u^\ast\|^{p^\ast}_{p^\ast}}\sum_{j=j_0}^{\infty}\frac{\Big(\alpha\Big(C(p,Q)R^{-\frac{Q}{p^\ast}}\|u^\ast\|_{p^\ast}\Big)^{\frac{Q}{Q-1}}\Big)^j}{j!}\int_{\mathbb{H}^n\setminus B_R}|u^\ast|^{p^\ast}\,\mathrm{d}\xi\\ 
     &\leq \Big(C(p,Q) R^{\frac{\vartheta-Q}{p^\ast}}\Big)^{-p^\ast}\operatorname{\exp}\bigg(\alpha\Big(\widetilde{C}(p,Q)R^{-\frac{Q}{p^\ast}}\|u^\ast\|\Big)^{\frac{Q}{Q-1}}\bigg)<+\infty,
 \end{aligned}
 \end{equation}
 where $\widetilde{C}(p,Q)>0$ is a suitable constant depending on $p$ and $Q$. Define a function $v:B_R\to\mathbb{R}$ by
$$v(\xi):=u^\ast(\xi)-u^\sharp(|B_R|),~\forall~\xi\in B_R.$$
It follows from the embedding $W\hookrightarrow HW^{1,Q}(B_R)$ and \eqref{eq3.333} that $v\in HW_0^{1,Q}(B_R)$.
Notice that for all $q>1$ and $\eta>0$, we have 
 \begin{align}\label{eq3.6}
     (a+b)^q\leq (1+\eta)a^q+C_\eta b^q,~\forall~a,b>0,~\text{with}~~~C_\eta=(1-(1+\eta)^{-\frac{1}{q-1}})^{1-q}.
\end{align}
Let $\sigma>1$ be such that $\frac{1}{\sigma}+\frac{1}{\sigma^\prime}=1$ and $\vartheta<\frac{Q}{\sigma}$. Thus, for each $\alpha>0$, it follows from the Moser-Trudinger inequality \cite[Lemma 4.1]{Lam-Lu-Tang-2012}, H\"older's inequality, Lemma \ref{lem2.1}, \eqref{eq3.333} and \eqref{eq3.6} that
\begin{equation}\label{eq3.7}
\begin{aligned}
    J_2 &\leq \int_{B_R}\frac{\operatorname{\exp}\big(\alpha|u^\ast|^{\frac{Q}{Q-1}}\big)}{r(\xi)^\vartheta}\,\mathrm{d}\xi\leq \int_{B_R}\frac{\operatorname{\exp}\Big(\alpha(1+\eta)|v|^{\frac{Q}{Q-1}}+\alpha C_\eta|u^\sharp(|B_R|)|^{\frac{Q}{Q-1}}\Big)}{r(\xi)^\vartheta}\,\mathrm{d}\xi \\ 
    &\leq \Bigg(\int_{B_R}\operatorname{\exp}\Big(\alpha \sigma^\prime C_\eta|u^\sharp(|B_R|)|^{\frac{Q}{Q-1}}\Big)\,\mathrm{d}\xi \Bigg)^{\frac{1}{\sigma^\prime}}\Bigg(\int_{B_R}\frac{\operatorname{\exp}\big(\alpha \sigma(1+\eta)|v|^{\frac{Q}{Q-1}}\big)}{r(\xi)^{\vartheta \sigma}}\,\mathrm{d}\xi \Bigg)^{\frac{1}{\sigma}}\\ 
    &\leq \bigg(\operatorname{\exp}\Big(\alpha \sigma^\prime C_\eta\Big(C(p,Q)R^{-\frac{Q}{p^\ast}}\|u^\ast\|_{p^\ast}\Big)^{\frac{Q}{Q-1}}\Big)|B_R|\bigg)^{\frac{1}{\sigma^\prime}}\Bigg(\int_{B_R}\frac{\operatorname{\exp}\big(\alpha \sigma(1+\eta)|v|^{\frac{Q}{Q-1}}\big)}{r(\xi)^{\vartheta \sigma}}\,\mathrm{d}\xi \Bigg)^{\frac{1}{\sigma}} \\
    &<+\infty.
\end{aligned}
\end{equation}
In view of \eqref{eq3.5} and \eqref{eq3.7}, we conclude the proof.
\end{proof}
\section{Sharp singular concentration-compactness principle: Proof of Theorem \ref{thm1.2}}\label{sec4}
This section is devoted to proving a concentration-compactness principle of the singular Moser-Trudinger inequality discussed in the previous section (see Theorem \ref{thm1.1}), which is also based on a rearrangement-free approach. For a detailed study, refer to \cite{Chen-Lu-Zhu-2021, Li-Lu-Zhu-2018, Li-Lu-Zhu-2021}. Moreover, we can now start to prove the following theorem. 
\begin{proof}[\bf{Proof of Theorem \ref{thm1.2}}.]
To prove Theorem \ref{thm1.2}, first notice that $\Phi_{\alpha,j_0}(|u|)=\Phi_{\alpha,j_0}(u)$, so without loss of generality, we can assume that $\{u_k\}_k\subset W$ and its weak limit $u$ are nonnegative functions. Now, by using the weak lower semicontinuity of the norm in $W$, one has
$$\|u\|\leq\liminf_{k\to\infty}\|u_k\|=1. $$
In conclusion, the proof of Theorem \ref{thm1.2} is divided into two cases as follows:

\text{Case{-}1:} In this case, we assume that $\|u\|=1$. Notice that $W$ is a reflexive Banach space, $u_k\rightharpoonup u$ in $W$ and $\|u_k\|\to \|u\|$ in $\mathbb{R}$ as $k\to\infty$. Thus, we obtain from \cite[Corollary A.2]{Autuori-Pucci-2013} that $u_k\to u$ in $W$ as $k\to\infty$. In addition, by employing a similar idea as in the proof of Case{-}$2$ \cite[Theorem 1.5]{Shen-Radulescu-2024}, up to a subsequence of $\{u_k\}_k\subset W$ still denoted by itself, there exists some $v\in W$ such that $|u_k(\xi)|\leq v(\xi)$ for a.e. $\xi\in\mathbb{H}^n$. Now, by using Corollary \ref{cor3.1} and the Lebesgue dominated convergence theorem that
\begin{align}\label{eq3.8}
\lim_{k\to\infty} \int_{\mathbb{H}^n}\frac{\Phi_{\zeta\alpha_{\vartheta,Q},j_0}(u_k)}{r(\xi)^\vartheta}\,\mathrm{d}\xi=\int_{\mathbb{H}^n}\frac{\Phi_{\zeta\alpha_{\vartheta,Q},j_0}(u)}{r(\xi)^\vartheta}\,\mathrm{d}\xi<+\infty.   
\end{align}
Moreover, again by Corollary \ref{cor3.1} and \eqref{eq3.8}, one can easily conclude that \eqref{eq1.13} holds.

\text{Case{-}2:} In this case, we assume that $0<\|u\|<1$. Next, we shall verify this situation by the method of contradiction. Assume that there exists some $0<\zeta<Z_Q(u)$ such that
\begin{align*}
\sup_{k\in\mathbb{N}}\int_{\mathbb{H}^n}\frac{\Phi_{\zeta\alpha_{\vartheta,Q},j_0}(u_k)}{r(\xi)^\vartheta}\,\mathrm{d}\xi=+\infty.
\end{align*}
Let $L>0$ be a constant and $\mathbb{U}:=\{\xi\in\mathbb{H}^n\colon~u_k(\xi)\geq L\}$. It follows that $\mathbb{U}$ is a bounded set. Indeed, by the given hypothesis and Lemma \ref{lem2.1}, one has
$$ |\mathbb{U}|= \int_\mathbb{U} \,\mathrm{d}\xi\leq L^{-p^\ast}\int_\mathbb{U} |u_k|^{p^\ast} \,\mathrm{d}\xi\lesssim \|u_k\|^{p^\ast}\lesssim 1.$$
We can write
\begin{align*}
 \int_{\mathbb{H}^n}\frac{\Phi_{\zeta\alpha_{\vartheta,Q},j_0}(u_k)}{r(\xi)^\vartheta}\,\mathrm{d}\xi=\bigg(\int_{\mathbb{H}^n\setminus\mathbb{U}}+\int_{\mathbb{U}}\bigg)\frac{\Phi_{\zeta\alpha_{\vartheta,Q},j_0}(u_k)}{r(\xi)^\vartheta}\,\mathrm{d}\xi:=J_1+J_2.   
\end{align*}
In virtue of $\frac{Qj_0}{Q-1}\geq p^\ast$ and Lemma \ref{lem2.1}, we have
\begin{align*}
J_1&=\int_{\mathbb{H}^n\setminus \mathbb{U}}\sum_{j=j_0}^{\infty} \frac{\big(\zeta\alpha_{\vartheta,Q} L^{Q^\prime}\big)^j}{j!}\bigg|\frac{u_k}{L}\bigg|^{\frac{Qj}{Q-1}}\frac{1}{r(\xi)^\vartheta} \,\mathrm{d}\xi\leq \sum_{j=j_0}^{\infty} \frac{\big(\zeta\alpha_{\vartheta,Q} L^{Q^\prime}\big)^j}{j!}\int_{\mathbb{H}^n\setminus \mathbb{U}}\bigg|\frac{u_k}{L}\bigg|^{\frac{Qj_0}{Q-1}}\frac{1}{r(\xi)^\vartheta} \,\mathrm{d}\xi\\
&\lesssim \int_{\mathbb{H}^n\setminus \mathbb{U}}\frac{|u_k|^{\frac{Qj_0}{Q-1}}}{r(\xi)^\vartheta} \,\mathrm{d}\xi\lesssim\int_{\{\xi\in\mathbb{H}^n\colon~r(\xi)<1\}}\frac{1}{r(\xi)^\vartheta} \,\mathrm{d}\xi+\int_{\{\xi\in\mathbb{H}^n\colon~r(\xi)\geq1\}}|u_k|^{\frac{Qj_0}{Q-1}}\,\mathrm{d}\xi~\lesssim 1.
\end{align*}
Define $v_k(\xi):=u_k(\xi)-L$ in $\mathbb{U}$, then we obtain from \eqref{eq3.6} that
\begin{align*}
  J_2&\leq \int_{\mathbb{U}}\frac{\operatorname{\exp}\big(\zeta\alpha_{\vartheta,Q}|u_k|^{\frac{Q}{Q-1}}\big)}{r(\xi)^\vartheta}\,\mathrm{d}\xi\leq\int_{\mathbb{U}}\frac{\operatorname{\exp}\big(\zeta\alpha_{\vartheta,Q}(1+\eta)|v_k|^{\frac{Q}{Q-1}}+\zeta\alpha_{\vartheta,Q}C_\eta L^{\frac{Q}{Q-1}}\big)}{r(\xi)^\vartheta}\,\mathrm{d}\xi\\
    &\lesssim \int_{\mathbb{U}}\frac{\operatorname{\exp}\big(\zeta\alpha_{\vartheta,Q}(1+\eta)|v_k|^{\frac{Q}{Q-1}}\big)}{r(\xi)^\vartheta}\,\mathrm{d}\xi.   
\end{align*}
From the above estimates, we obtain for $\eta>0$ sufficiently small that
\begin{align}\label{eq3.9}
\sup_{k\in\mathbb{N}}\int_{\mathbb{U}}\frac{\operatorname{\exp}\big(\bar{\zeta}\alpha_{\vartheta,Q}|v_k|^{\frac{Q}{Q-1}}\big)}{r(\xi)^\vartheta}\,\mathrm{d}\xi=+\infty,\quad\text{where}\quad \bar{\zeta}:=\zeta(1+\eta)<Z_Q(u).    
\end{align}
For $L>0$ and $v\geq 0$, we define the truncation functions $G^L$ and $G_L$ as follows:
$$ G^L(v):=\min\{v,L\}\quad\text{and}\quad G_L(v):=v-G^L(v).$$
Notice that $\|G^L(u)\|\to\|u\|$ as $L\to\infty$. It follows that for $\bar{\zeta}\in(0,Z_Q(u))$, we can assume $L>0$ sufficiently large enough such that
\begin{align}\label{eq3.10}
    \frac{1-\| u\|^Q}{1-\|G^L(u)\|^Q}>\bigg(\frac{\bar{\zeta}}{Z_Q(u)}\bigg)^{Q-1}.
\end{align}
To finish the proof, it suffices to show that
\begin{align}\label{eq3.11}
    \limsup_{k\to\infty}\|\bar{\zeta}^{~\frac{Q-1}{Q}} D_H v_k\|_{L^Q(\mathbb{U})}<1.
\end{align}
Indeed, if not, up to a subsequence still denoted by itself, we can assume 
$$ \| D_H v_k\|^Q_{L^Q(\mathbb{U})}=\| D_H G_L( u_k)\|^Q_Q\geq \bar{\zeta}^{~-(Q-1)}+o_k(1)\quad\text{as}\quad k\to\infty.  $$
This yields at once that
\begin{equation}\label{eq3.12}
\begin{aligned}
&\bar{\zeta}^{~-(Q-1)}+\| G^L(u_k)\|^Q+o_k(1)\\ 
&\qquad=\bar{\zeta}^{~-(Q-1)}+\|D_H G^L(u_k)\|^Q_p+\|D_H G^L(u_k)\|^Q_Q+o_k(1)\\ 
&\qquad\leq \|D_H G^L(u_k)\|^Q_p+\|D_H G^L(u_k)\|^Q_Q+\| D_H G_L( u_k)\|^Q_Q\\ 
&\qquad=\|D_H u_k\|^Q_{L^p(\mathbb{H}^n\setminus \mathbb{U})}+\|D_H u_k\|^Q_{L^Q(\mathbb{H}^n\setminus \mathbb{U})}+\|D_H u_k\|^Q_{L^Q(\mathbb{U})}\\ 
&\qquad\leq \|D_H u_k\|^Q_p+\|D_H u_k\|^Q_Q=\|u_k\|^Q=1\quad\text{as}\quad k\to\infty.
\end{aligned}
\end{equation}
Let us first fix $L>0$. It is easy to see that the sequence $\{G^L(u_k)\}_k\subset W$ is bounded in $W$. Thus, up to a subsequence still denoted by itself, we can easily obtain $G^L(u_k)\rightharpoonup G^L(u)$ in $W$ and $G^L(u_k)\to G^L(u)$ a.e. in $\mathbb{H}^n$ as $k\to\infty$. In addition, due to the weakly lower semicontinuity of the norm in $W$, \eqref{eq3.10} and \eqref{eq3.12},  one has
\begin{align*}
    \bar{\zeta} &\geq \Big(1-\liminf_{k\to\infty}\| G^L(u_k)\|^Q\Big)^{-\frac{1}{Q-1}}\geq \Big(1-\| G^L(u)\|^Q\Big)^{-\frac{1}{Q-1}}\\ 
    &>\frac{\bar{\zeta}}{Z_Q(u)} \big(1-\| u\|^Q\big)^{-\frac{1}{Q-1}}=\bar{\zeta}, \end{align*}
which is a contradiction. It follows that \eqref{eq3.11} holds. Hence, up to a subsequence still denoted by itself, we can assume that $$\bar{\zeta} \alpha_{\vartheta,Q}\|D_H v_k\|^{\frac{Q}{Q-1}}_{L^Q(\mathbb{U})}<\alpha_{\vartheta,Q},~~\forall~k\in\mathbb{N}.$$ In conclusion, we obtain from the Moser-Trudinger inequality \cite[Lemma 4.1]{Lam-Lu-Tang-2012} that
$$ \sup_{k\in\mathbb{N}}\int_{\mathbb{U}}\frac{\operatorname{\exp}\big(\bar{\zeta}\alpha_{\vartheta,Q}|v_k|^{\frac{Q}{Q-1}}\big)}{r(\xi)^\vartheta}\,\mathrm{d}\xi=\sup_{k\in\mathbb{N}}\int_{\mathbb{U}}\frac{\operatorname{\exp}\Big(\bar{\zeta}\alpha_{\vartheta,Q}\|D_H v_k\|^{\frac{Q}{Q-1}}_{L^Q(\mathbb{U})}|v_k/\|D_H v_k\|_{L^Q(\mathbb{U})}|^{\frac{Q}{Q-1}}\Big)}{r(\xi)^\vartheta}\,\mathrm{d}\xi<+\infty, $$
which contradicts \eqref{eq3.9}. This completes the proof in this situation.

Next, to check the validity of the sharpness of $Z_Q(u)$, it is enough to construct a sequence $\{u_k\}_k\subset W$ and a function $u\in W$ satisfying
\begin{align}\label{eq3.13}
    \|u_k\|=1,\quad u_k\rightharpoonup u\quad\text{in}\quad W\quad\text{as}\quad k\to\infty,\quad\|u\|=\varrho<1
\end{align}
but 
\begin{align*}
   \lim_{k\to\infty} \int_{\mathbb{H}^n}\frac{\Phi_{\zeta\alpha_{\vartheta,Q},j_0}(u_k)}{r(\xi)^\vartheta}\,\mathrm{d}\xi=+\infty,~~\forall~ \zeta>Z_Q(u).
\end{align*}
Let $s>0$ be fixed and we define a sequence $\{\psi_k\}_k\subset W$ as follows:
\begin{equation}\label{eq3.14}
  \psi_k(\xi):=\begin{cases}
   \big( k \alpha_Q^{-1}\big)^{\frac{Q-1}{Q}}\quad\quad\quad\quad\quad\quad\text{if}\quad r(\xi)\in\big[0,s\operatorname{\exp}(-kQ^{-1})\big],\\
   \big( k \alpha_Q^{-1}\big)^{\frac{Q-1}{Q}}\frac{Q}{k} \ln\big(\frac{s}{r(\xi)}\big)\quad\quad~~\text{if}\quad r(\xi)\in\big[s\operatorname{\exp}(-kQ^{-1}),s\big],\\
    0 \quad\quad\quad\quad\quad\quad\quad\quad\quad\quad\quad~~\text{if}\quad r(\xi)\in[s,+\infty),
    \end{cases}
\end{equation}
where $\alpha_Q=Qc_Q^{\frac{1}{Q-1}}$. By arguing similarly to the proof of Theorem \ref{thm1.1}, we derive 
$$ |D_H \psi_k|_H=\big( k \alpha_Q^{-1}\big)^{\frac{Q-1}{Q}}\frac{Q}{k} \frac{|z|}{r(\xi)^2}\chi _{B_s\setminus B_{s\operatorname{\exp}(-kQ^{-1})}}(\xi).$$
It follows at once that
\begin{align*}
    \int_{\mathbb{H}^n}|D_H \psi_k|^p_H\,\mathrm{d}\xi&=\int_\Sigma\int_{s\operatorname{\exp}(-kQ^{-1})}^{s}\bigg(\big( k \alpha_Q^{-1}\big)^{\frac{Q-1}{Q}}\frac{Q}{k}\frac{|z^\ast|}{\eta}\bigg)^p\eta^{Q-1}\,\mathrm{d}\eta\,\mathrm{d}\mu(\xi^\ast)\\
    &=\alpha^{-\frac{(Q-1)p}{Q}}_Q Q^p k^{-\frac{p}{Q}}c_p\int_{s\operatorname{\exp}(-kQ^{-1})}^{s}\eta^{Q-p-1}\,\mathrm{d}\eta\\
    &=\frac{\alpha^{-\frac{(Q-1)p}{Q}}_Q Q^p k^{-\frac{p}{Q}}c_p s^{Q-p}}{Q-p}\Big(1-\operatorname{\exp}\big(-k(Q-p)Q^{-1}\big)\Big)
\end{align*}
and 
\begin{align*}
    \int_{\mathbb{H}^n}|D_H \psi_k|^Q_H\,\mathrm{d}\xi&=\int_\Sigma\int_{s\operatorname{\exp}(-kQ^{-1})}^{s}\bigg(\big( k \alpha_Q^{-1}\big)^{\frac{Q-1}{Q}}\frac{Q}{k}\frac{|z^\ast|}{\eta}\bigg)^Q\eta^{Q-1}\,\mathrm{d}\eta\,\mathrm{d}\mu(\xi^\ast)\\
    &=\alpha^{-(Q-1)}_Q  Q^Q k^{-1}c_Q\int_{s\operatorname{\exp}(-kQ^{-1})}^{s}\eta^{-1}\,\mathrm{d}\eta\\
    &=\alpha^{-(Q-1)}_Q  Q^Q k^{-1} c_Q\ln\bigg(\frac{s}{s\operatorname{\exp}(-kQ^{-1})}\bigg)=\alpha^{-(Q-1)}_Q  Q^{Q-1}c_Q=1.
\end{align*}
By standard arguments, one has
$$\psi_k\rightharpoonup 0\quad\text{in}\quad W,\quad \|D_H \psi_k\|^p_p\to 0\quad\text{as}\quad k\to\infty\quad\text{and}\quad \|D_H \psi_k\|^Q_Q=1.$$
Let $B>0$ be a constant and $S=3s$. Define a function $u\colon \mathbb{H}^n\to\mathbb{R}$ as follows:
$$ u(\xi):=\begin{cases}
    B\quad\quad\quad\quad\quad\quad\quad\quad\text{if}\quad r(\xi)\in[0,\frac{2S}{3}],\\
    3B\Big(1-\frac{r(\xi)}{S}\Big)\quad\quad\quad\text{if}\quad r(\xi)\in[\frac{2S}{3},S],\\
    0\quad\quad\quad\quad\quad\quad\quad\quad\text{if}\quad r(\xi)\in[S,+\infty).
\end{cases} $$
Direct computations yield that
$$|D_H u|_H=\frac{3B}{S}\frac{|z|}{r(\xi)}\chi_{B_S\setminus B_{\frac{2S}{3}}}(\xi). $$
In conclusion, we deduce that
\begin{align*}
    \|u\|^Q &=\bigg(\int_{\mathbb{H}^n}|D_H u|^p_H\,\mathrm{d}\xi\bigg)^{\frac{Q}{p}}+\int_{\mathbb{H}^n}|D_H u|^Q_H\,\mathrm{d}\xi\\
    &=\bigg(\int_\Sigma\int_{\frac{2S}{3}}^{S}\bigg(\frac{3B}{S}|z^\ast|\bigg)^p\eta^{Q-1}\,\mathrm{d}\eta\,\mathrm{d}\mu(\xi^\ast)\bigg)^{\frac{Q}{p}}+\int_\Sigma\int_{\frac{2S}{3}}^{S}\bigg(\frac{3B}{S}|z^\ast|\bigg)^Q\eta^{Q-1}\,\mathrm{d}\eta\,\mathrm{d}\mu(\xi^\ast)\\
    &=\bigg(\frac{3B}{S}\bigg)^Q\Bigg[\bigg(c_p\int_{\frac{2S}{3}}^{S}\eta^{Q-1}\,\mathrm{d}\eta\bigg)^{\frac{Q}{p}}+c_Q\int_{\frac{2S}{3}}^{S}\eta^{Q-1}\,\mathrm{d}\eta\Bigg]\\
    &=B^Q\Bigg[\Bigg(\frac{S^{Q-p}\big(3^Q-2^Q\big)c_p}{3^{Q-p}Q}\Bigg)^{\frac{Q}{p}}+\frac{\big(3^Q-2^Q\big)c_Q}{Q}\Bigg].
\end{align*}
Thus, we can find $B>0$ in such a way that $\|u\|=\varrho<1$. Let us define a sequence $\{w_k\}_k\subset W$ as follows:
$$w_k(\xi):=u(\xi)+\big(1-\varrho^Q\big)^{\frac{1}{Q}}\psi_k(\xi)\in W,~\forall~\xi\in\mathbb{H}^n. $$
Notice that $B_s\cap (\mathbb{H}^n\setminus B_{\frac{2S}{3}})=\emptyset$, we thank to $S=3s$. It follows that $D_H u$ and $D_H \psi_k$ have disjoint supports. By direct computations, we have
$$ \|D_H w_k\|^Q_Q=\|D_H u\|^Q_Q+\big(1-\varrho^Q\big)\quad\text{and}\quad \|D_H w_k\|^p_p=\|D_H u\|^p_p+\tau_k,$$
where $\tau_k\to 0$ as $k\to\infty$. This shows that $\|w_k\|^Q\to 1$ as $k\to\infty$. {Hence}, without loss of generality, we can assume that $\|w_k\|^Q=1+\sigma_k$, where $\sigma_k\to 0$ as $k\to\infty$. Next, we choose a sequence $\{u_k\}_k\subset W$ as follows:
$$ u_k(\xi):=\frac{w_k(\xi)}{(1+\sigma_k)^{\frac{1}{Q}}},~\forall~\xi\in\mathbb{H}^n. $$
This implies the validity of \eqref{eq3.13}. Moreover, we assume that $\zeta>Z_Q(u)$ and hence there exists $\varepsilon_0>0$ such that $\zeta=(1+\varepsilon_0)Z_Q(u)$. Thus, one can easily see that
\begin{align*}
\int_{\mathbb{H}^n}\frac{\Phi_{\zeta\alpha_{\vartheta,Q},j_0}(u_k)}{r(\xi)^\vartheta}&\,\mathrm{d}\xi\geq \int_{\big\{\xi\in \mathbb{H}^n\colon ~r(\xi)<s\operatorname{\exp}(-kQ^{-1})\big\}}\frac{\Phi_{\zeta\alpha_{\vartheta,Q},j_0}\Big((1+\sigma_k)^{-\frac{1}{Q}}\big(u+(1-\varrho^Q)^{\frac{1}{Q}}\psi_k\big)\Big)}{r(\xi)^\vartheta}\,\mathrm{d}\xi\\
  &=\int_\Sigma\int_{0}^{s\operatorname{\exp}(-kQ^{-1})}\frac{\Phi_{\zeta\alpha_{\vartheta,Q},j_0}\Big((1+\sigma_k)^{-\frac{1}{Q}}\Big(B+(1-\varrho^Q)^{\frac{1}{Q}}\big( k \alpha_Q^{-1}\big)^{\frac{Q-1}{Q}}\Big)\Big)}{\eta^{\vartheta-Q+1}}\,\mathrm{d}\eta\,\mathrm{d}\mu(\xi^\ast)\\
  &\geq \bigg(\int_\Sigma \mathrm{d}\mu(\xi^\ast)\bigg)\frac{s^{Q-\vartheta}\operatorname{\exp}\big(-k\big(1-\frac{\vartheta}{Q}\big)\big)}{Q-\vartheta}\Phi_{\zeta\alpha_{\vartheta,Q},j_0}\Big((1+\sigma_k)^{-\frac{1}{Q}}(1-\varrho^Q)^{\frac{1}{Q}}\big( k \alpha_Q^{-1}\big)^{\frac{Q-1}{Q}}\Big)\\
  & \gtrsim \bigg(\int_\Sigma \mathrm{d}\mu(\xi^\ast)\bigg)\frac{s^{Q-\vartheta}\operatorname{\exp}\big(-k\big(1-\frac{\vartheta}{Q}\big)\big)}{Q-\vartheta}\operatorname{\exp}\Bigg[\zeta\alpha_{\vartheta,Q}\Bigg|\frac{(1-\varrho^Q)^{\frac{1}{Q}}\big( k \alpha_Q^{-1}\big)^{\frac{Q-1}{Q}}}{(1+\sigma_k)^{\frac{1}{Q}}}\Bigg|^{\frac{Q}{Q-1}}\Bigg]\\
  &=\bigg(\int_\Sigma \mathrm{d}\mu(\xi^\ast)\bigg)\frac{s^{Q-\vartheta}}{Q-\vartheta}\operatorname{\exp}\Bigg[\bigg(1-\frac{\vartheta}{Q}\bigg) k\Bigg(\frac{1+\varepsilon_0}{(1+\sigma_k)^{\frac{1}{Q-1}}}-1\Bigg)\Bigg]\to+\infty\quad\text{as}\quad k\to\infty.
\end{align*}
This completes the proof of the sharpness of $Z_Q(u)$ and thus\textcolor{blue}{,} the theorem is well established.
\end{proof}
\section{Asymptotic behavior of  sharp subcritical singular Moser-Trudinger type inequality: Proof of
 Theorem \ref{thm5.1}}\label{sec5}
 In this section, we shall discuss the asymptotic behavior of the supremum $\text{SCMTI}(\zeta,\vartheta)$ for the sharp subcritical singular Moser-Trudinger inequality. To prove this, we first establish the following lemma. 
\begin{lemma}\label{lem4.1}
The following result holds:
$$\text{SCMTI}(\zeta,\vartheta)=\sup_{u\in W,~\|D_H u\|_{Q}\leq 1,~\|D_H u\|_p=1}\int_{\mathbb{H}^n}\frac{\Phi_{\zeta(1-\frac{\vartheta}{Q}),j_0}(u)}{r(\xi)^\vartheta}\,\mathrm{d}\xi.$$    
\end{lemma}
\begin{proof}
Let us assume that $u\in W$ and $\|D_H u\|_Q\leq 1$. Define the dilation function $\delta_\theta\colon\mathbb{H}^n\to \mathbb{H}^n$ by $$ \delta_\theta(\xi):=(\theta z, \theta^2t),~\forall~\xi=(z,t)\in \mathbb{H}^n~~\text{and}~~\theta>0.$$ In addition, we define a function $\psi\colon\mathbb{H}^n\to\mathbb{R}$ by
$$ \psi(\xi):=u(\delta_\theta(\xi))\quad\text{with}\quad \theta:=\|D_H u\|_p^{\frac{p^\ast}{Q}}>0. $$
Since the Jacobian determinant of $\delta_\theta$ coincides with $\theta^Q$, we have
$$ D_H \psi(\xi)=\theta D_H u(\delta_\theta(\xi)),~\|D_H \psi\|_Q=\|D_H u\|_Q\leq 1\quad\text{and}\quad \|D_H \psi\|_p=\theta^{-\frac{Q}{p^\ast}}\|D_H u\|_p =1.$$
Now, by using the fact that $r(\delta_\theta(\xi))=\theta r(\xi)$, one has
\begin{align*}
 \int_{\mathbb{H}^n}\frac{\Phi_{\zeta(1-\frac{\vartheta}{Q}),j_0}(\psi)}{r(\xi)^\vartheta}\,\mathrm{d}\xi & =\theta ^\vartheta\int_{\mathbb{H}^n}\frac{\Phi_{\zeta(1-\frac{\vartheta}{Q}),j_0}(u(\delta_\theta(\xi))}{r(\delta_\theta(\xi))^\vartheta}\,\mathrm{d}\xi=\theta ^{\vartheta-Q}\int_{\mathbb{H}^n}\frac{\Phi_{\zeta(1-\frac{\vartheta}{Q}),j_0}(u(\xi))}{r(\xi)^\vartheta}\,\mathrm{d}\xi\\
 &=\|D_H u\|_p^{-p^\ast(1-\frac{\vartheta}{Q})}\int_{\mathbb{H}^n}\frac{\Phi_{\zeta(1-\frac{\vartheta}{Q}),j_0}(u)}{r(\xi)^\vartheta}\,\mathrm{d}\xi.
\end{align*}
This finishes the proof.
\end{proof}
In virtue of the above lemma, we can always assume that $\|D_H u\|_p=1$ in the sharp subcritical singular Moser-Trudinger inequality. Now, we can prove Theorem \ref{thm5.1} as follows.
\begin{proof}[\bf{Proof of Theorem \ref{thm5.1}}.]
It follows from the density argument of $C_0^\infty(\mathbb{H}^n)$ in $W$ and Lemma \ref{lem4.1} that without loss of generality, we can choose $u\in C_0^\infty(\mathbb{H}^n)\setminus\{0\}$, $u\geq 0$, $\|D_Hu\|_Q\leq 1$ and $\|D_Hu\|_p= 1$. Denote
$$ \mathbb{U}:=\Bigg\{\xi\in\mathbb{H}^n\colon~u(\xi)>\bigg(1-\Big(\frac{\zeta}{\alpha_Q}\Big)^{Q-1}\bigg)^{\frac{1}{Q}}\bigg\}. $$
By using the embedding $S^{1,p}(\mathbb{H}^n)\hookrightarrow L^{p^\ast}(\mathbb{H}^n)$, one sees that $\mathbb{U}$ is a bounded set. In addition, we have
\begin{align}\label{eq4.1}
 |\mathbb{U}|= \int_\mathbb{U} \,\mathrm{d}\xi\leq \bigg(1-\Big(\frac{\zeta}{\alpha_Q}\Big)^{Q-1}\bigg)^{-\frac{p^\ast}{Q}}\|u\|_{p^\ast} ^{p^\ast}\leq \bigg(1-\Big(\frac{\zeta}{\alpha_Q}\Big)^{Q-1}\bigg)^{-\frac{p^\ast}{Q}} C(p,Q),   
\end{align}
where $C(p,Q)>0$ is a suitable constant depending on $p$ and $Q$. Notice that one can write
\begin{align*}
 \int_{\mathbb{H}^n}\frac{\Phi_{\zeta(1-\frac{\vartheta}{Q}),j_0}(u)}{r(\xi)^\vartheta}\,\mathrm{d}\xi=\bigg(\int_{\mathbb{H}^n\setminus \mathbb{U}}+\int_{\mathbb{U}}\bigg)\frac{\Phi_{\zeta(1-\frac{\vartheta}{Q}),j_0}(u)}{r(\xi)^\vartheta}\,\mathrm{d}\xi:=J_1+J_2.   
\end{align*}
Recall that $\mathbb{H}^n\setminus \mathbb{U}\subset \{\xi\in \mathbb{H}^n:~u(\xi)\leq 1\}$. Now, again by taking advantage of  the embedding $S^{1,p}(\mathbb{H}^n)\hookrightarrow L^{p^\ast}(\mathbb{H}^n)$ and $\|D_H u\|_p= 1$, we obtain
\begin{equation}\label{eq4.2}
\begin{aligned}
J_1 & \leq \int_{\{\xi\in \mathbb{H}^n\colon~ u(\xi)\leq 1\}}\frac{1}{r(\xi)^\vartheta}\sum_{j=j_0}^\infty \frac{\big(\zeta\big(1-\frac{\vartheta}{Q}\big)\big)^j}{j!}|u|^{\frac{Qj}{Q-1}}\,\mathrm{d}\xi \leq \int_{\{x\in \mathbb{H}^n\colon~ u(\xi)\leq 1\}}\frac{1}{r(\xi)^\vartheta}\sum_{j=j_0}^\infty \frac{\big(\zeta\big(1-\frac{\vartheta}{Q}\big)\big)^j}{j!}|u|^{p^\ast}\,\mathrm{d}\xi \\ 
 & \leq \operatorname{\exp}\bigg(\zeta\bigg(1-\frac{\vartheta}{Q}\bigg)\bigg) {\Bigg(}\int_{\{\xi\in \mathbb{H}^n\colon~ u(\xi)\leq 1,~r(\xi)<1\}}\frac{|u|^{p^\ast}}{r(\xi)^\vartheta}\,\mathrm{d}\xi+  \int_{\{\xi\in \mathbb{H}^n\colon~u(\xi)\leq1,~r(\xi)\geq1\}}\frac{|u|^{p^\ast}}{r(\xi)^\vartheta}\,\mathrm{d}\xi{\Bigg)}\\ 
& \leq \operatorname{\exp}\bigg(\zeta\bigg(1-\frac{\vartheta}{Q}\bigg)\bigg){\Bigg(} \int_{\{ \xi\in \mathbb{H}^n\colon~r(\xi)<1\}}\frac{1}{r(\xi)^\vartheta}\,\mathrm{d}\xi+ \int_{\{ \xi\in \mathbb{H}^n\colon~r(\xi)\geq1\}}|u|^{p^\ast}\,\mathrm{d}\xi{\Bigg)} \leq  C(\zeta,\vartheta,p,Q),
\end{aligned}
\end{equation}
where $C(\alpha,\vartheta,p,Q)>0$ is a constant depending on  $\zeta$, $\vartheta$, $p$ and $Q$. Let us define $\psi\colon \mathbb{U}\to\mathbb{R}$ as follows: $$\psi(\xi):=u(\xi)-\bigg(1-\Big(\frac{\zeta}{\alpha_Q}\Big)^{Q-1}\bigg)^{\frac{1}{Q}},~\forall~\xi\in \mathbb{U}.$$ 
By standard arguments, one sees that $\psi\in HW^{1,Q}_0(\mathbb{U})$ and $\|D_H \psi\|_{L^Q(\mathbb{U})}\leq 1$. In view of \eqref{eq3.6} with $\eta=\frac{\alpha_Q}{\zeta}-1$, we have
\begin{equation}\label{eq4.3}
\begin{aligned}
    |u(\xi)|^{\frac{Q}{Q-1}}&\leq \Bigg(|\psi(\xi)|+\bigg(1-\Big(\frac{\zeta}{\alpha_Q}\Big)^{Q-1}\bigg)^{\frac{1}{Q}}\Bigg)^{\frac{Q}{Q-1}}\leq (1+\eta)|\psi(\xi)|^{\frac{Q}{Q-1}}+C_\eta \bigg(1-\Big(\frac{\zeta}{\alpha_Q}\Big)^{Q-1}\bigg)^{\frac{1}{Q-1}}\\ 
    &=  \frac{\alpha_Q}{\zeta}|\psi(\xi)|^{\frac{Q}{Q-1}}+\bigg(1-\Big(\frac{\zeta}{\alpha_Q}\Big)^{Q-1}\bigg)^{-\frac{1}{Q-1}}\bigg(1-\Big(\frac{\zeta}{\alpha_Q}\Big)^{Q-1}\bigg)^{\frac{1}{Q-1}}\\
    &=\frac{\alpha_Q}{\zeta}|\psi(\xi)|^{\frac{Q}{Q-1}}+1,~\forall~\xi\in \mathbb{U}.
\end{aligned}
\end{equation}
Due to the Moser-Trudinger inequality \cite[Lemma 4.1]{Lam-Lu-Tang-2012}, \eqref{eq4.1} and \eqref{eq4.3}, we obtain
\begin{equation}\label{eq4.4}
\begin{aligned}
 J_2 &  \leq \int_\mathbb{U}\frac{\operatorname{\exp}\Big(\zeta\big(1-\frac{\vartheta}{Q}\big) |u|^{\frac{Q}{Q-1}}\Big)}{r(\xi)^\vartheta}\,\mathrm{d}\xi\leq \operatorname{\exp}\bigg(\zeta\bigg(1-\frac{\vartheta}{Q}\bigg)\bigg)\int_\mathbb{U}\frac{\operatorname{\exp}(\alpha_{\vartheta,Q} |\psi|^{\frac{Q}{Q-1}})}{r(\xi)^\vartheta}\,\mathrm{d}\xi \\ 
 &\leq \operatorname{\exp}\bigg(\zeta\bigg(1-\frac{\vartheta}{Q}\bigg)\bigg)C(\vartheta,Q)|\mathbb{U}|^{1-\frac{\vartheta}{Q}}\leq C(\vartheta,p,Q)\bigg(1-\Big(\frac{\zeta}{\alpha_Q}\Big)^{Q-1}\bigg)^{-\frac{p^\ast}{Q}\big(1-\frac{\vartheta}{Q}\big)},
\end{aligned}
\end{equation}
where $C(\vartheta,p,Q)>0$ is a suitable constant depending on $\vartheta$, $p$ and $Q$. {By} combining \eqref{eq4.2} and \eqref{eq4.4}, one sees that
$$\text{SCMTI}(\zeta,\vartheta)\leq \bigg(1-\Big(\frac{\zeta}{\alpha_Q}\Big)^{Q-1}\bigg)^{-\frac{p^\ast}{Q}\big(1-\frac{\vartheta}{Q}\big)}C(\vartheta,p,Q).   
$$

Next, we aim to prove the sharpness of $\alpha_Q$. In this purpose, take the same sequence $\{\psi_k\}_k\subset W$ that appeared in \eqref{eq3.14} with $s=1$. It follows that
$$ \|D_H\psi_k\|_Q=1\quad\text{and}\quad \|D_H\psi_k\|_p\to 0\quad\text{as}\quad k\to\infty.$$
Moreover, one has
\begin{align*}
  \int_{\mathbb{H}^n}\frac{\Phi_{\alpha_{\vartheta,Q},j_0}(\psi_k)}{r(\xi)^\vartheta}\,\mathrm{d}\xi&\geq \int_{\big\{\xi\in \mathbb{H}^n\colon~r(\xi)<\operatorname{\exp}(-kQ^{-1})\big\}}\frac{\Phi_{\alpha_{\vartheta,Q},j_0}(\psi_k)}{r(\xi)^\vartheta}\,\mathrm{d}\xi\\
  &=\int_\Sigma\int_{0}^{\operatorname{\exp}(-kQ^{-1})}\frac{\Phi_{\alpha_{\vartheta,Q},j_0}\Big(\big( k \alpha_Q^{-1}\big)^{\frac{Q-1}{Q}}\Big)}{\eta^{\vartheta-Q+1}}\,\mathrm{d}\eta\,\mathrm{d}\mu(\xi^\ast)\\
  &\gtrsim \bigg(\int_\Sigma \,\mathrm{d}\mu(\xi^\ast)\bigg)\frac{\operatorname{\exp}\big(-k\big(1-\frac{\vartheta}{Q}\big)\big)\operatorname{\exp}\big(\alpha_{\vartheta,Q}k\alpha_Q^{-1}\big)}{Q-\vartheta}=\bigg(\int_\Sigma \,\mathrm{d}\mu(\xi^\ast)\bigg)\frac{1}{Q-\vartheta}.  
\end{align*}
This yields that
$$ \|D_H \psi_k\|_p^{-p^\ast(1-\frac{\vartheta}{Q})}\int_{\mathbb{H}^n}\frac{\Phi_{\alpha_{\vartheta,Q},j_0}(\psi_k)}{r(\xi)^\vartheta}\,\mathrm{d}\xi\to+\infty\quad\text{as}\quad k\to\infty. $$
This shows that $\text{SCMTI}(\alpha_Q,\vartheta)=+\infty$ and thus\textcolor{blue}{,} the proof of the sharpness of $\alpha_Q$ is well verified. 

Now, we introduce the same Moser's function sequence $\{v_k\}_k\subset W$, which was used in \eqref{eq3.99}. It is easy to see that
$$ \|D_H v_k\|_Q=1\quad\text{and}\quad \|D_H v_k\|_p\leq (\ln k)^{-\frac{1}{Q}}\Big(\Pi_1+\Pi_2 k^{-(Q-p)}\Big)^{\frac{1}{p}},$$
where $\Pi_1$ and $\Pi_2$ are two positive constants. Hence, we have the following estimates
\begin{align*}
   \text{SCMTI}(\zeta,\vartheta)&\geq \cfrac{1}{\|D_H v_k\|_p^{p^\ast(1-\frac{\vartheta}{Q})}}\int_{\{\xi\in\mathbb{H}^n\colon~r(\xi)<sk^{-1}\}}\frac{\Phi_{\zeta(1-\frac{\vartheta}{Q}),j_0}(v_k)}{r(\xi)^\vartheta}\,\mathrm{d}\xi\\
   &\geq \cfrac{(\ln k)^{\frac{p^\ast}{Q}(1-\frac{\vartheta}{Q})}}{\Big(\Pi_1+\Pi_2 k^{-(Q-p)}\Big)^{\frac{p^\ast}{p}(1-\frac{\vartheta}{Q})}}\int_\Sigma\int_{0}^{sk^{-1}}\Phi_{\zeta(1-\frac{\vartheta}{Q}),j_0}\Big(c^{-\frac{1}{Q}}_Q(\ln k)^{\frac{Q-1}{Q}}\Big)\eta^{Q-\vartheta-1}\,\mathrm{d}\eta\,\mathrm{d}\mu(\xi^\ast)\\
   &\gtrsim \cfrac{(\ln k)^{\frac{p^\ast}{Q}(1-\frac{\vartheta}{Q})} \bigg(\displaystyle\int_\Sigma\,\mathrm{d}\mu(\xi^\ast)\bigg)(sk^{-1})^{Q-\vartheta}}{\Big(\Pi_1+\Pi_2 k^{-(Q-p)}\Big)^{\frac{p^\ast}{p}(1-\frac{\vartheta}{Q})}(Q-\vartheta)}~\operatorname{\exp}\Bigg[\zeta\Big(1-\frac{\vartheta}{Q}\Big)\bigg|c^{-\frac{1}{Q}}_Q(\ln k)^{\frac{Q-1}{Q}}\bigg|^{\frac{Q}{Q-1}}\Bigg]\\
   &\gtrsim \cfrac{(\ln k)^{\frac{p^\ast}{Q}(1-\frac{\vartheta}{Q})} }{\Big(\Pi_1+\Pi_2 k^{-(Q-p)}\Big)^{\frac{p^\ast}{p}(1-\frac{\vartheta}{Q})}}~\operatorname{\exp}\bigg((Q-\vartheta)\ln k\Big(\frac{\zeta}{\alpha_Q}-1\Big)\bigg).
\end{align*}
If we choose $k$ large enough, then we have
$$ \cfrac{1}{\Big(\Pi_1+\Pi_2 k^{-(Q-p)}\Big)^{\frac{p^\ast}{p}(1-\frac{\vartheta}{Q})}}\approx \cfrac{1}{\Pi_1^{\frac{p^\ast}{p}(1-\frac{\vartheta}{Q})}}. $$
Thus, we deduce from the above inequality that
\begin{align*}
   \text{SCMTI}(\zeta,\vartheta) \gtrsim (\ln k)^{\frac{p^\ast}{Q}(1-\frac{\vartheta}{Q})} ~\operatorname{\exp}\bigg((Q-\vartheta)\ln k\Big(\frac{\zeta}{\alpha_Q}-1\Big)\bigg).
\end{align*}
In addition, since $\zeta\underset{\approx}{<}\alpha_Q$, one can assume $k$ large enough such that
$$ \ln k\approx\frac{p}{1-\frac{\zeta}{\alpha_Q}}\quad\text{and}\quad \operatorname{\exp}\bigg((Q-\vartheta)\ln k\Big(\frac{\zeta}{\alpha_Q}-1\Big)\bigg)\approx\operatorname{\exp}\big(-p(Q-\vartheta)\big).$$
It follows that when $\zeta\underset{\approx}{<}\alpha_Q$, there holds
$$ \text{SCMTI}(\zeta,\vartheta) \gtrsim\frac{\operatorname{\exp}\big(-p(Q-\vartheta)\big)}{\bigg(1-\frac{\zeta}{\alpha_Q}\bigg)^{\frac{p^\ast}{Q}(1-\frac{\vartheta}{Q})}}\approx\frac{(Q-1)^{\frac{p^\ast}{Q}(1-\frac{\vartheta}{Q})}~\operatorname{\exp}\big(-p(Q-\vartheta)\big)}{\bigg(1-\Big(\frac{\zeta}{\alpha_Q}\Big)^{Q-1}\bigg)^{\frac{p^\ast}{Q}(1-\frac{\vartheta}{Q})}}, $$
we thank to
$$ \frac{1-\Big(\frac{\zeta}{\alpha_Q}\Big)^{Q-1}}{1-\frac{\zeta}{\alpha_Q}}\approx Q-1\quad\text{for}\quad \zeta\underset{\approx}{<}\alpha_Q.$$
Finally, we have
$$\text{SCMTI}(\zeta,\vartheta)\geq \bigg(1-\Big(\frac{\zeta}{\alpha_Q}\Big)^{Q-1}\bigg)^{-\frac{p^\ast}{Q}\big(1-\frac{\vartheta}{Q}\big)}c(\vartheta,p,Q), $$
where $c(\vartheta,p,Q)>0$ is a constant depending on $\vartheta$, $p$ and $Q$. Thus, the theorem is well established.
\end{proof}
\section{Equivalence of sharp critical and subcritical singular Moser-Trudinger type inequalities: Proof of
 Theorem \ref{thm5.22}}\label{sec6}
In this section, our goal is to establish the relationship between the suprema for the sharp critical and subcritical singular Moser-Trudinger inequalities. Initially, we prove the following influential lemma.
 \begin{lemma}\label{lem4.2}
The sharp subcritical singular Moser-Trudinger type inequality in $W$ follows from the sharp critical singular Moser-Trudinger type inequality in $W$. For instance, if $\text{CMTI}_{a,b}(\vartheta)<+\infty$, then one has $\text{SCMTI}(\zeta,\vartheta)<+\infty$. Furthermore, it holds:
$$\text{SCMTI}(\zeta,\vartheta)\leq\Bigg(\cfrac{\big(\frac{\zeta}{\alpha_Q}\big)^{(\frac{Q-1}{Q})b}}{1-\big(\frac{\zeta}{\alpha_Q}\big)^{(\frac{Q-1}{Q})a}}\Bigg)^{\frac{p^\ast}{b}\big(1-\frac{\vartheta}{Q}\big)}\text{CMTI}_{a,b}(\vartheta). $$
Notably, we have
$$\text{SCMTI}(\zeta,\vartheta)\leq\Bigg(\cfrac{\big(\frac{\zeta}{\alpha_Q}\big)^{Q-1}}{1-\big(\frac{\zeta}{\alpha_Q}\big)^{Q-1}}\Bigg)^{\frac{p^\ast}{Q}\big(1-\frac{\vartheta}{Q}\big)}\text{CMTI}_{Q,Q}(\vartheta). $$
\end{lemma}
\begin{proof}
Suppose $u\in W$, $\|D_H u\|_Q\leq 1$ and $\|D_H u\|_p= 1$. By using the similar notations as in Lemma \ref{lem4.1}, we define a function  $\psi\colon\mathbb{H}^n\to\mathbb{R}$ by
   $$\psi(\xi):=\bigg(\frac{\zeta}{\alpha_Q}\bigg)^{\frac{Q-1}{Q}}u(\delta_\theta(\xi))\quad\text{and}\quad \theta:=\Bigg(\cfrac{\big(\frac{\zeta}{\alpha_Q}\big)^{(\frac{Q-1}{Q})b}}{1-\big(\frac{\zeta}{\alpha_Q}\big)^{(\frac{Q-1}{Q})a}}\Bigg)^{\frac{p^\ast}{Qb}}. $$
It follows that
$$ D_H\psi(\xi)=\bigg(\frac{\zeta}{\alpha_Q}\bigg)^{\frac{Q-1}{Q}}\theta D_Hu(\delta_\theta(\xi)),~\|D_H \psi\|_Q^a=\bigg(\frac{\zeta}{\alpha_Q}\bigg)^{\big(\frac{Q-1}{Q}\big)a}\|D_H u\|^a_Q\leq \bigg(\frac{\zeta}{\alpha_Q}\bigg)^{\big(\frac{Q-1}{Q}\big)a}$$
and $$\|D_H \psi\|_p^b=1- \bigg(\frac{\zeta}{\alpha_Q}\bigg)^{\big(\frac{Q-1}{Q}\big)a}. $$
From the above estimates, we obtain $\|D_H \psi\|_Q^a+\|D_H \psi\|_p^b\leq 1$. Moreover, one can easily deduce from the definition of $\text{CMTI}_{a,b}(\vartheta)$ that
\begin{align*}
  \int_{\mathbb{H}^n}\frac{\Phi_{\zeta(1-\frac{\vartheta}{Q}),j_0}(u)}{r(\xi)^\vartheta}\,\mathrm{d}\xi & =\int_{\mathbb{H}^n}\frac{\Phi_{\zeta(1-\frac{\vartheta}{Q}),j_0}\big(u(\delta_\theta(\xi))\big)}{r(\delta_\theta(\xi))^\vartheta}\,\mathrm{d}(\delta_\theta(\xi))=\theta ^{Q-\vartheta}\int_{\mathbb{H}^n}\frac{\Phi_{\alpha_{\vartheta,Q},j_0}(\psi)}{r(\xi)^\vartheta}\,\mathrm{d}\xi\\
 &\leq\Bigg(\cfrac{\big(\frac{\zeta}{\alpha_Q}\big)^{(\frac{Q-1}{Q})b}}{1-\big(\frac{\zeta}{\alpha_Q}\big)^{(\frac{Q-1}{Q})a}}\Bigg)^{\frac{p^\ast}{b}\big(1-\frac{\vartheta}{Q}\big)}\text{CMTI}_{a,b}(\vartheta).  
\end{align*}
This completes the proof. 
\end{proof}
Now, we can prove Theorem \ref{thm5.22} as follows.
\begin{proof}[\bf{Proof of Theorem \ref{thm5.22}}.]
Suppose $b\in(0,Q]$ and $u\in W\setminus\{0\}$ such that $\|D_H u\|^a_{Q}+\|D_H u\|^b_p\leq 1$. Set
$$\|D_H u\|_Q:=\sigma\in(0,1)\quad\text{and}\quad\|D_H u\|_p\leq(1-\sigma^a)^{\frac{1}{b}} .$$
Now, we fix $\sigma\in \big(2^{-\frac{Q}{p^\ast}},1\big)$ and define a function $\psi\colon\mathbb{H}^n\to\mathbb{R}$ by
$$\psi(\xi):=\sigma^{-1}u(\delta_\theta(\xi))\quad\text{with}\quad\theta:=\Big(\sigma^{-1}(1-\sigma^a)^{\frac{1}{b}}\Big)^{\frac{p^\ast}{Q}}>0. $$
Direct calculations yield that
$$ D_H \psi(\xi)=\sigma^{-1}\theta D_H u(\delta_\theta(\xi)),~\|D_H \psi\|_Q= 1\quad\text{and}\quad \|D_H \psi\|_p\leq\sigma^{-1}\theta^{-\frac{Q}{p^\ast}}(1-\sigma^a)^{\frac{1}{b}}=1.  $$
In virtue of Theorem \ref{thm5.1} and $b\leq Q$, we have the following estimates
\begin{align*}
  \int_{\mathbb{H}^n}\frac{\Phi_{\alpha_{\vartheta,Q},j_0}(u)}{r(\xi)^\vartheta}\,\mathrm{d}\xi & =\int_{\mathbb{H}^n}\frac{\Phi_{\alpha_{\vartheta,Q},j_0}\big(u(\delta_\theta(\xi))\big)}{r(\delta_\theta(\xi))^\vartheta}\,\mathrm{d}(\delta_\theta(\xi))\leq\theta ^{Q-\vartheta}\int_{\mathbb{H}^n}\frac{\Phi_{\alpha_{\vartheta,Q},j_0}(\sigma\psi)}{r(\xi)^\vartheta}\,\mathrm{d}\xi\\
 &\leq \theta ^{Q-\vartheta} \text{SCMTI}(\sigma^{Q^\prime}\alpha_Q,\vartheta)\leq \Big(\sigma^{-1}(1-\sigma^a)^{\frac{1}{b}}\Big)^{p^\ast\big(1-\frac{\vartheta}{Q}\big)}\frac{C(\vartheta,p,Q)}{(1-\sigma^Q)^{\frac{p^\ast}{Q}\big(1-\frac{\vartheta}{Q}\big)}}\\
 &\leq \Bigg(\frac{(1-\sigma^a)^{\frac{1}{b}}}{(1-\sigma^Q)^{\frac{1}{Q}}}\Bigg)^{p^\ast\big(1-\frac{\vartheta}{Q}\big)}C(\vartheta,p,Q)\leq C(\vartheta,a,b,p,Q),  
\end{align*}
where $C(\vartheta,a,b,p,Q)$ is a suitable constant depending on $\vartheta,a,b,p$ and $Q$. {Next, we suppose}  that $\sigma\in \big(0,2^{-\frac{Q}{p^\ast}}\big]$ and $\psi\colon\mathbb{H}^n\to\mathbb{R}$ is given by
$$\psi(\xi):=2^{\frac{Q}{p^\ast}}u(\delta_2(\xi)),~\forall~\xi\in \mathbb{H}^n.$$
Then\textcolor{blue}{,} one can easily get
$$ D_H \psi(\xi)=2^{\frac{Q}{p^\ast}+1}D_H u(\delta_2(\xi)),~~\|D_H\psi\|_Q=2^{\frac{Q}{p^\ast}}\|D_H u\|_Q\leq 1\quad\text{and}\quad \|D_H\psi\|_p=\|D_H u\|_p\leq 1.$$
It follows from Theorem \ref{thm5.1} that
\begin{align*}
  \int_{\mathbb{H}^n}\frac{\Phi_{\alpha_{\vartheta,Q},j_0}(u)}{r(\xi)^\vartheta}\,\mathrm{d}\xi & =\int_{\mathbb{H}^n}\frac{\Phi_{\alpha_{\vartheta,Q},j_0}\big(u(\delta_2(\xi))\big)}{r(\delta_2(\xi))^\vartheta}\,\mathrm{d}(\delta_2(\xi))\leq 2 ^{Q-\vartheta}\int_{\mathbb{H}^n}\frac{\Phi_{\alpha_{\vartheta,Q},j_0}\big(2^{-\frac{Q}{p^\ast}}\psi\big)}{r(\xi)^\vartheta}\,\mathrm{d}\xi\\
 &\leq 2 ^{Q-\vartheta} \text{SCMTI}\Big(2^{-\frac{Q^2}{(Q-1)p^\ast}}\alpha_Q,\vartheta\Big)\\
 &\leq 2 ^{Q-\vartheta}\Big(1-2^{-\frac{Q^2}{p^\ast}}\Big)^{-\frac{p^\ast}{Q}(1-\frac{\vartheta}{Q})}C(\vartheta,p,Q)\leq C(\vartheta,p,Q),  
\end{align*}
where $C(\vartheta,p,Q)$ is a constant depending on $\vartheta,p$ and $Q$. This shows that $\text{CMTI}_{a,b}(\vartheta)<+\infty$ {for all} $b\in(0,Q]$. 

To check the sharpness of $\alpha_Q$, we reconsider the sequence $\{v_k\}_k\subset W$, which is defined in \eqref{eq3.99}. It is obvious that
$$\|D_H v_k\|_Q=1\quad\text{and}\quad \|D_H v_k\|_p^p=\frac{c^{-\frac{p}{Q}}_Q(\ln k)^{-\frac{p}{Q}}c_p s^{Q-p}}{Q-p}\bigg(1-\frac{1}{k^{Q-p}}\bigg).$$
Let $\{\tau_k\}_k\subset (0,1)$ be a sequence such that $$\tau_k^a+\tau_k^b\|D_H v_k\|^b_p=1\quad\text{and}\quad \tau_k\to 1\quad\text{as}\quad k\to\infty.$$  Define a new sequence $\{u_k\}_k\subset W$ by
$$ u_k(\xi):=\tau_k v_k(\xi),~~\forall~\xi\in\mathbb{H}^n. $$
Thus, one can easily deduce that $\|D_H u_k\|^a_Q+\|D_H u_k\|^b_p=1$. In addition, if we take $\zeta>\alpha_Q$, then 
\begin{align*}
   \int_{\mathbb{H}^n}\frac{\Phi_{\zeta(1-\frac{\vartheta}{Q}),j_0}(u_k)}{r(\xi)^\vartheta}\,\mathrm{d}\xi&\geq \int_{\{\xi\in\mathbb{H}^n\colon~r(\xi)<sk^{-1}\}}\frac{\Phi_{\zeta(1-\frac{\vartheta}{Q}),j_0}(u_k)}{r(\xi)^\vartheta}\,\mathrm{d}\xi\\
   &= \int_\Sigma\int_{0}^{sk^{-1}}\Phi_{\zeta(1-\frac{\vartheta}{Q}),j_0}\Big(\tau_k c^{-\frac{1}{Q}}_Q(\ln k)^{\frac{Q-1}{Q}}\Big)\eta^{Q-\vartheta-1}\,\mathrm{d}\eta\,\mathrm{d}\mu(\xi^\ast)\\
   &\gtrsim \bigg(\displaystyle\int_\Sigma\,\mathrm{d}\mu(\xi^\ast)\bigg)\cfrac{(sk^{-1})^{Q-\vartheta}}{Q-\vartheta}~\operatorname{\exp}\Bigg[\zeta\Big(1-\frac{\vartheta}{Q}\Big)\bigg|\tau_k c^{-\frac{1}{Q}}_Q(\ln k)^{\frac{Q-1}{Q}}\bigg|^{\frac{Q}{Q-1}}\Bigg]\\
   &=\bigg(\displaystyle\int_\Sigma\,\mathrm{d}\mu(\xi^\ast)\bigg) \cfrac{s^{Q-\vartheta}}{Q-\vartheta}~\operatorname{\exp}\bigg((Q-\vartheta)\ln k\bigg(\frac{\zeta \tau_k^{\frac{Q}{Q-1}}}{\alpha_Q}-1\bigg)\bigg)\to+\infty\quad\text{as}\quad k\to\infty.
\end{align*}
It follows that $\text{CMTI}_{a,b}(\vartheta)=+\infty$ {for all} $\zeta>\alpha_Q$ and thus\textcolor{blue}{,} the proof of the sharpness of $\alpha_Q$ is completed. 

Due to Lemma \ref{lem4.2}, we have
$$ \sup_{\zeta\in(0,\alpha_Q)}\Bigg(\cfrac{1-\big(\frac{\zeta}{\alpha_Q}\big)^{(\frac{Q-1}{Q})a}}{\big(\frac{\zeta}{\alpha_Q}\big)^{(\frac{Q-1}{Q})b}}\Bigg)^{\frac{p^\ast}{b}\big(1-\frac{\vartheta}{Q}\big)}\text{SCMTI}(\zeta,\vartheta)\leq \text{CMTI}_{a,b}(\vartheta).$$
Let $\{\psi_k\}_k\subset W\setminus\{0\}$ be a maximizing sequence for $\text{CMTI}_{a,b}(\vartheta)$. It follows that
$$ \|D_H \psi_k\|^a_Q+\|D_H \psi_k\|^b_p\leq1\quad\text{and}\quad \lim_{k\to\infty} \int_{\mathbb{H}^n}\frac{\Phi_{\alpha_{\vartheta,Q},j_0}(\psi_k)}{r(\xi)^\vartheta}\,\mathrm{d}\xi=\text{CMTI}_{a,b}(\vartheta).$$
Define another sequence $\{w_k\}_k\subset W$ by
$$ w_k(\xi):=\frac{\psi_k(\delta_{\theta_k} (\xi))}{\|D_H \psi_k\|_Q}\quad\text{with}\quad\theta_k:=\Bigg(\frac{1-\|D_H \psi_k\|^a_{Q}}{\|D_H \psi_k\|^b_Q}\Bigg)^{\frac{p^\ast}{Qb}} .$$
By direct calculations, we have 
$$ D_H w_k=\frac{\theta_k D_H\psi_k(\delta_{\theta_k} (\xi))}{\|D_H \psi_k\|_Q},\quad \|D_H w_k\|_Q=1$$
and $$ \|D_H w_k\|_p=\frac{\|D_H \psi_k\|_p}{\|D_H \psi_k\|_Q}\theta_k^{-\frac{Q}{p^\ast}}\leq \frac{(1-\|D_H \psi_k\|^a_{Q})^{\frac{1}{b}}}{\|D_H \psi_k\|_Q}\theta_k^{-\frac{Q}{p^\ast}}=1. $$
From the definition of $\text{SCMTI}(\zeta,\vartheta)$, one sees that 
\begin{align*}
  \int_{\mathbb{H}^n}\frac{\Phi_{\alpha_{\vartheta,Q},j_0}(\psi_k)}{r(\xi)^\vartheta}\,\mathrm{d}\xi&=\int_{\mathbb{H}^n}\frac{\Phi_{\alpha_{\vartheta,Q},j_0}(\psi_k(\delta_{\theta_k} (\xi)))}{r(\delta_{\theta_k} (\xi))^\vartheta}\,\mathrm{d}(\delta_{\theta_k} (\xi))\leq \theta_k^{Q-\vartheta}\int_{\mathbb{H}^n}\frac{\Phi_{\alpha_{\vartheta,Q},j_0}(\|D_H \psi_k\|_Qw_k)}{r(\xi)^\vartheta}\,\mathrm{d}\xi \\
   &\leq \theta_k^{Q-\vartheta}\text{SCMTI}\Big(\|D_H \psi_k\|_{Q}^{\frac{Q}{Q-1}}\alpha_Q,\vartheta\Big)\\
   &\leq \sup_{\zeta\in(0,\alpha_Q)}\Bigg(\cfrac{1-\big(\frac{\zeta}{\alpha_Q}\big)^{(\frac{Q-1}{Q})a}}{\big(\frac{\zeta}{\alpha_Q}\big)^{(\frac{Q-1}{Q})b}}\Bigg)^{\frac{p^\ast}{b}\big(1-\frac{\vartheta}{Q}\big)}\text{SCMTI}(\zeta,\vartheta).
\end{align*}
Letting $k\to\infty$ in the above inequality, we obtain
\begin{align*}
  \text{CMTI}_{a,b}(\vartheta)\leq\sup_{\zeta\in(0,\alpha_Q)}\Bigg(\cfrac{1-\big(\frac{\zeta}{\alpha_Q}\big)^{(\frac{Q-1}{Q})a}}{\big(\frac{\zeta}{\alpha_Q}\big)^{(\frac{Q-1}{Q})b}}\Bigg)^{\frac{p^\ast}{b}\big(1-\frac{\vartheta}{Q}\big)}\text{SCMTI}(\zeta,\vartheta).  
\end{align*}
{By} combining the estimations above yields \eqref{eq1.15}. To finish the proof, it suffices to show that $\text{CMTI}_{a,b}(\vartheta)<+\infty$ implies $b\leq Q$. Indeed, if not, suppose there exists some $b$ satisfying $b>Q$ such that $\text{CMTI}_{a,b}(\vartheta)<+\infty$. Consequently, one can observe from \eqref{eq1.15} that
\begin{align*}
    \limsup_{\zeta\to\alpha_Q^-}\Bigg(\cfrac{1-\big(\frac{\zeta}{\alpha_Q}\big)^{(\frac{Q-1}{Q})a}}{\big(\frac{\zeta}{\alpha_Q}\big)^{(\frac{Q-1}{Q})b}}\Bigg)^{\frac{p^\ast}{b}\big(1-\frac{\vartheta}{Q}\big)}\text{SCMTI}(\zeta,\vartheta)<+\infty.
\end{align*}
In addition, we also obtain from Theorem \ref{thm5.1} that
\begin{align*}
    \limsup_{\zeta\to\alpha_Q^-}\Bigg(\cfrac{1-\big(\frac{\zeta}{\alpha_Q}\big)^{(\frac{Q-1}{Q})a}}{\big(\frac{\zeta}{\alpha_Q}\big)^{(\frac{Q-1}{Q})b}}\Bigg)^{\frac{p^\ast}{Q}\big(1-\frac{\vartheta}{Q}\big)}\text{SCMTI}(\zeta,\vartheta)>0.
\end{align*}
It follows immediately from the above two inequalities that
$$ \liminf_{\zeta\to\alpha_Q^-}\Bigg(\cfrac{1-\big(\frac{\zeta}{\alpha_Q}\big)^{(\frac{Q-1}{Q})a}}{\big(\frac{\zeta}{\alpha_Q}\big)^{(\frac{Q-1}{Q})b}}\Bigg)^{\frac{p^\ast}{b}\big(1-\frac{\vartheta}{Q}\big)}\Bigg(\cfrac{\big(\frac{\zeta}{\alpha_Q}\big)^{(\frac{Q-1}{Q})b}}{1-\big(\frac{\zeta}{\alpha_Q}\big)^{(\frac{Q-1}{Q})a}}\Bigg)^{\frac{p^\ast}{Q}\big(1-\frac{\vartheta}{Q}\big)}<+\infty. $$
By using $b>Q$ in the above relation, we arrive at a contradiction. This ends the proof.
\end{proof}
\section{Compact embedding result: Proof of Theorem \ref{thm1.4}}\label{sec7}
This section is concerned with the study of a crucial compact embedding result, that is, Theorem \ref{thm1.4}, which plays a key role in proving Theorem \ref{thm1.5}. Moreover, we have
\begin{proof}[\bf{Proof of Theorem \ref{thm1.4}}.]
For any $u\in W$ and $\sigma,~\sigma^\prime>1$ satisfying $\frac{1}{\sigma}+\frac{1}{\sigma^\prime}=1$ with $\vartheta \sigma^\prime<Q$, we obtain from H\"{o}lder's inequality and Lemma \ref{lem2.1} that
\begin{align*}
   \int_{\mathbb{H}^n}\frac{|u|^\theta}{r(\xi)^\vartheta}\,\mathrm{d}\xi  
   \leq \bigg(\int_{B_1}|u|^{\theta\sigma}\,\mathrm{d}\xi\bigg)^{\frac{1}{\sigma}}\bigg(\int_{B_1}\frac{1}{r(\xi)^{\vartheta \sigma^\prime}}\,\mathrm{d}\xi\bigg)^{\frac{1}{\sigma^\prime}}+\int_{\mathbb{H}^n}|u|^\theta\,\mathrm{d}\xi \lesssim\|u\|^\theta<+\infty.
\end{align*}
Thus, the embedding $W\hookrightarrow L^\theta(\mathbb{H}^n,r(\xi)^{-\vartheta}\,\mathrm{d}\xi)$ is continuous for all $\theta\in [p^\ast,+\infty)$ and $\vartheta\in(0,Q)$. In addition, to prove the compactness, we choose a bounded sequence $\{u_k\}_k\subset W$. Hence, there exists a subsequence $\{u_{k_j}\}_j$ and $u$ in $W$ such that from Lemma \ref{lem2.1}, we obtain 
\begin{align*}
  u_{k_j}\rightharpoonup u~\text{in} ~ W,~ ~u_{k_j}\to u~~\text{in} ~~ L^\eta(B_R),~~u_{k_j}\to u~\text{a.e. in}~\mathbb{H}^n~~\text{as}~~j\to\infty\quad{\text{for all}}\quad R>0~~\text{and}~~\eta\geq 1{.}
\end{align*}
By the Egoroff's theorem, one sees that for any open ball $B_R$ and $\delta>0$, there exists $\mathbb{U}_\delta\subset B_R$ with $|\mathbb{U}_\delta|<\delta$ such that 
$$\sup_{B_R\setminus \mathbb{U}_\delta}|u_{k_j}- u|\to 0\quad\text{as} \quad j\to\infty.$$
{We can write}
\begin{align*}
    \lim_{j\to\infty}\int_{\mathbb{H}^n}\frac{|u_{k_j}- u|^\theta}{r(\xi)^\vartheta}\,\mathrm{d}\xi =\lim_{R\to\infty}\lim_{\delta\to 0}\lim_{j\to\infty}\bigg(\int_{ \mathbb{U}_\delta}+\int_{B_R\setminus \mathbb{U}_\delta}+\int_{\mathbb{H}^n\setminus B_R}\bigg)\frac{|u_{k_j}- u|^\theta}{r(\xi)^\vartheta}\,\mathrm{d}\xi.
\end{align*}
Due to H\"{o}lder's inequality and the embedding $W\hookrightarrow L^{\theta \sigma^\prime}(\mathbb{H}^n,r(\xi)^{-\vartheta\sigma^\prime}\,\mathrm{d}\xi)$, we have
\begin{align*}
 \int_{ \mathbb{U}_\delta}\frac{|u_{k_j}- u|^\theta}{r(\xi)^\vartheta}\,\mathrm{d}\xi  \leq |\mathbb{U}_\delta|^{\frac{1}{\sigma}}\bigg(\int_{ \mathbb{U}_\delta}\frac{|u_{k_j}- u|^{\theta \sigma^\prime}}{r(\xi)^{\vartheta\sigma^\prime }}\,\mathrm{d}\xi \bigg)^{\frac{1}{\sigma^\prime}} 
 \leq  |\mathbb{U}_\delta|^{\frac{1}{\sigma}}\sup_{j\in\mathbb{N}}\|u_{k_j}- u\|^\theta_{\theta \sigma^\prime,\vartheta\sigma^\prime}\lesssim |\mathbb{U}_\delta|^{\frac{1}{\sigma}}.
\end{align*}
This shows that
$$ \lim_{R\to\infty}\lim_{\delta\to 0}\lim_{j\to\infty}\int_{ \mathbb{U}_\delta}\frac{|u_{k_j}- u|^\theta}{r(\xi)^\vartheta}\,\mathrm{d}\xi=0. $$
Notice that
\begin{align*}
 \int_{B_R\setminus \mathbb{U}_\delta}\frac{|u_{k_j}- u|^\theta}{r(\xi)^\vartheta}\,\mathrm{d}\xi\leq \sup_{B_R\setminus \mathbb{U}_\delta}|u_{k_j}- u|^\theta\int_{B_R} \frac{1}{r(\xi)^{ \vartheta}}\,\mathrm{d}\xi\lesssim \sup_{B_R\setminus \mathbb{U}_\delta}|u_{k_j}- u|^\theta.
\end{align*}
It follows immediately that
$$ \lim_{R\to\infty}\lim_{\delta\to 0}\lim_{j\to\infty}\int_{B_R\setminus \mathbb{U}_\delta}\frac{|u_{k_j}- u|^\theta}{r(\xi)^\vartheta}\,\mathrm{d}\xi=0.$$
Moreover, we also have
\begin{align*}
  \int_{\mathbb{H}^n\setminus B_R}\frac{|u_{k_j}- u|^\theta}{r(\xi)^\vartheta}\,\mathrm{d}\xi\leq \frac{1}{R^\vartheta}  \int_{\mathbb{H}^n}|u_{k_j}- u|^\theta\,\mathrm{d}\xi\lesssim \frac{1}{R^\vartheta}\sup_{j\in\mathbb{N}}\|u_{k_j}- u\|^\theta\lesssim \frac{1}{R^\vartheta}.
\end{align*}
This yields that
$$ \lim_{R\to\infty}\lim_{\delta\to 0}\lim_{j\to\infty}\int_{\mathbb{H}^n\setminus B_R}\frac{|u_{k_j}- u|^\theta}{r(\xi)^\vartheta}\,\mathrm{d}\xi=0.$$
In conclusion, we obtain
$$ u_{k_j}\to u\quad\text{in}\quad L^\theta(\mathbb{H}^n,r(\xi)^{-\vartheta}\,\mathrm{d}\xi)\quad \text{as}\quad j\to\infty\quad{\text{for all}}\quad \theta\in[p^\ast,+\infty)\quad\text{and}\quad \vartheta\in(0,Q).$$ 
This concludes the proof.
\end{proof}
The following result can be directly deduced from the above theorem.
\begin{corollary}
For all $\theta\in[ p^\ast,+\infty)$ and $\vartheta\in(0,Q)$, the best constant $S_\theta$ in the embedding $W\hookrightarrow L^\theta(\mathbb{H}^n,r(\xi)^{-\vartheta}\,\mathrm{d}\xi)$ can be attained, where 
$$ S_\theta:=\inf_{u\in W\setminus\{0\}}\frac{\|u\|}{~\|u\|_{\theta,\vartheta}}. $$  
\end{corollary}
\section{Existence of ground state solutions: Proof of Theorem \ref{thm1.5}}\label{sec8}
In this section, by using typical variational approaches, we address the presence of ground state solutions to the {problem} \eqref{main problem}. To begin, we define the Palais-Smale compactness criterion and the famous mountain pass theorem of Rabinowitz \cite{Rabinowitz-1986} as follows.
\begin{definition}
 Let $X$ be a Banach space and $J:X\to\mathbb{R}$ be a functional of class $C^1(X,\mathbb{R})$. We say that  $J$ satisfies the Palais-Smale compactness condition at a suitable level $c\in \mathbb{R}$, if for any sequence $\{u_k\}_k\subset X$ such that
 \begin{equation}\label{eq6.1}
 J (u_k)\to c\quad\text{and}\quad\sup_{\|\varphi\|_{X}=1}|\langle{J^\prime(u_k),\varphi}\rangle|\to 0\quad\text{as}\quad k\to\infty   
 \end{equation}
has a strongly convergent subsequence in  $X$. Note that the sequence $\{u_k\}_k\subset X$ satisfying \eqref{eq6.1} is known as a Palais-Smale sequence at level $c\in \mathbb{R}$, which is denoted by the symbol $(PS)_c$.
\end{definition}
\begin{theorem}\label{thm6.2}
Let $X$ be a real Banach space and $J\in C^1(X,\mathbb{R})$. Suppose $J(0)=0$ and there hold:
\begin{itemize}
    \item[\textnormal{(a)}] There exist two constants $\alpha,\rho>0$ such that $J(u)\geq \alpha$ for all $u\in X$ with $\|u\|=\rho$.
     \item[\textnormal{(b)}] There exists $e\in X$ satisfying $\|e\|>\rho$ such that $J(e)<0$.  
\end{itemize}
Define $$\Gamma:=\big\{\gamma\in C^1([0,1],X)\colon \gamma(0)=0,~\gamma(1)=e\big\}. $$
Then $$c:=\inf_{\gamma\in \Gamma}\max_{s\in[0,1]} J(\gamma(s))\geq \alpha $$
and there exists a $(PS)_c$ sequence $\{u_k\}_k$ for $J$ in $X$. In addition, if $J$ satisfies the Palais-Smale condition, then $J$ possesses a critical value $c\geq\alpha$.   
\end{theorem}

The following lemma shows that the mountain pass geometrical structures are satisfied by the energy functional $I$.
\begin{lemma}\label{lem6.3}
  Suppose that the hypotheses $(f_1)-(f_3)$ are satisfied. Then, the following results hold:
    \begin{itemize}
        \item[\textnormal{(a)}]  Any nontrivial critical point of the energy functional $I$ is nonnegative.
        \item[\textnormal{(b)}] There exist $\delta,\rho>0$ such that $I(u)\geq \delta$ for all $u\in W$ and $\|u\|=\rho$.
       \item[\textnormal{(c)}] There exists $e\in W$ with $\|e\|>\rho$ such that $I(e)<0$.
    \end{itemize}
\end{lemma}
\begin{proof}
 Let $u\in W\setminus\{0\}$ be a critical point of the energy functional $I$. Denote $u=u^+-u^-$ and take $u^-\in W$ as a test function. Thus, we obtain from $\langle{I^\prime(u),u^{-}}\rangle=0$ and $(f_1)$ that
    $$\|D_Hu^-\|^p_p+\|D_H u^-\|^Q_Q=0. $$
This shows that
$$ \|D_Hu^-\|_p=\|D_H u^-\|_Q=0,~\text{that is},~~\|u^-\|=0. $$
It follows that $u^-=0$ a.e. in $\mathbb{H}^n$ and $u=u^+\geq 0$  a.e. in $\mathbb{H}^n$. This finishes the proof of $\textnormal{(a)}$. Next, we assume that $\sigma,\sigma^\prime>1$ with $\frac{1}{\sigma}+\frac{1}{\sigma^\prime}=1$ and $\tau>0$ such that $\alpha \sigma^\prime\tau^{\frac{Q}{Q-1}}\leq\alpha_{\vartheta,Q}$ for all $u\in W$ with $\|u\|\leq\tau$. By using \eqref{eq1.2}, H\"older's inequality, {Lemma \ref{lem2.2}}, Theorem \ref{thm1.1} and Theorem \ref{thm1.4}, one has
    \begin{align*}
      \int_{\mathbb{H}^n}\frac{F(\xi,u)}{r(\xi)^\vartheta}\,\mathrm{d}\xi &\leq \varepsilon \|u\|^q_{q,\vartheta}+C_\varepsilon \|u\|^\ell_{\ell\sigma,\vartheta}\bigg( \int_{\mathbb{H}^n}\frac{\Phi_{\alpha \sigma^\prime,j_0}(u)}{r(\xi)^\vartheta}\,\mathrm{d}\xi \bigg)^{\frac{1}{\sigma^\prime}}\\
      &\leq \varepsilon S^{-q}_q\|u\|^q+C_\varepsilon S^{-\ell}_{\ell\sigma}\|u\|^\ell\bigg( \int_{\mathbb{H}^n}\frac{\Phi_{\alpha_{\vartheta,Q},j_0}(u/\|u\|)}{r(\xi)^\vartheta}\,\mathrm{d}\xi \bigg)^{\frac{1}{\sigma^\prime}}\\
      &\leq D\big(\varepsilon \|u\|^q+\|u\|^\ell\big),~~\forall~\varepsilon>0,
    \end{align*}
where $D>0$ is a constant defined as follows:
\begin{align*}
    D:=\max\Bigg\{S^{-q}_q,C_\varepsilon S^{-\ell}_{\ell\sigma}\bigg( \int_{\mathbb{H}^n}\frac{\Phi_{\alpha_{\vartheta,Q},j_0}(u/\|u\|)}{r(\xi)^\vartheta}\,\mathrm{d}\xi \bigg)^{\frac{1}{\sigma^\prime}}\Bigg\}<+\infty.
\end{align*}
Set $\rho=\min\{1,\tau\}$. It follows that for all $u\in W$ with $\|u\|=\rho$, one {can obtain from \eqref{eq1.18} that}
\begin{align*}
    I(u)\geq \|u\|^Q\bigg(\frac{1}{Q}- D\big(\varepsilon \|u\|^{q-Q}+\|u\|^{\ell-Q}\big)\bigg).
\end{align*}
In addition, we choose $0<\varepsilon<\frac{1}{QD}$. Using the fact that $q, \ell>Q$, one can able to choose $\rho>0$ small enough such that $$\frac{1}{Q}- D\big(\varepsilon \rho^{q-Q}+\rho^{\ell-Q}\big)>0.$$
In conclusion, we have
\begin{align*}
  I(u)\geq \rho^{Q}\bigg(\frac{1}{Q}- D\big(\varepsilon \rho^{q-Q}+\rho^{\ell-Q}\big)\bigg) :=\delta>0\quad\text{for}\quad \|u\|=\rho. 
\end{align*}
This completes the proof of $\textnormal{(b)}$. Finally, let $u\in W\setminus\{0\}$ and $u\geq 0$ be such that $\|u\|=1$. Define a function  $\mathcal{H}\colon [1,+\infty)\to\mathbb{R}$ by $$\mathcal{H}(s):=s^{-\mu}F(\xi,su)-F(\xi,u) \quad{\text{for all}}\quad s\geq 1~~~\text{and}~~~u\in (0,+\infty).$$ 
Due to $(f_3)$, one sees that the map $ s\mapsto \mathcal{H}(s)$ is increasing in $[1,+\infty)$. It follows that $$F(\xi,su)\geq s^{\mu}F(\xi,u)\quad{\text{for all}}\quad s\geq 1~~~\text{and}~~~u\in (0,+\infty).$$ 
By choosing $s> 1$ large enough, we obtain
\begin{align*}
    I(su)\leq \frac{2}{p}s^{Q}-s^\mu \int_{\mathbb{H}^n}\frac{F(\xi,u)}{r(\xi)^\vartheta}\,\mathrm{d}\xi\to-\infty\quad\text{as}\quad s\to\infty.
\end{align*}
The proof now follows by choosing $ e = su$ with $s$ sufficiently large. This completes the proof of $\textnormal{(c)}$ and consequently of Lemma \ref{lem6.3}.
\end{proof}
Notice that by using Lemma \ref{lem6.3} and Theorem \ref{thm6.2}, one sees that the mountain pass minimax level 
 $$c:=\inf_{\gamma\in \Gamma}\max_{s\in[0,1]} I(\gamma(s))\geq \rho>0,$$
 where $$\Gamma:=\big\{\gamma\in C^1([0,1],W)\colon \gamma(0)=0,~ I(\gamma(1))<0\big\}. $$

The following result characterizes the range of the mountain pass minimax level $c$.
\begin{lemma}\label{lem6.4}
There exists $\beta_0>0$ large enough such that if the hypothesis $(f_5)$ is satisfied for all $\beta\geq \beta_0$, then there holds:
\begin{align}\label{eq6.2}
0<c<c_0:=\min\Bigg\{\frac{1}{2^{Q}}\bigg(\frac{\mu-Q}{Q\mu}\bigg)^{\frac{Q}{p}}\bigg(\frac{\alpha_{\vartheta,Q}}{\alpha_0}\bigg)^{Q-1},\frac{1}{2^p}\bigg(\frac{\mu-Q}{Q\mu}\bigg)\bigg(\frac{\alpha_{\vartheta,Q}}{\alpha_0}\bigg)^{(1-\frac{1}{Q})p}\Bigg\}.  
\end{align}    
\end{lemma}
\begin{proof}
Let $\varphi\in C^\infty_0(\mathbb{H}^n,[0,1])$ be a cut-off function satisfying $\varphi\equiv 1~~\text{in}~~B_{\frac{1}{2}}$, $\varphi\equiv 0~~\text{in}~~\mathbb{H}^n\setminus B_1$ and $|D_H \varphi|_H\leq 2~~\text{in}~~\mathbb{H}^n$.
Moreover, by using $(f_5)$, we obtain for all $s\in[0,1]$ that
\begin{align*}
I(s\varphi)&\leq \frac{s^p}{p}\|D_H \varphi\|^p_{L^p(B_1)}+\frac{s^{Q}}{Q}\|D_H \varphi\|^Q_{L^Q(B_1)}-\beta s^\varrho\|r(\xi)^{-\vartheta}\|_{L^1\big(B_{\frac{1}{2}}\big)}\\
 &\leq \bigg(\int_{\Sigma}\,\mathrm{d}\mu(\xi^\ast)\bigg) \Bigg[\bigg(\frac{2^{p}}{Qp}+\frac{2^{Q}}{Q^2}\bigg) s^p-\frac{\beta }{2^{Q-\vartheta}(Q-\vartheta)}s^\varrho\Bigg].
\end{align*}
Let $h\colon[0,1]\to W$ be such that $h(s)=s\varphi$. In addition, we choose $\beta_1>0$ in such a way that 
 \begin{align*}
        I(\varphi)\leq \bigg(\int_{\Sigma}\,\mathrm{d}\mu(\xi^\ast)\bigg) \Bigg[\bigg(\frac{2^{p}}{Qp}+\frac{2^{Q}}{Q^2}\bigg) -\frac{\beta_1 }{2^{Q-\vartheta}(Q-\vartheta)}\Bigg]<0{\quad\text{for all}\quad}\beta>\beta_1.
    \end{align*}
It follows that $h\in \Gamma$. Now, by simple computations, we deduce that 
\begin{align*}
    c&\leq \max_{s\in[0,1]} I(h(s))\leq \max_{s\geq 0} I(s\varphi)\leq \bigg(\int_{\Sigma}\,\mathrm{d}\mu(\xi^\ast)\bigg) \max_{s\geq0}\Bigg[\bigg(\frac{2^{p}}{Qp}+\frac{2^{Q}}{Q^2}\bigg)s^p -\frac{\beta }{2^{Q-\vartheta}(Q-\vartheta)}s^\varrho\Bigg]\\
    &=\bigg(\int_{\Sigma}\,\mathrm{d}\mu(\xi^\ast)\bigg)\bigg(1-\frac{p}{\varrho}\bigg)\bigg(\frac{2^{p}}{Qp}+\frac{2^{Q}}{Q^2}\bigg)^{\frac{\varrho}{\varrho-p}}\bigg(\frac{2^{Q-\vartheta}(Q-\vartheta)p}{\beta \varrho}\bigg)^{\frac{p}{\varrho-p}}\to 0\quad\text{as}\quad \beta\to\infty.
\end{align*}
It follows that there exists $\beta_0>\beta_1$ large enough such that $c<c_0$ for all $\beta\geq\beta_0$, where $c_0$ is given in \eqref{eq6.2}. This finishes the proof.
\end{proof}
\begin{lemma}\label{lem5.5}
Every $(PS)_c$ sequence $\{u_k\}_k\subset W$ for $I$ is bounded in $W$, where $c\in(0,c_0)$ and $c_0$ is given in Lemma \ref{lem6.4}. In addition, it holds:
    \begin{align}\label{eq6.3}
        \limsup_{k\to\infty}\|u_k\|^{\frac{Q}{Q-1}}<\frac{\alpha_{\vartheta,Q}}{\alpha_0}.
    \end{align}
\end{lemma}
\begin{proof}
Let $c\in(0,c_0)$ and $\{u_k\}_k\subset W$ be a $(PS)_c$ sequence for $I$. Now, by using $(f_3)$, one has
\begin{align}\label{eq6.4}
        c+o_k(1)+o_k(1)\|u_k\|\geq \bigg(\frac{1}{Q}-\frac{1}{\mu}\bigg)\Big(\|D_H u_k\|_p^p+\|D_H u_k\|_Q^Q\Big)\quad\text{as}\quad k\to\infty.
\end{align}
It is obvious that $\{u_k\}_k$ is bounded in $W$. Indeed, if not, assume that $\{u_k\}_k$ is unbounded in $W$. Hence, we have three possibilities as follows.

\text{Case-1:}
Assume that $\|D_H u_k\|_p\to \infty$ and $\|D_H u_k\|_Q\to\infty$ as $k\to\infty$. By using $1<p<Q$, we have  
\begin{align*}
\|D_H u_k\|_Q^Q\geq\|D_H u_k\|_Q^p>1\quad\text{for large}\quad k.   
\end{align*}
It follows from \eqref{eq1.99} and \eqref{eq6.4} that
\begin{align*}
  c+o_k(1)+o_k(1)\|u_k\|\geq 2^{1-p}\bigg(\frac{1}{Q}-\frac{1}{\mu}\bigg)\|u_k\|^p \quad\text{as}\quad k\to\infty. 
\end{align*}
Dividing $\|u_k\|^p$ on the both the sides and letting $k\to\infty$, one has $0\geq 2^{1-p}\big(\frac{1}{Q}-\frac{1}{\mu}\big)>0$, which is a contradiction.

\text{Case-2:}
Assume that $\|D_H u_k\|_p\to \infty$ as $k\to\infty$ and $\|D_H u_k\|_Q$ is bounded. Thus, we obtain from \eqref{eq1.99} and \eqref{eq6.4} that
\begin{align*}
 c+o_k(1)+o_k(1)\|u_k\|_W\geq \bigg(\frac{1}{Q}-\frac{1}{\mu}\bigg)\|D_H u_k\|_p^p\quad\text{as}\quad k\to\infty.  
\end{align*}
Dividing $\|D_H u_k\|_p^p$ on the both sides and letting $k\to\infty$, one has $0\geq \big(\frac{1}{Q}-\frac{1}{\mu}\big)>0$, which is again a contradiction.

\text{Case-3:}
Assume that $\|D_H u_k\|_Q\to\infty$ as $k\to\infty$ and $\|D_H u_k\|_p$ is bounded, then by arguing similarly as in Case-$2$, we arrive at contradiction.

Thus, one sees that $\{u_k\}_k\subset W$ is a bounded sequence in $W$. In addition, due to \eqref{eq6.4}, we have
\begin{align}\label{eq6.5}
\limsup_{k\to\infty} \|D_H u_k\|^Q_Q\leq \bigg(\frac{Q\mu}{\mu-Q}\bigg)c\quad\text{and}\quad \limsup_{k\to\infty} \|D_H u_k\|_p^Q\leq \bigg[\bigg(\frac{Q\mu}{\mu-Q}\bigg)c\bigg]^{\frac{Q}{p}}.   
\end{align}
By using \eqref{eq6.5}, one has 
\begin{align*}
   \limsup_{k\to\infty}\|u_k\|^{\frac{Q}{Q-1}}\leq \Bigg[\bigg(\frac{Q\mu}{\mu-Q}\bigg)c\Bigg]^{\frac{1}{Q-1}}+\Bigg[\bigg(\frac{Q\mu}{\mu-Q}\bigg)c\Bigg]^{\frac{Q}{(Q-1)p}},  
\end{align*}
where we have used  $$(a+b)^\eta\leq a^\eta+b^\eta,~~\forall~a,b\geq 0;~\eta\in(0, 1].$$
In virtue of $\frac{Q\mu}{\mu-Q}>1$ and $\frac{1}{Q-1}<\frac{Q}{(Q-1)p}$, we obtain from the previous inequality that
\begin{align}\label{eq6.6}
  \limsup_{k\to\infty}\|u_k\|^{\frac{Q}{Q-1}}\leq \Bigg[2\bigg(\frac{Q\mu}{\mu-Q}\bigg)^{\frac{Q}{p}}\Bigg]^{\frac{1}{Q-1}} \Big(c^{\frac{1}{Q-1}}+c^{\frac{Q}{(Q-1)p}}\Big).
\end{align}
Pick $c\in(0,1]$. It follows that $c^{\frac{1}{Q-1}}+c^{\frac{Q}{(Q-1)p}}\leq 2c^{\frac{1}{Q-1}} $ and thus\textcolor{blue}{,} we obtain from \eqref{eq6.6} that
\begin{align}\label{eq6.7}
  \limsup_{k\to\infty}\|u_k\|^{\frac{Q}{Q-1}}\leq \Bigg[2^Q\bigg(\frac{Q\mu}{\mu-Q}\bigg)^{\frac{Q}{p}} c\Bigg]^{\frac{1}{Q-1}} .
\end{align}
For $c\geq 1$, we deduce that $c^{\frac{1}{Q-1}}+c^{\frac{Q}{(Q-1)p}}\leq 2c^{\frac{Q}{(Q-1)p}} $. It follows from \eqref{eq6.6} that
\begin{align}\label{eq6.8}
  \limsup_{k\to\infty}\|u_k\|^{\frac{Q}{Q-1}}\leq \Bigg[2^{Q}\bigg(\frac{Q\mu}{\mu-Q}\bigg)^{\frac{Q}{p}} c^{\frac{Q}{p}}\Bigg]^{\frac{1}{Q-1}} .
\end{align}
Due to \eqref{eq6.7} and \eqref{eq6.8}, one can easily obtain \eqref{eq6.3}. This finishes the proof.
\end{proof}
\begin{lemma}\label{lem5.6}
 Suppose the hypotheses $(f_1)-(f_3)$, $(f_5)$ and $(f_6)$ are satisfied. Let $\{u_k\}_k\subset W$ be an arbitrary $(PS)_c$ sequence for $I$ with $0<c<c_0$, where $c_0$ is given in Lemma \ref{lem6.4}. Then, there exists a subsequence of $\{u_k\}_k$ still denoted by itself and $u\in W$ such that $D_H u_k\to D_H u$ a.e. in $\mathbb{H}^n$ as $k\to\infty$. Moreover, there holds\textcolor{blue}{:}
 \begin{align}\label{eq6.9}
     |D_H u_k|_H^{m-2}D_H u_k \rightharpoonup |D_H u|^{m-2}D_H u\quad\text{in}\quad L^{\frac{m}{m-1}}(\mathbb{H}^n,\mathbb{R}^{2n})\quad\text{as}\quad k\to\infty\quad\text{for}\quad m\in\{p,Q\}.
 \end{align}
\end{lemma}
\begin{proof}
Suppose that $c\in(0,c_0)$ and $\{u_k\}_k\subset W$ is a $(PS)_c$ sequence for $I$. Due to Lemma \ref{lem5.5}, we deduce that $\{u_k\}_k$ is a bounded sequence in $W$ and \eqref{eq6.3} holds. It follows that up to a subsequence of $\{u_k\}_k$ still denoted by itself and $u\in W$ such that $u_k\rightharpoonup u$ in $W$ as $k\to\infty$. In addition, we obtain from Lemma \ref{lem2.1} and Theorem \ref{thm1.4} that
\begin{align}\label{eq6.10}
    \begin{cases}
        u_k\to u\quad\text{in}\quad L^\tau (B_R)\quad{\text{for all}}\quad R>0\quad\text{and}\quad \tau\in [1,+\infty);\\
         u_k\to u\quad\text{in}\quad L^\theta(\mathbb{H}^n,r(\xi)^{-\vartheta}\mathrm{d}\xi)\quad{\text{for all}}\quad \theta\in [p^\ast,+\infty);\\
         u_k\to u\quad\text{a.e. in}\quad \mathbb{H}^n\quad\text{as}\quad k\to\infty.
    \end{cases}
\end{align}
Take $R>0$ and $\psi\in C^\infty_0(\mathbb{H}^n)$ such that $0\leq\psi\leq 1$ in $\mathbb{H}^n$, $\psi\equiv 1$ in $B_R$ and $\psi\equiv 0$ in $\mathbb{H}^n\setminus B_{2R}$. Further, since $u_k\rightharpoonup u$ in $W$ as $k\to\infty$ and $I^\prime(u_k)\to 0$ in $W^\ast$ as $k\to\infty$, we get
\begin{align}\label{eq6.11}
    \langle{I^\prime(u_k)-I^\prime(u),(u_k-u)\psi}\rangle=o_k(1)\quad\text{as}\quad k\to\infty.
\end{align}
For each $k\in\mathbb{N}$ and $m\in\{p,Q\}$, we obtain from convexity that
\begin{align*}
    \mathcal{L}^m_k(u_k,u):=\big(|D_H u_k|_H^{m-2} D_H u_k-|D_H u|_H^{m-2} D_H u, D_H u_k-D_H u\big)_H\geq 0\quad\text{a.e. in}\quad \mathbb{H}^n.
\end{align*}
Due to the well known Simon's inequality \cite{Simon-1978} and \eqref{eq6.11}, we infer that there exists $c_Q>0$ such that
\begin{equation}\label{eq6.12}
\begin{aligned}		
c^{-1}_{Q} \int_{B_R} &|D_H u_k-D_H u|_H^{Q}\,\mathrm{d}\xi \\ &\leq    \int_{B_R} \mathcal{L}^{Q}_k(u_k,u) \,\mathrm{d}\xi\leq \displaystyle\sum_{m\in\{p,Q\}} \bigg(\int_{B_R}\mathcal{L}^m_k(u_k,u) \,\mathrm{d}\xi\bigg) \\
&\leq \displaystyle\sum_{m\in\{p,Q\}} \bigg(\int_{\mathbb{H}^n}\mathcal{L}^m_k(u_k,u) \psi\,\mathrm{d}\xi\bigg)\\ 
&= -\sum_{m\in\{p,Q\}} \Bigg(\int_{\mathbb{H}^n}\big(|D_H u_k|_H^{m-2}D_H u_k-|D_H u|_H^{m-2}D_H u,D_H\psi)_H ( u_k-u)\,\mathrm{d}\xi \Bigg) \\ 
&\qquad +\int_{\mathbb{H}^n}\frac{(f(\xi,u_k)-f(\xi,u))(u_k-u)\psi}{r(\xi)^\vartheta}\,\mathrm{d}\xi{+o_k(1)} \quad\text{as}\quad k\to\infty.
\end{aligned}
\end{equation}
By applying H\"older's inequality and \eqref{eq6.10}, we obtain for $m\in\{p,Q\}$ that
\begin{align*}
    \bigg|\int_{\mathbb{H}^n}\big(|D_H u_k|_H^{m-2}D_H u_k-&|D_H u|_H^{m-2}D_H u,D_H\psi)_H ( u_k-u)\,\mathrm{d}\xi  \bigg|
		\\ & \leq \|D_H \psi\|_\infty\big(\|D_H u_k\|^{m-1}_m+\|D_H u\|^{m-1}_m\big) \bigg(\int_{B_{2R}}|u_k-u|^m\,\mathrm{d}\xi\bigg)^{\frac{1}{m}}\\
        &\lesssim \bigg(\int_{B_{2R}}|u_k-u|^m\,\mathrm{d}\xi\bigg)^{\frac{1}{m}}\to 0\quad\text{as }k\to\infty.
\end{align*}
It follows that for $m\in\{p,Q\}$, one has
\begin{align}\label{eq6.13}
  \lim_{k\to\infty} \int_{\mathbb{H}^n}\big(|D_H u_k|_H^{m-2}D_H u_k-|D_H u|_H^{m-2}D_H u,D_H\psi)_H ( u_k-u)\,\mathrm{d}\xi=0. 
\end{align}
Define
\begin{align*}
    J_1:=\varepsilon\int_{B_{2R}}\frac{(|u_k|^{q-1}+|u|^{q-1})|u
    _k-u|}{r(\xi)^\vartheta}\,\mathrm{d}\xi
\end{align*}
and \begin{align*}
    J_2:=C_\varepsilon\int_{B_{2R}}\frac{\big(|u_k|^{\ell-1} \Phi_{\alpha,j_0}(u_k)+|u|^{\ell-1} \Phi_{\alpha,j_0}(u)\big)|u
    _k-u|}{r(\xi)^\vartheta}\,\mathrm{d}\xi.
\end{align*}
Let $1<\sigma<\frac{Q}{\vartheta}$ and $\varrho\geq p^\ast$ be such that $\frac{1}{\zeta}+\frac{1}{\zeta^\prime}=1$ for $\zeta\in\{\sigma,\varrho\}$. By employing H\"older's inequality and \eqref{eq6.10}, one has
\begin{align*}
J_1&\leq \varepsilon\Bigg[\bigg(\int_{B_{2R}}\frac{|u_k|^{(q-1)\varrho^\prime}}{r(\xi)^\vartheta}\,\mathrm{d}\xi\bigg)^{\frac{1}{\varrho^\prime}}+\bigg(\int_{B_{2R}}\frac{|u|^{(q-1)\varrho^\prime}}{r(\xi)^\vartheta}\,\mathrm{d}\xi\bigg)^{\frac{1}{\varrho^\prime}}\Bigg]\bigg(\int_{B_{2R}}\frac{|u
_k-u|^\varrho}{r(\xi)^\vartheta}\,\mathrm{d}\xi\bigg)^{\frac{1}{\varrho}}\\
&\leq \varepsilon\Big(\|u_k\|^{q-1}_{L^{(q-1)\varrho^\prime \sigma^\prime}(B_{2R})}+\|u\|^{q-1}_{L^{(q-1)\varrho^\prime \sigma^\prime}(B_{2R})}\Big)\bigg(\int_{B_{2R}}\frac{1}{r(\xi)^{\vartheta \sigma}}\,\mathrm{d}\xi\bigg)^{\frac{1}{\varrho^\prime \sigma}}\bigg(\int_{B_{2R}}\frac{|u
_k-u|^\varrho}{r(\xi)^\vartheta}\,\mathrm{d}\xi\bigg)^{\frac{1}{\varrho}}\\
& \lesssim \bigg(\int_{\mathbb{H}^n}\frac{|u
_k-u|^\varrho}{r(\xi)^\vartheta}\,\mathrm{d}\xi\bigg)^{\frac{1}{\varrho}}\to 0\quad\text{as}\quad k\to\infty.
\end{align*}
In virtue of \eqref{eq6.3}, passing if necessary to a subsequence, still denoted by $\{u_k\}_k$, one can assume
\begin{align*}
    \sup_{k\in\mathbb{N}}\|u_k\|^{\frac{Q}{Q-1}}<\frac{\alpha_{\vartheta,Q}}{\alpha_0}.
\end{align*}
Suppose that $\eta,\eta^\prime>1$ and $\frac{1}{\eta}+\frac{1}{\eta^\prime}=1$. Now, we fix $\bar{m}\in \big(\|u_k\|^{\frac{Q}{Q-1}},\frac{\alpha_{\vartheta,Q}}{\alpha_0}\big)$, $\alpha>\alpha_0$ close to $\alpha_0$ and $\eta>1$ close to $1$ such that  $\alpha \eta \bar{m}<\alpha_{\vartheta,Q}$. Hence, we obtain from H\"older's inequality, Theorem \ref{thm1.1}, {Lemma \ref{lem2.2}},  Corollary \ref{cor3.1} and \eqref{eq6.10} that
\begin{align*}
J_2&\leq C_\varepsilon \Bigg[\bigg(\int_{\mathbb{H}^n}\frac{|u_k|^{(\ell-1)\eta^\prime}|u_k-u|^{\eta^\prime}}{r(\xi)^{\vartheta}}\,\mathrm{d}\xi\bigg)^{\frac{1}{\eta^\prime}}\bigg(\int_{\mathbb{H}^n}\frac{\Phi_{\alpha \eta,j_0}(u_k)}{r(\xi)^{\vartheta}}\,\mathrm{d}\xi\bigg)^{\frac{1}{\eta}}\\
&\quad\quad+\bigg(\int_{\mathbb{H}^n}\frac{|u|^{(\ell-1)\eta^\prime}|u_k-u|^{\eta^\prime}}{r(\xi)^{\vartheta}}\,\mathrm{d}\xi\bigg)^{\frac{1}{\eta^\prime}}\bigg(\int_{\mathbb{H}^n}\frac{\Phi_{\alpha \eta,j_0}(u)}{r(\xi)^{\vartheta}}\,\mathrm{d}\xi\bigg)^{\frac{1}{\eta}}\Bigg]\\
&\lesssim \Bigg[\bigg(\int_{\mathbb{H}^n}\frac{|u_k|^{\ell \eta^\prime}}{r(\xi)^{\vartheta}}\,\mathrm{d}\xi\bigg)^{\frac{\ell-1}{\ell \eta^\prime}} \bigg(\int_{\mathbb{H}^n}\frac{\Phi_{\alpha_{\vartheta,Q},j_0}(u_k/\|u_k\|)}{r(\xi)^{\vartheta}}\,\mathrm{d}\xi\bigg)^{\frac{1}{\eta}}\\
&\quad\quad+\bigg(\int_{\mathbb{H}^n}\frac{|u|^{\ell \eta^\prime}}{r(\xi)^{\vartheta}}\,\mathrm{d}\xi\bigg)^{\frac{\ell-1}{\ell \eta^\prime}} \bigg(\int_{\mathbb{H}^n}\frac{\Phi_{\alpha \eta,j_0}(u)}{r(\xi)^{\vartheta}}\,\mathrm{d}\xi\bigg)^{\frac{1}{\eta}}\Bigg]\bigg(\int_{\mathbb{H}^n}\frac{|u_k-u|^{\ell \eta^\prime}}{r(\xi)^{\vartheta}}\,\mathrm{d}\xi\bigg)^{\frac{1}{\ell \eta^\prime}}\\
&\lesssim \bigg(\int_{\mathbb{H}^n}\frac{|u_k-u|^{\ell \eta^\prime}}{r(\xi)^{\vartheta}}\,\mathrm{d}\xi\bigg)^{\frac{1}{\ell \eta^\prime}}\to 0\quad\text{as}\quad k\to\infty.
\end{align*}
Moreover, by using \eqref{eq1.1}, we have
\begin{align*}
 \bigg|\int_{\mathbb{H}^n}\frac{(f(\xi,u_k)-f(\xi,u))(u_k-u)\psi}{r(\xi)^\vartheta}\,\mathrm{d}\xi\bigg|&\leq \int_{B_{2R}}\cfrac{(|f(\xi,u_k)|+|f(\xi,u)|)|u_k-u|}{r(\xi)^\vartheta}\, \mathrm{d}\xi \\
 &\leq J_1+J_2\to 0\quad\text{as}\quad k\to\infty.   
\end{align*}
It follows that
\begin{equation}\label{eq6.14}
  \lim_{k\to\infty} \int_{\mathbb{H}^n}\frac{(f(\xi,u_k)-f(\xi,u))(u_k-u)\psi}{r(\xi)^\vartheta}\,\mathrm{d}\xi=0.  
\end{equation}
Passing $k\to\infty$ in \eqref{eq6.12} and using \eqref{eq6.13} and \eqref{eq6.14}, one sees that $D_H u_k\to D_H u$ in $L^{Q}(B_{R})$ as $k\to\infty$ for all $R>0$. Consequently, up to a subsequence still denoted by itself such that  $D_H u_k\to D_H u$ a.e. in $\mathbb{H}^n$ as $k\to\infty$. This directly implies that $|D_H u_k|_H^{m-2}D_H u_k \to |D_H u|^{m-2}D_H u$ a.e. in $\mathbb{H}^n$ as $k\to\infty$ for $m\in\{p,Q\}$. In conclusion, since $\{|D_H u_k|_H^{m-2}D_H u_k\}_k$ is bounded in $L^{\frac{m}{m-1}}(\mathbb{H}^n,\mathbb{R}^{2n})$ for $m\in\{p,Q\}$, we deduce that \eqref{eq6.9} holds.
\end{proof}
\begin{proposition}\label{prop5.7}
    Suppose that the hypotheses $(f_1)-(f_5)$ are satisfied. Then, the functional $I$ satisfies $(PS)_c$ condition for any $c\in(0,c_0)$, where $c_0$ is given in Lemma \ref{lem6.4}.
\end{proposition}
\begin{proof}
Let $\{u_k\}_k\subset W$ be a $(PS)_c$ sequence for $I$ with $c\in(0,c_0)$. It directly implies that
\begin{align}\label{eq5.16}
    \frac{1}{p}\|D_H u_k\|^p_p+\frac{1}{Q}\|D_H u_k\|^Q_Q-\int_{\mathbb{H}^n}\frac{F(\xi,u_k)}{r(\xi)^\vartheta}\,\mathrm{d}\xi\to c\quad\text{as}\quad k\to\infty
\end{align}
and 
\begin{align}\label{eq5.17}
  | \langle{I^\prime(u_k),\psi}\rangle |\leq \tau_k\|\psi\|,~\forall~\psi\in W,
\end{align}
where $\tau_k\to 0$ as $k\to\infty$. Take $\psi=u_k$ in \eqref{eq5.17}, then we obtain
\begin{align}\label{eq5.18}
 \int_{\mathbb{H}^n}\frac{f(\xi,u_k)u_k}{r(\xi)^\vartheta}\,\mathrm{d}\xi-\|D_H u_k\|^p_p-\|D_H u_k\|^Q_Q\leq \tau_k\|\ u_k\|.  
\end{align}
In view of Lemma \ref{lem5.5}, one sees that $\{u_k\}_k$ is a bounded sequence in $W$ and \eqref{eq6.3} holds. Moreover, we obtain from \eqref{eq5.16} and \eqref{eq5.18} that
\begin{align}\label{eq5.19}
\int_{\mathbb{H}^n}\frac{F(\xi,u_k)}{r(\xi)^\vartheta}\,\mathrm{d}\xi\leq C\quad\text{and}\quad \int_{\mathbb{H}^n}\frac{f(\xi,u_k)u_k}{r(\xi)^\vartheta}\,\mathrm{d}\xi \leq C,  
\end{align}
where $C>0$ is a suitable constant. Observe that \eqref{eq6.10} also holds. By using  \eqref{eq1.1}, Corollary \ref{cor3.1}, Lemma \ref{lem2.1} and H\"older's inequality, we can easily deduce that $\frac{f(\xi,u)}{r(\xi)^\vartheta}\in L^1_{\textit{loc}}(\mathbb{H}^n)$. In addition, by arguing similarly as in \cite[Lemma 3.1]{de Souza-do O-2011}, one has
\begin{align}\label{eq5.20}
 \frac{f(\xi,u_k)}{r(\xi)^\vartheta}\to \frac{f(\xi,u)}{r(\xi)^\vartheta}\quad\text{in}\quad L^1_{\textit{loc}}(\mathbb{H}^n)\quad\text{as}\quad k\to\infty. 
\end{align}
Next, we claim that
\begin{align}\label{eq5.21}
  \frac{F(\xi,u_k)}{r(\xi)^\vartheta}\to \frac{F(\xi,u)}{r(\xi)^\vartheta}\quad\text{in}\quad L^1(\mathbb{H}^n)\quad\text{as}\quad k\to\infty. 
\end{align}
To prove \eqref{eq5.21}, we choose $R>0$ and write 
\begin{align*}
\int_{\mathbb{H}^n}\frac{|F(\xi,u_k)-F(\xi,u)|}{r(\xi)^\vartheta}\,\mathrm{d}\xi=\bigg(\int_{B_R}+\int_{\mathbb{H}^n\setminus B_R}\bigg)\frac{|F(\xi,u_k)-F(\xi,u)|}{r(\xi)^\vartheta}\,\mathrm{d}\xi.    
\end{align*}
By using $(f_3)$ and $(f_4)$, one can find a constant $M_0>0$ such that
\begin{align}\label{eq5.22}
    \frac{F(\xi,u_k)}{r(\xi)^\vartheta}\leq M_0 \frac{|f(\xi,u_k)|}{r(\xi)^\vartheta},~\forall~ \xi\in \mathbb{H}^n. 
\end{align}
It follows from \eqref{eq6.10}, \eqref{eq5.20}, \eqref{eq5.22} and the generalized dominated convergence theorem of Lebesgue's that
\begin{align}\label{eq5.23}
    \lim_{R\to\infty}\lim_{k\to\infty} \int_{B_R}\frac{|F(\xi,u_k)-F(\xi,u)|}{r(\xi)^\vartheta}\,\mathrm{d}\xi=0.
\end{align}
Let us write
\begin{align*}
 \int_{\mathbb{H}^n\setminus B_R}\frac{|F(\xi,u_k)-F(\xi,u)|}{r(\xi)^\vartheta}\,\mathrm{d}\xi&=\int_{\mathbb{H}^n\setminus B_R~\cap\{\xi\in\mathbb{H}^n\colon~|u_k|> M\}}\frac{|F(\xi,u_k)-F(\xi,u)|}{r(\xi)^\vartheta}\,\mathrm{d}\xi\\
 &\quad\quad+\int_{\mathbb{H}^n\setminus B_R~\cap\{\xi\in\mathbb{H}^n\colon~|u_k|\leq M\}}\frac{|F(\xi,u_k)-F(\xi,u)|}{r(\xi)^\vartheta}\,\mathrm{d}\xi.   
\end{align*}
Due to \eqref{eq5.19} and \eqref{eq5.22}, one has
\begin{align*}
\int_{\mathbb{H}^n\setminus B_R~\cap\{\xi\in\mathbb{H}^n\colon~|u_k|> M\}}\frac{|F(\xi,u_k)-F(\xi,u)|}{r(\xi)^\vartheta}\,\mathrm{d}\xi &\leq \frac{M_0}{M}\int_{\mathbb{H}^n\setminus B_R~\cap\{\xi\in\mathbb{H}^n\colon~|u_k|> M\}}\frac{|f(\xi,u_k)u_k|}{r(\xi)^\vartheta}\,\mathrm{d}\xi\\
&\quad\quad+\int_{\mathbb{H}^n\setminus B_R}\frac{F(\xi,u)}{r(\xi)^\vartheta}\,\mathrm{d}\xi\\
 &\leq \frac{CM_0}{M}+\int_{\mathbb{H}^n\setminus B_R}\frac{F(\xi,u)}{r(\xi)^\vartheta}\,\mathrm{d}\xi.
\end{align*}
The above estimate together with \eqref{eq1.2}, Corollary \ref{cor3.1}, Theorem \ref{thm1.4} and the Lebesgue dominated convergence theorem implies that
\begin{align*}
 \lim_{M\to\infty} \lim_{R\to\infty}\lim_{k\to\infty}\int_{\mathbb{H}^n\setminus B_R~\cap\{\xi\in\mathbb{H}^n\colon~|u_k|> M\}}\frac{|F(\xi,u_k)-F(\xi,u)|}{r(\xi)^\vartheta}\,\mathrm{d}\xi=0. 
\end{align*}
By using $(f_1)$, one can find a constant $C_1>0$ such that $|F(\xi,s)|\leq C_1|s|^q$ for all $|s|\leq M$. It follows from Theorem \ref{thm1.4} that
\begin{align*}
   \int_{\mathbb{H}^n\setminus B_R~\cap\{\xi\in\mathbb{H}^n\colon~|u_k|\leq M\}}\frac{|F(\xi,u_k)-F(\xi,u)|}{r(\xi)^\vartheta}\,\mathrm{d}\xi&\leq \frac{C_1}{R^{\frac{\vartheta}{2}}} \int_{\mathbb{H}^n\setminus B_R~\cap\{\xi\in\mathbb{H}^n\colon~|u_k|\leq M\}}\frac{|u_k|^q}{r(\xi)^{\frac{\vartheta}{2}}}\,\mathrm{d}\xi\\
   &\qquad+\int_{\mathbb{H}^n\setminus B_R}\frac{F(\xi,u)}{r(\xi)^\vartheta}\,\mathrm{d}\xi\\
    & \leq \frac{C_2}{R^{\frac{\vartheta}{2}}}\sup_{k\in\mathbb{N}}\|u_k\|^q+\int_{\mathbb{H}^n\setminus B_R}\frac{F(\xi,u)}{r(\xi)^\vartheta}\,\mathrm{d}\xi,
\end{align*}
where $C_2>0$ is a constant independent on $k$. Moreover, we obtain from \eqref{eq1.2}, Corollary \ref{cor3.1}, Theorem \ref{thm1.4} and the Lebesgue dominated convergence theorem that
\begin{align*}
 \lim_{M\to\infty} \lim_{R\to\infty}\lim_{k\to\infty}  \int_{\mathbb{H}^n\setminus B_R~\cap\{\xi\in\mathbb{H}^n\colon~|u_k|\leq M\}}\frac{|F(\xi,u_k)-F(\xi,u)|}{r(\xi)^\vartheta}\,\mathrm{d}\xi=0. 
\end{align*}
This yields
\begin{align}\label{eq5.24}
    \lim_{R\to\infty}\lim_{k\to\infty} \int_{\mathbb{H}^n\setminus B_R}\frac{|F(\xi,u_k)-F(\xi,u)|}{r(\xi)^\vartheta}\,\mathrm{d}\xi=0.
\end{align}
Due to \eqref{eq5.23} and \eqref{eq5.24}, one sees that \eqref{eq5.21} holds. Now, by using \eqref{eq5.17}, \eqref{eq5.20}, Lemma \ref{lem5.6} and the Lebesgue dominated convergence theorem, we get
\begin{align*}
   \int_{\mathbb{H}^n} |D_H u|_H^{p-2}(D_H u,D_H \psi)_H\,\mathrm{d}\xi+\int_{\mathbb{H}^n} |D_H u|_H^{Q-2}(D_H u,D_H \psi)_H\,\mathrm{d}\xi=\int_{\mathbb{H}^n}\frac{f(\xi,u)}{r(\xi)^\vartheta}\psi\,\mathrm{d}\xi,~\forall~ \psi\in C^\infty_0(\mathbb{H}^n).
\end{align*}
By using the density argument of $C^\infty_0(\mathbb{H}^n)$ in $W$, we deduce that $u$ is a weak solution of {the problem} \eqref{main problem}. To finish the proof, it suffices to show the compactness of the $(PS)_c$ sequence $\{u_k\}_k$ in $W$. Next, by using $c>0$, we have two situations as follows.

\text{Case-1:} $(c\neq 0,~u=0)$ In this situation, one can easily obtain from \eqref{eq5.21} that
\begin{align*}
    \lim_{k\to\infty} \int_{\mathbb{H}^n}\frac{F(\xi,u_k)}{r(\xi)^\vartheta}\,\mathrm{d}\xi= 0.
\end{align*}
Thus, we deduce by using $I(u_k)\to c$ as $k\to\infty$ that
\begin{align*}
    c+o_k(1)&\notag=\frac{1}{p}\|D_H u_k\|^p_p+\frac{1}{Q}\|D_H u_k\|^Q_Q\geq\frac{1}{Q}\big(\|D_H u_k\|^p_p+\|D_H u_k\|^Q_Q\big) \\
    &\geq \bigg(\frac{1}{Q}-\frac{1}{\mu}\bigg)\Big(\|D_H u_k\|^p_p+\|D_H u_k\|^Q_Q\Big)\quad\text{as}\quad k\to\infty.
\end{align*}
Since $c\in(0,c_0)$, we obtain from the above estimate and Lemma \ref{lem5.5} that \eqref{eq6.3} holds. In addition, up to a subsequence, still denoted by itself, one can assume 
\begin{align*}
    \sup_{k\in\mathbb{N}}\|u_k\|^{\frac{Q}{Q-1}}<\frac{\alpha_{\vartheta,Q}}{\alpha_0}.
\end{align*}
Suppose that $\bar{m}\in\big(\|u_k\|^{\frac{Q}{Q-1}},\frac{\alpha_{\vartheta,Q}}{\alpha_0}\big)$ and $\sigma$, $\sigma^\prime>1$ satisfying $\frac{1}{\sigma}+\frac{1}{\sigma^\prime}=1$. Take $\alpha>\alpha_0$ close to $\alpha_0$, $\sigma>1$ close to $1$ such that $\alpha \sigma \bar{m}< \alpha_{\vartheta,Q}$. Now, by using \eqref{eq1.1}, Theorem \ref{thm1.1}, Theorem \ref{thm1.4}, {Lemma \ref{lem2.2}} and H\"older's inequality, one has
\begin{align*}
    \bigg|\int_{\mathbb{H}^n}\frac{f(\xi,u_k)u_k}{r(\xi)^\vartheta}\,\mathrm{d}\xi\bigg|&\leq\varepsilon \int_{\mathbb{H}^n}\frac{|u_k|^q}{r(\xi)^\vartheta}\,\mathrm{d}\xi+C_\varepsilon \bigg(\int_{\mathbb{H}^n}\frac{|u_k|^{\ell\sigma^\prime}}{r(\xi)^\vartheta}\,\mathrm{d}\xi\bigg)^{\frac{1}{\sigma^\prime}}\bigg(\int_{\mathbb{H}^n}\frac{\Phi_{\alpha \sigma,j_0}(u_k)}{r(\xi)^\vartheta}\,\mathrm{d}\xi\bigg)^{\frac{1}{\sigma}}\\
    & \lesssim \int_{\mathbb{H}^n}\frac{|u_k|^q}{r(\xi)^\vartheta}\,\mathrm{d}\xi+ \bigg(\int_{\mathbb{H}^n}\frac{|u_k|^{\ell\sigma^\prime}}{r(\xi)^\vartheta}\,\mathrm{d}\xi\bigg)^{\frac{1}{\sigma^\prime}}\bigg(\int_{\mathbb{H}^n}\frac{\Phi_{\alpha_{\vartheta,Q},j_0}(u_k/\|u_k\|)}{r(\xi)^\vartheta}\,\mathrm{d}\xi\bigg)^{\frac{1}{\sigma}}\\
    & \lesssim \int_{\mathbb{H}^n}\frac{|u_k-u|^q}{r(\xi)^\vartheta}\,\mathrm{d}\xi+ \bigg(\int_{\mathbb{H}^n}\frac{|u_k-u|^{\ell\sigma^\prime}}{r(\xi)^\vartheta}\,\mathrm{d}\xi\bigg)^{\frac{1}{\sigma^\prime}}\to 0\quad\text{as}\quad k\to\infty.
\end{align*}
This implies that
\begin{align}\label{eq5.25}
    \lim_{k\to\infty} \int_{\mathbb{H}^n}\frac{f(\xi,u_k)u_k}{r(\xi)^\vartheta}\,\mathrm{d}\xi=0.
\end{align}
It follows from $\langle{I^\prime(u_k),u_k}\rangle\to 0$ as $k\to\infty$ and \eqref{eq5.25} that $u_k\to 0$ in $W$ as $k\to\infty$. Further, by using $I\in C^1(W,\mathbb{R})$, we have $c=0$, which is a contradiction. Hence, this situation is not possible.

\text{Case-2:} $(c\neq 0,~u\neq 0)$
By using $u_k\rightharpoonup u$ in $W$ as $k\to\infty$, we deduce from the lower semicontinuity of the norm that, up to a subsequence, still denoted by itself, such that $\|u\|\leq\lim_{k\to\infty}\|u_k\|$. Moreover, if $\|u\|=\lim_{k\to\infty}\|u_k\|$, then by applying \cite[Corollary A.2]{Autuori-Pucci-2013}, one sees that $u_k\to u$ in $W$ as $k\to\infty$. In conclusion, we infer that the sequence $\{u_k\}_k$ satisfies the $(PS)_c$ compactness condition and thus\textcolor{blue}{,} $u$ is a nontrivial mountain pass solution of {the problem} \eqref{main problem}, we thank to $I\in C^1(W,\mathbb{R})$.

In addition, to complete the proof, one needs to show that $\|u\|<\lim_{k\to\infty}\|u_k\|$ is not possible. Indeed, we assume by contradiction that $\|u\|<\lim_{k\to\infty}\|u_k\|$ holds. Define the sequence $\{w_k\}_k$ and $w_0$ as follows:
\begin{align*}
    w_k:=\frac{u_k}{\|u_k\|}\quad\text{and}\quad w_0:=\frac{u}{\lim_{k\to\infty}\|u_k\|}.
\end{align*}
Due to \eqref{eq6.10}, one can easily check that $\|w_k\|=1$, $w_0\neq 0$ and $w_k\rightharpoonup w_0$ in $W$ as $k\to\infty$. Further, if $\|w_0\|=1$, then by using \cite[Corollary A.2]{Autuori-Pucci-2013}, one has $u_k\to u$ in $W$ as $k\to\infty$. It contradicts $\|u\|<\lim_{k\to\infty}\|u_k\|$ and the proof is completed here. 

Now, we study the case $\|w_0\|<1$. It follows from direct calculations and Lemma \ref{lem5.5} that up to a subsequence still denoted by itself, we have
\begin{align*}
\lim_{k\to\infty}\|u_k\|^{\frac{Q}{Q-1}}<\frac{\alpha_{\vartheta,Q}}{\alpha_0}\big(1-\|w_0\|^Q\big)^{-\frac{1}{Q-1}} .   
\end{align*}
Assume that $\sigma$, $\sigma^\prime>1$ and $\frac{1}{\sigma}+\frac{1}{\sigma^\prime}=1$. From the previous inequality, one sees that there esists $\bar{m}>0$, $\alpha>\alpha_0$ close to $\alpha_0$, $\sigma>1$ close to $1$ such that
\begin{align*}
    \alpha \sigma \|u_k\|^{\frac{Q}{Q-1}}<\bar{m}<\alpha_{\vartheta,Q}\big(1-\|w_0\|^Q\big)^{-\frac{1}{Q-1}} \quad\text{for sufficiently large}\quad k.
\end{align*}
By using \eqref{eq1.1}, H\"older's inequality, Theorem \ref{thm1.2}, Theorem \ref{thm1.4}, {Lemma \ref{lem2.2}}, we obtain for sufficiently large $k$ that
\begin{align*}
    \bigg|\int_{\mathbb{H}^n}\frac{f(\xi,u_k)(u_k-u)}{r(\xi)^\vartheta}\,\mathrm{d}\xi\bigg|&\leq \varepsilon \bigg(\int_{\mathbb{H}^n}\frac{|u_k|^{q}}{r(\xi)^\vartheta}\,\mathrm{d}\xi\bigg)^{\frac{q-1}{q}}\bigg(\int_{\mathbb{H}^n}\frac{|u_k-u|^{q}}{r(\xi)^\vartheta}\,\mathrm{d}\xi\bigg)^{\frac{1}{q}}\\
    &\quad+C_\varepsilon \bigg(\int_{\mathbb{H}^n}\frac{|u_k|^{(\ell-1)\sigma^\prime}|u_k-u|^{\sigma^\prime}}{r(\xi)^\vartheta}\,\mathrm{d}\xi\bigg)^{\frac{1}{\sigma^\prime}}\bigg(\int_{\mathbb{H}^n}\frac{\Phi_{\alpha \sigma,j_0}(u_k)}{r(\xi)^\vartheta}\,\mathrm{d}\xi\bigg)^{\frac{1}{\sigma}}\\
    & \lesssim \bigg(\int_{\mathbb{H}^n}\frac{|u_k|^{q}}{r(\xi)^\vartheta}\,\mathrm{d}\xi\bigg)^{\frac{q-1}{q}}\bigg(\int_{\mathbb{H}^n}\frac{|u_k-u|^{q}}{r(\xi)^\vartheta}\,\mathrm{d}\xi\bigg)^{\frac{1}{q}}+\bigg(\int_{\mathbb{H}^n}\frac{|u_k|^{\ell \sigma^\prime}}{r(\xi)^\vartheta}\,\mathrm{d}\xi\bigg)^{\frac{\ell-1}{\ell \sigma^\prime}}\\
    &\quad\times\bigg(\int_{\mathbb{H}^n}\frac{|u_k-u|^{\ell \sigma^\prime}}{r(\xi)^\vartheta}\,\mathrm{d}\xi\bigg)^{\frac{1}{\ell \sigma^\prime}} \bigg(\int_{\mathbb{H}^n}\frac{\Phi_{\bar{m},j_0}(w_k)}{r(\xi)^\vartheta}\,\mathrm{d}\xi\bigg)^{\frac{1}{\sigma}} \\
    &\lesssim \bigg(\int_{\mathbb{H}^n}\frac{|u_k-u|^{q}}{r(\xi)^\vartheta}\,\mathrm{d}\xi\bigg)^{\frac{1}{q}}+\bigg(\int_{\mathbb{H}^n}\frac{|u_k-u|^{\ell \sigma^\prime}}{r(\xi)^\vartheta}\,\mathrm{d}\xi\bigg)^{\frac{1}{\ell \sigma^\prime}}=o_k(1)\quad\text{as}\quad k\to\infty.
\end{align*}
This shows that
\begin{align*}
\lim_{k\to\infty}\int_{\mathbb{H}^n}\frac{f(\xi,u_k)(u_k-u)}{r(\xi)^\vartheta}\,\mathrm{d}\xi=0.   
\end{align*}
From the above estimate and $\langle{I^\prime(u_k),u_k-u}\rangle=o_k(1)$ as $k\to\infty$, one has 
\begin{align*}
   \int_{\mathbb{H}^n} |D_H u_k|_H^{p-2}(D_H u_k,D_H (u_k-u))_H\,\mathrm{d}\xi+\int_{\mathbb{H}^n}|D_H u_k|_H^{Q-2}(D_H u_k,D_H (u_k-u))_H\,\mathrm{d}\xi=o_k(1)~\text{as}~ k\to\infty.
\end{align*}
In addition, since $u_k\rightharpoonup u$ in $W$ as $k\to\infty$, we also have 
\begin{align*}
   \int_{\mathbb{H}^n} |D_H u|_H^{p-2}(D_H u,D_H (u_k-u))_H\,\mathrm{d}\xi+\int_{\mathbb{H}^n}|D_H u|_H^{Q-2}(D_H u,D_H (u_k-u))_H\,\mathrm{d}\xi=o_k(1)~\text{as}~ k\to\infty.
\end{align*}
It follows from convexity and the above convergences that
\begin{align}\label{eq5.26}
   \int_{\mathbb{H}^n} \big(|D_H u_k|_H^{m-2} D_H u_k-|D_H u|_H^{m-2} D_H u, D_H u_k-D_H u\big)_H\,\mathrm{d}\xi=o_k(1) \quad\text{as}\quad k\to\infty 
\end{align}
for $m\in\{p,Q\}$. Now, once more by the Simon's inequality \cite{Simon-1978} and \eqref{eq5.26}, it is not difficult to show that $D_H u_k\to D_H u$ in $L^m(\mathbb{H}^n)$ as $k\to\infty$ for $ m\in\{p,Q\}$. {Hence}, we have $u_k\to u$ in $W$ as $k\to\infty$, which contradicts the fact that $\|u\|<\lim_{k\to\infty}\|u_k\|$. This finishes the proof.
\end{proof}
Next, we are in a position to prove Theorem \ref{thm1.5}.
\begin{proof}[\bf{Proof of Theorem \ref{thm1.5}}.]
Let $\beta_0$ be given in Lemma \ref{lem6.4} and the hypothesis $(f_5)$ holds for all $\beta\geq\beta_0$. In view of Theorem \ref{thm6.2} and Proposition \ref{prop5.7}, we infer that $u$ is a critical point of $I$. Now, by using Lemma \ref{lem6.3}, one sees that $u\geq 0$ in $\mathbb{H}^n$. Additionally, we can observe that $I(u)=c>0$. It follows that $u$ is a positive solution to the {problem} \eqref{main problem}. Denote
$$\Lambda:=\inf_{u\in U} I(u)\quad\text{with}\quad U:=\{u\in W\setminus\{0\}\colon~I^\prime(u)=0\}. $$
Notice that to prove $u$ is a ground state solution to the {problem} \eqref{main problem}, it suffices to show that $c\leq\Lambda$. Define a map $g\colon(0,+\infty)\to\mathbb{R}$ by $g(s)=I(su)$ for all $s>0$ and $u\in U$. Since $I\in C^1(W,\mathbb{R})$, we deduce that $g$ is differentiable and
\begin{align}\label{eq6.26}
    g^\prime(s)=s^{p-1}\|D_H u\|^p_p+s^{Q-1}\|D_H u\|^Q_Q-\int_{\mathbb{H}^n}\frac{f(\xi,su)u}{r(\xi)^\vartheta}\,\mathrm{d}\xi.
\end{align}
Moreover, one can easily obtain from $\langle{I^\prime(u),u}\rangle=0$ that
\begin{align}\label{eq6.27}
    0=s^{Q-1}\|D_H u\|^p_p+s^{Q-1}\|D_H u\|^Q_Q-s^{Q-1}\int_{\mathbb{H}^n}\frac{f(\xi,u)u}{r(\xi)^\vartheta}\,\mathrm{d}\xi.
\end{align}
On subtracting \eqref{eq6.27} from \eqref{eq6.26}, we have
$$g^\prime(s)=\big(s^{p-1}-s^{Q-1}\big)\|D_H u\|^p_p+s^{Q-1}\int_{\mathbb{H}^n}\bigg(\frac{f(\xi,u)}{u^{Q-1}}-\frac{f(\xi,su)}{(su)^{Q-1}}\bigg)\frac{u^Q}{r(\xi)^\vartheta}\,\mathrm{d}\xi. $$
Consequently, we obtain from $(f_6)$ that
$$ g^\prime(s)>0,~\forall~s\in(0,1)\quad\text{and}\quad g^\prime(s)<0,~\forall~s\in(1,+\infty). $$
By using the elementary calculus, one sees that $1$ is the point of maximum of $g\colon(0,+\infty)\to\mathbb{R}$. It follows that
$$ I(u)=\max_{s\geq 0} I(su). $$
Let $h\colon[0,1]\to W$ be such that $h(s)=ss_0u$, where $s_0$ is assume in such a way that $I(s_0u)<0$. Hence, we have $h\in\Gamma$ and
$$c\leq \max_{s\in[0,1]} I(h(s))\leq \max_{s\geq 0} I(su)=I(u). $$
Due to the arbitrariness of $u\in U$, one sees that $c\leq \Lambda$. This completes the proof.
\end{proof}
\section*{Acknowledgements}
The authors would like to thank the anonymous referees for their careful reading and valuable comments and suggestions, which significantly enhanced the quality of the paper.
\section*{Funding}
Giovany M. Figueiredo has been partially supported by FAPESP 2023/10287-7, CNPq, and FAPDF. Deepak Kumar Mahanta sincerely thanks the Government of India for the DST INSPIRE Fellowship (Ref. No. DST/INSPIRE/03/2019/000265) and gratefully acknowledges their support. Abhishek Sarkar received funding from the DST-INSPIRE Grant DST/INSPIRE/04/2018/002208, funded by the Government of India.

\end{document}